\documentclass[12pt]{amsart}
\usepackage{amsmath,amssymb,amsfonts}
\usepackage{amscd}
\usepackage{color}
\usepackage{latexsym}
\usepackage{amsmath}
\usepackage{amssymb,amsmath,amsfonts,mathrsfs,amsthm}
\usepackage[all]{xy}
\usepackage{hyperref}
\hypersetup{
    colorlinks = true,
    linkcolor = {black},
   citecolor  = {black},
}

 \usepackage{enumitem}

\usepackage{tikz}
\usetikzlibrary{arrows}
\usepackage{tikz-cd}
\usetikzlibrary{cd}

\usepackage[margin=1in]{geometry}

\usepackage{changes}

\definechangesauthor[color=blue]{GA}
\definechangesauthor[color=red]{VK}

\newcounter{characterizations}

\newcommand{\R}{\mathbb{R}}
\newcommand{\Z}{\mathbb{Z}}

\newcommand{\s}{\sigma}

\newcommand{\obs}[1]{\operatorname{{\mathcal O}_3\!}{(#1)}}  
\newcommand{\obsfull}[2]{\operatorname{{\mathcal O}_3\!}{(#1, #2)}} 
\newcommand{\obsfr}[1]{\operatorname{{\mathcal O}^{\mathrm fr}_3\!}{(#1)}}
\newcommand{\obsn}[1]{\operatorname{{\mathcal O}_n\!}{(#1)}} 
\newcommand{\obsnfull}[2]{\operatorname{{\mathcal O}_n\!}{(#1, #2)}}
\newcommand{\obsnfr}[1]{\operatorname{{\mathcal O}^{\mathrm fr}_n\!}{(#1)}} 
 
\newcommand{\vk}[1]{\operatorname{{\mathcal O}_2\!}{(#1)}}  
\newcommand{\svk}[1]{\operatorname{{\mathcal W}_2\!}{(#1)}}  
\newcommand{\vkfr}[1]{\operatorname{{\mathcal O}^{\mathrm fr}_2\!}{(#1)}}

\newcommand{\gobs}[1]{\operatorname{{\mathcal W}_3\!}{(#1)}}  
\newcommand{\gobsn}[1]{\operatorname{{\mathcal W}_n\!}{(#1)}}  

\newcommand{\confs}[2]{\operatorname{C_s\!}{(#1, #2)}} 

\newcommand{\conf}[2]{\operatorname{C}{(#1, #2)}}
\newcommand{\subconf}[2]{\operatorname{C}_0{(#1, #2)}}
\newcommand{\deletedpower}[2]{{#1}^{#2}_{\Delta}}
\newcommand{\Fin}{\mathbb I} 
\newcommand{\op}{\operatorname{op}}
\newcommand{\Top}{\operatorname{Top}} 
\newcommand{\emb}{\operatorname{Emb}} 
\newcommand{\nat}{\operatorname{Nat}} 
\newcommand{\hnat}{\operatorname{hNat}} 
\newcommand{\holim}{\operatorname{holim}}
\newcommand{\map}{\operatorname{map}} 

\newcommand{\co}{\colon\thinspace}

\newtheorem{theorem}{Theorem}[section]
\newtheorem{lemma}[theorem]{Lemma}
\newtheorem{construction}[theorem]{Construction}
\newtheorem{conj}[theorem]{Conjecture}
\newtheorem{proposition}[theorem]{Proposition}
\newtheorem{corollary}[theorem]{Corollary}
\numberwithin{equation}{section}

\theoremstyle{definition}
\newtheorem{notation}[theorem]{Notation}
\newtheorem{definition}[theorem]{Definition}

\newtheorem{remark}[theorem]{Remark}
\newtheorem{convention}[theorem]{Convention}

\theoremstyle{remark}

\begin{document}

\title[Embedding obstructions in ${\mathbb \R^d}$]{Embedding obstructions in ${\mathbb \R^d}$ from the Goodwillie-Weiss calculus and Whitney disks}

\author{Gregory Arone}
\address{Gregory Arone{\hfil\break} Department of Mathematics, Stockholm University}
\email{}

\author{Vyacheslav Krushkal}
\address{Vyacheslav Krushkal{\hfil\break} Department of Mathematics, University of Virginia, Charlottesville, VA 22904-4137}
\email{krushkal\char 64 virginia.edu}

\begin{abstract}
Given a finite CW complex $K$, we use a version of the Goodwillie-Weiss tower to formulate an obstruction theory for embedding $K$ into a Euclidean space $\R^d$. For $2$-dimensional complexes in $\R^4$, a geometric analogue is also introduced, based on intersections of Whitney disks and more generally on the intersection theory of Whitney towers developed by Schneiderman and Teichner. We focus on the first obstruction beyond the classical embedding obstruction of van Kampen. In this case  we show the two approaches lead to essentially the same obstruction. We also give another geometric interpretation of our obstruction, as a triple collinearity condition. Furthermore, we relate our obstruction to the Arnold class in the cohomology of configuration spaces. The obstructions are shown to be realized in a family of examples. Conjectures are formulated, relating higher versions of these homotopy-theoretic, geometric and cohomological theories.    
\end{abstract}

\maketitle

\section{Introduction}

Let $K$ be a finite $CW$ complex of dimension $m$.
In this paper we introduce a new obstruction to the existence of a topological embedding $K\hookrightarrow \R^d$. The obstruction is defined for all $m$ and $d$, but our motivation comes primarily from questions about embedding $2$-dimensional complexes in $\R^4$.

\begin{remark}
It is worth noting that by a theorem of Stallings \cite{Stallings} (see also \cite{DR}), a $k$-connected $m$-complex is simple homotopy equivalent to a subcomplex of ${\mathbb R}^{2m-k}$. In particular, the embedding problem {\em up to homotopy} for $2$-complexes in ${\mathbb R}^4$ is trivial, cf. \cite{Curtis}. The subject of this paper is the much more subtle problem of embeddability of a given complex $K$ without changing it by a homotopy.
\end{remark}

We will give several definitions of the obstruction. One of the constructions is topological, and is inspired by the Embedding Calculus of Goodwillie and Weiss. The second construction is geometric, and is based on intersection theory of Whitney disks. We also give another geometric interpretation of the obstruction, as a triple collinearity condition. Finally we  give an algebraic description, in terms of the Arnold relation in the cohomology of configuration spaces. We will show that the different definitions agree, in an appropriate sense.

The most general construction is the topological one. It uses configuration spaces. Let 
$\conf{X}{n}= \{(x_1,\ldots, x_n)| \, x_i\neq x_j,\; {\rm for} \; i\ne j\}$
denote the $n$-point configuration space of a space $X$. The space $\conf{X}{n}$ has an action of the symmetric group $\Sigma_n$, permuting the coordinates. 

Let $\emb(K, \R^d)$ be the space of topological embeddings of $K$ into $\R^d$. Suppose $f\colon K\hookrightarrow \R^d$ is an embedding. Let $f^n\colon K^n\to (\R^d)^n$ be the $n$-th cartesian power of $f$. Since $f$ is injective, $f^n$ restricts to a map from $\conf{K}{n}$ to $\conf{\R^d}{n}$. This map is sometimes called the deleted $n$-th power of $f$. Note that the deleted power of $f$ is a $\Sigma_n$-equivariant map. Thus for each $n\ge 1$ we have defined an evaluation map, where the notation on the right indicates $\Sigma_n$-equivariant maps:
\begin{equation}\label{eq: naive}
\emb(K, \R^d)\to \map(\conf{K}{n}, \conf{\R^d}{n})^{\Sigma_n}.
\end{equation}
The map~\eqref{eq: naive} implies that for $K$ to be embeddable in $\R^d$, it is necessary that for every $n$ there exists a $\Sigma_n$-equivariant map from $\conf{K}{n}$ to $\conf{\R^d}{n}$. This observation gives rise to obstructions to existence of embeddings. The study of the obstruction arising from the case $n=2$ of~\eqref{eq: naive} goes back to van Kampen. We will review it in subsection~\ref{2 points} below.

The key idea of the paper is to use a refinement of the map~\eqref{eq: naive}. Rather than just consider the action of the symmetric groups on configuration spaces, we also take into account projection maps $ \conf{X}{i}\to \conf{X}{i-1}$ that omit one of the points. For each $n$ we define $T_n\emb(K, \R^d)$ to be, roughly speaking, the space of compatible $n$-tuples of functions 
\[(f_1, \ldots, f_n)\in \prod_{i=1}^n\map(\conf{K}{i}, \conf{\R^d}{i})^{\Sigma_i},\] 
where the maps $f_i$ respect the action the projection maps, at least up to coherent homotopies. More precisely, $T_n\emb(K, \R^d)$ is the space of derived natural transformations from the functor $\conf{K}{-}$ to $\conf{\R^d}{-}$ over the category of sets of size at most $n$ and injective functions between them. More details and a formal definition can be found in Section~\ref{the tower}.

The spaces $T_n\emb(K, \R^d)$ fit into a tower of spaces under $\emb(K, \R^d)$, as follows
\begin{equation}\label{eq: primitive}
\begin{tikzcd}[row sep=scriptsize,column sep=scriptsize]
\emb(K, \R^d)\arrow[d] \arrow{dr}\arrow{drrr} & & & &
\\\cdots \to T_n\emb(K, \R^d)\arrow{r} &  T_{n-1}\emb(K, \R^d) \arrow{r}& [-2em]  \cdots\arrow{r}   & [-2em] T_2\emb(K, \R^d)\arrow{r} & *.
\end{tikzcd}
\end{equation}
Since there is a map $\emb(K, \R^d) \to T_n\emb(K, \R^d)$, a necessary condition for $\emb(K, \R^d)$ to be non-empty is that $T_n\emb(K, \R^d)$ is non-empty for all $n$. This is the basis for our obstructions to embeddability of $K$ into $\R^d$. More specifically, our strategy is to look for an obstruction for a path component of $T_{n-1}\emb(K, \R^d)$ to be in the image of a path component of $T_{n}\emb(K, \R^d)$.
There is a cohomological obstruction $\obsn{K}$ to this lifting problem, formulated in Theorem \ref{thm: cohom obs}. In this way we obtain an infinite sequence of obstructions to the existence of a topological embedding of $K$ into $\R^d$. As we will review shortly, the case $n=2$ is classical. The case $n=3$ is the main subject of this paper. We hope that a more detailed study of higher obstructions, corresponding to $n>3$, will be pursued in future work.
\begin{remark}
The tower~\eqref{eq: primitive} is inspired by the embedding calculus of Goodwillie and Weiss~\cite{Weiss, GW}. Goodwillie and Weiss constructed a tower of approximations - the so called ``Taylor tower'' - to the space of {\it smooth} embeddings $\emb(M, N)$, where $M$ and $N$ are smooth manifolds. The tower~\eqref{eq: primitive} is a simplified version of their Taylor tower. The crucial difference between their construction and ours is that they impose compatibility not just with reordering and forgetting points, but also with doubling points. 
\end{remark}
Note that we make no claim that the induced map 
\[
\emb(K, \R^d)\longrightarrow \holim T_n\emb(K, \R^d)
\]
is an equivalence. This is in contrast with the Taylor tower of Goodwillie and Weiss, which is known to converge when the codimension is at least three.  Nevertheless, our version of the tower is useful for detecting non-embeddability of topological embeddings. In particular, it contains information about the problem of embedding $2$-complexes in $\R^4$. 

For $2$-complexes in $\R^4$ we also consider an alternative, geometric approach based on the failure of the Whitney trick in this dimension. Some instances of this approach are well-known, for example in the study of Milnor's invariants \cite{Milnor}. More generally, Schneiderman and Teichner \cite{ST} developed the intersection theory of {\em Whitney towers} in $4$-manifolds. We use these ideas to formulate embedding obstructions for $2$-complexes in $\R^4$.  

Considering the first new obstruction, we show that these a priori unrelated approaches in fact give the same result (Theorem~\ref{thm: obstructions coincide}). This provides a useful perspective on both of them: the homotopy-theoretic obstruction is manifestly well-defined but lacks an immediate geometric interpretation; the Whitney tower approach has a clear geometric meaning but establishing its well-definedness directly is a challenging problem.

In the following two subsections we discuss in concrete terms the obstructions arising at the bottom stages of the tower $T_n\emb(K, \R^d)$: the case $n=2$ corresponding to the classical van Kampen obstruction and the new obstruction arising from $n=3$.

\subsection{The van Kampen obstruction from ${\mathbf 2}$-point configuration spaces} \label{2 points}
We will now recall some of the classical results in the subject and relate them to our setting.
Suppose there exists a topological embedding $f\colon K\hookrightarrow \R^d$. We saw that it gives rise to a $\Sigma_2$-equivariant map - the deleted square of $f$:
\begin{equation}\label{eq: f2}
\deletedpower{f}{2}\colon \conf{K}{2} \to \conf{\R^d}{2}.
\end{equation}

The existence of a $\Sigma_2$-equivariant map $\conf{K}{2} \to \conf{\R^d}{2}$ is a necessary condition for the existence of a topological embedding $f\colon K\hookrightarrow \R^d$. To relate this discussion to our tower \eqref{eq: primitive}, let us note that it is easy to see that $T_1\emb(K, \R^d)\simeq \{*\}$, and
\begin{equation} \label{T2}
T_2\emb(K, \R^d)\simeq \map(\conf{K}{2}, \conf{\R^d}{2})^{\Sigma_2}.
\end{equation}
Thus the condition that there exists a $\Sigma_2$-equivariant map $\conf{K}{2} \to \conf{\R^d}{2}$ is equivalent to the condition that $T_2\emb(K, \R^d)$ is non-empty.

The van Kampen obstruction is a cohomological obstruction to the existence of such a $\Sigma_2$-equivariant map. It is an element, which we denote $\vk{K}$, of the equivariant cohomology group $H^{d}_{\Sigma_2}(\conf{K}{2}; \Z[(-1)^d])$, where $\Z[(-1)^d]$ denotes the integers with the action of $\Sigma_2$ by $(-1)^d$. There are many ways to construct the element~$\vk{K}$. 
The original formulation of van Kampen \cite{vK} predated a formal definition of cohomology, and it was based on a geometric approach. Moreover, van Kampen's formulation concerned the case $2{\rm dim}(K)=d$. We denote the geometric version of the obstruction by $\svk{K}$. It is defined by counting intersections of non-adjacent cells. We give a homotopy-theoretic definition of the obstruction $\vk{K}$ in Section~\ref{sec: Constructing the obstruction}, and review van Kampen's geometric definition of $\svk{K}$ in Section~\ref{sec:Whitney towers}. The following theorem summarizes the relevant facts about the van Kampen obstruction
\begin{theorem}
    The homotopy-theoretic obstruction~$\vk{K}$ agrees with the geometric obstruction~$\svk{K}$. When $2\dim(K)=d$, $\vk{K}$ is a complete obstruction for  $T_2\emb(K, \R^d)$ to be non-empty. Furthermore, when $2\dim(K)=d\ne 4$, $\svk{K}$ (and therefore also $\vk{K}$) is a complete obstruction to $K$ being embeddable in $\R^d$
\end{theorem}
This theorem is classical, though we hope that our formulation helps place it in a wider context. The fact that the homotopy-theoretic and the geometric formulations of the obstructions agree is explained, for example,  in~\cite[Section 3]{Melikhov}. That the van Kampen obstruction is complete when  $2\dim(K)=d> 4$ follows from the validity of the Whitney trick \cite{Shapiro, Wu}; a modern treatment may be found in \cite{FKT}. For $1$-complexes in $\R^2$ this follows from the Kuratowski graph planarity criterion \cite{Kuratowski} and
the naturality of van Kampen’s obstruction under embeddings. That $\vk{K}$ is a complete obstruction to $T_2\emb(K, \R^d)$ to be non-empty follows from Lemma~\ref{lem: characterizations}. See also the discussion following the proof of the lemma.

\begin{remark}
Building on work of Haefliger \cite{Haefliger}, Weber \cite{Weber} extended the embeddability result to the ``metastable range'' of dimensions. More precisely, it is shown in \cite{Weber} that
given an $m$-dimensional simplicial complex $K$ and a $\Sigma_2$-equivariant map $f_2\co \conf{K}{2}\longrightarrow \conf{\R^d}{2}$ with $2d\geq 3(m+1)$, there exists a PL  embedding $f\co K\longrightarrow \R^d$ such that the induced map $f^2_{\Delta}$ is $\Sigma_2$-equivariantly homotopic to $f_2$. \end{remark}

By contrast to all the cases when $2\dim(K)=d\ne 4$, it was shown in \cite{FKT} that when $K$ is a $2$-dimensional complex, the existence of a $\Sigma_2$-equivariant map $\conf{K}{2} \to \conf{\R^4}{2}$ is insufficient for embeddability of $K$ in $\R^4$, and thus the van Kampen obstruction is incomplete. The underlying geometric reason, the failure of the Whitney trick in $4$ dimensions, is well-known. However, as in many other aspects of $4$-manifold topology, it is a non-trivial problem to formulate an invariant that captures this geometric fact. In this paper, as we discuss below, we formulate such an invariant in the context of $2$-complexes in $\R^4$.

To summarize, the obstruction $\vk{K}$ to lifting from $T_1$ to $T_2$, which is the same as the obstruction 
for the space \eqref{T2} to be non-empty, is precisely the van Kampen obstruction. 
The lifting problem to the next stage of the tower, $T_3\emb(K, \R^d)$, discussed in the next subsection, yields an embedding obstruction for $m$-complexes in $\R^d$ beyond the metastable range: for $2d< 3(m+1)$.

\subsection{The obstruction from ${\mathbf 3}$-point configuration spaces} \label{3 points}

Suppose $K$ is a finite-dimensional complex for which the van Kampen obstruction vanishes. Then there exists a $\Sigma_2$-equivariant~map 
\[
f_2\colon \conf{K}{2} \longrightarrow \conf{\R^d}{2}
\]
Our goal is to give an effective necessary condition for the existence of an embedding $f\colon K\hookrightarrow \R^d$ such that the deleted square $\deletedpower{f}{2}\colon \conf{K}{2}\to \conf{\R^d}{2}$ is equivariantly homotopic to $f_2$. There is a cubical diagram of configuration spaces, where the projection $p^i$ omits the $i$-th coordinate:

\begin{equation} \label{cube}
\begin{tikzcd}[row sep=scriptsize,column sep=scriptsize]
& \conf{X}{2}  \arrow[rr, "p^1"] \arrow[dd, "p^2", pos=0.2] & & X  \arrow[dd] \\ 
\conf{X}{3} \arrow[rr, crossing over, "p^1", pos=0.7] \arrow[dd, "p^3"] \arrow[ur, "p^2"] & & \conf{X}{2} \arrow[ur, "p^1"]  \\
 & X \arrow[rr] & & \{*\} \\ 
 \conf{X}{2} \arrow[rr, "p^1"] \arrow[ur, "p^2"] & & X \arrow[ur] \arrow[uu, crossing over, leftarrow, "p^2"', pos=0.8]\\
\end{tikzcd}
\end{equation}
Now suppose we have a topological embedding $f\colon K\hookrightarrow \R^d$. Such an embedding induces a map of cubical diagrams (\ref{cube}) for $K$ and $\R^d$. In the diagram for $\R^d$ the space $\conf{\R^d}{1}=\R^d$ is contractible, and (up to homotopy) the map of cubical diagrams may be replaced by a smaller diagram (\ref{Delta3}) below.
Denote by $p_X$ the canonical $\Sigma_3$-equivariant map
\[
\begin{array}{cccccccc}
p_X\colon &\conf{X}{3} & \longrightarrow & \conf{X}{2}&\times &\conf{X}{2} & \times & \conf{X}{2} \\[5pt]
& (x_1, x_2, x_3) &\mapsto & (x_1, x_2) &, & (x_2, x_3) &, & (x_3, x_1) 
\end{array}
\]

Then $f$ induces a commutative diagram
\begin{equation}  \label{Delta3} \begin{tikzcd}
\conf{K}{3} \arrow{r}{\deletedpower{f}{3}} \arrow{d}{ p_K} & \conf{\R^d}{3}\arrow{d}{P_{\R^d}} \\
\conf{K}{2}\times \conf{K}{2}\times\conf{K}{2}\arrow{r}{(\deletedpower{f}{2})^3}& \conf{\R^d}{2}\times\conf{\R^d}{2}\times \conf{\R^d}{2} 
\end{tikzcd}
\end{equation}
Therefore, given a $\Sigma_2$-equivariant map $f_2\colon\conf{K}{2}\longrightarrow\conf{\R^d}{2}$, a necessary condition for it being induced by an embedding, is that the lifting problem in the following diagram has a solution
\begin{equation} \label{Delta3prime} \begin{tikzcd}
 & & \conf{\R^d}{3} \arrow{d}{p_{\R^d}}  \\
\conf{K}{3}\arrow[rru, dashed, shift left=2]\arrow[r, "p_K"] &\conf{K}{2}^{\times 3} \arrow{r}{(f_2)^3} & \conf{\R^d}{2}^{\times 3}
\end{tikzcd}
\end{equation}
There exists a cohomological obstruction to the existence of a $\Sigma_3$-equivariant dashed arrow that makes the diagram commute up to homotopy. We denote this obstruction by $\obsfull{K}{f_2}$, or simply by $\obs{K}$ when the choice of $f_2$ is immaterial. It turns out to be an element of an equivariant cohomology group of $\conf{K}{3}$. More specifically,  $$\obs{K}\in H^{2d-2}_{\Sigma_3}\left(\conf{K}{3}; \Z[(-1)^{d-1}]\right).$$ See Section \ref{sec: Constructing the obstruction} for a detailed discussion.
In terms of the tower \eqref{eq: primitive}, $\obsfull{K}{f_2}$ is the primary obstruction for the path component of $f_2$ in $T_2\emb(K, \R^d)$ to be in the image of the map $T_3\emb(K, \R^d)\to T_2\emb(K, \R^d)$.

We will give several topological, geometric and algebraic interpretations of $\obs{K}$; its properties are summarized below, along with references to the sections in the text where they are established.

\begin{itemize}
\item
For $2$-complexes in $\R^4$,  $\obs{K}$ counts intersections of $K$ with the Whitney disks that arise from the vanishing of the van Kampen obstruction (Subsection~\ref{intro 2-complexes} below, and Sections \ref{sec:Whitney towers}, \ref{sec:equality}). 
\item  $\obsfull{K}{f_2}$ also admits  another geometric interpretation as the fundamental class of the subspace of points $(k_1, k_2, k_3)\in\conf{K}{3}$ for which the vectors $f_2(k_1, k_2)$, $f_2(k_2, k_3)$ and $f_2(k_3, k_1)$ are co-directed (Section \ref{sec: construction}). 
\item 
Lemma \ref{lem: o3 Arnold} interprets $\obs{K}$ as the kernel of the Arnold relation in cohomology of configuation spaces. 
\item 
This algebraic interpretation is used to verify that $\obs{K}$ detects non-embeddability of a family of examples in Section \ref{sec:Examples} with vanishing van Kampen's obstruction. 
    \end{itemize}

\subsection{${\mathbf 2}$-complexes in ${\mathbf \R^4}$: obstructions from intersections of Whitney disks} \label{intro 2-complexes}
We outline in more detail the geometric approach to embedding obstructions in terms of intersections of Whitney disks for simplicial $2$-complexes in $\R^4$.
In this case, as we recall in Section \ref{sec:Whitney towers}, the vanishing of the van Kampen obstruction implies that a general position map $f\colon K\longrightarrow \R^4$ may be found such that for any two non-adjacent $2$-simplices $\s_i, \s_j$ of $K$, the algebraic intersection number $f(\s_i)\cdot f(\s_j)$ is zero. In higher dimensions in this setup the Whitney trick enables one to find an actual embedding, cf. \cite[Theorem 3]{FKT}. In dimension $4$ one may still consider Whitney disks $W_{ij}$ pairing up the intersections points $f(\s_i)\cap f(\s_j)$ but the Whitney disks themselves have self-intersections and intersect other $2$-cells, see \cite[Section 1.4]{FQ}  and also Figure \ref{Whitney} in Section \ref{sec:Whitney towers} below. 

Our geometric obstruction $\gobs{K}$ is an element of the equivariant cohomology group $$H^{6}_{\Sigma_3}(\confs{K}{3}; \Z[(-1)]);$$
this is the same cohomology group as the one discussed above except that now 
$\confs{K}{3}$ denotes the {\em simplicial} configuration space, that is $K^{ 3}$ minus the simplicial diagonal consisting of products $\s_1\times \s_2\times \s_3$ of simplices where at least two of them have a vertex in common.
The obstruction is defined on the cochain level by sending a $6$-cell $\s_1\times \s_2\times \s_3$ (where each $\s_i$ is a $2$-simplex of $K$) to the sum of intersection numbers $W_{ij}\cdot f(\s_k)$ over distinct indices $i,j,k$; see Section \ref{sec:higher geometric} for details.
Informally, the obstruction may be thought of as measuring the failure of the Whitney trick in $4$ dimensions. In the special case of disks in the $4$-ball with a prescribed boundary -- a link in the $3$-sphere $\partial D^4$ -- the analogous invariant equals the Milnor $\bar\mu$-invariant \cite{Milnor} of a $3$-component link, sometimes referred to as the triple linking number. For {\em knots}, a similar expression measuring self-intersections of a disk in $D^4$ equals the Arf invarint, see Remark \ref{rem: triple intersection} and references therein.

The obstruction $\gobs{K}$ depends on the map $f\colon K\longrightarrow \R^4$ and also on Whitney disks $W_{ij}$. In fact, we show in Lemma \ref{Whitney-config} that a choice of Whitney disks determines a $\Sigma_2$-equivariant map $\confs{K}{2}\longrightarrow \conf{\R^4}{2}$; in this sense the geometric setup is parallel to the homotopy-theoretic context discussed above. 

The following theorem summarizes some of our results about the obstructions $\obs{K}$ and $\gobs{K}$.
\begin{theorem}
    The obstructions $\obs{K}$ and $\gobs{K}$ are in fact equal (Theorem~\ref{thm: obstructions coincide}). These obstructions detect some non-embeddable complexes for which $\vk{K}$ vanishes (Section~\ref{sec:Examples}). The obstruction~$\obs{K}$ is a complete obstruction for lifting from $T_2\emb(K, \R^d)$ to $T_3\emb(K, \R^d)$ if $3m=2(d-1)$ (Proposition~\ref{prop: O3}).
\end{theorem}

The proof that $\obs{K}$ and $\gobs{K}$ coincide proceeds by localizing the problem, using subdivisions of the $2$-complex $K$ and splittings of Whitney disks, and identifying the Whitehead product in the homotopy fiber of the map 
$p_{\R^4}\colon \conf{\R^d}{3}\longrightarrow \conf{\R^d}{2}^{\times 3}$ in the notation of 
(\ref{Delta3prime}) using the Pontryagin-Thom construction; see Section \ref {sec:equality} for details.

\subsection{Lift of the obstructions from cohomology to framed cobordism}
In addition to constructing the cohomological obstructions, we define, in Sections~\ref{sec: Constructing the obstruction} and~\ref{sec: construction}, a lift of $\obs{K}$ which we denote $\obsfr{K}$. We hasten to add that an analogue of $\obsfr{K}$ in the context of smooth embeddings was studied by Munson~\cite{Munson}. Just as $\obs{K}$ is an element of the equivariant cohomology of $\conf{K}{3}$, $\obsfr{K}$ is an object of a suitable equivariant framed cobordism group (a.k.a stable cohomotopy group) of $\conf{K}{3}$. The Hurewicz homomorphism from stable homotopy to homology takes $\obsfr{K}$ to $\obs{K}$.

The class $\obsfr{K}$ is a complete obstruction to the lifting problem~\eqref{Delta3prime} whenever $\dim(K)\le d-2$. By contrast, $\obs{K}$ is a complete obstruction to the same lifting problem when $\dim(K)\le \frac{2}{3} d - \frac{2}{3}$. Thus $\obsfr{K}$ is a stronger invariant than $\obs{K}$. But when $d=4$ the difference is immaterial. All this is explained in Section~\ref{sec: Constructing the obstruction}. In Section~\ref{sec: construction} we give an explicit description of $\obsfr{K}$ in terms of a classifying map. As a consequence, we obtain in Section~\ref{geometric interpretation} {\it another} geometric interpretation of $\obsfull{K}{f_2}$ as the cohomology class represented by triples $(k_1, k_2, k_3)\in \conf{K}{3}$ for which $f_2(k_1, k_2)=f_2(k_2, k_3)=f_2(k_3, k_1)$.

In Section~\ref{the tower} we describe the general obstruction $\obsn{K}$ as an element in the equivariant cohomology of $\conf{K}{n}$ with coefficients in the cyclic Lie representation of $\Sigma_n$. We also give a conjectural description of $\obsnfr{K}$ in terms of equivariant stable cohomotopy of $\conf{K}{n}$ with coefficients in a space of trees that realizes the Lie representation.

\subsection{Outline of the paper}

Section \ref{sec: Constructing the obstruction} starts with the discussion of van Kampen's obstruction and its properties, and proceeds to define the new obstruction $\obs{K}$. We also describe a lift of $\obs{K}$ to an equivariant framed cobordism class $\obsfr{K}$, which is defined in terms of a {\em classifying map} $\conf{\R^d}{2}^3 \to \widehat \Omega^2\Omega^\infty\Sigma^\infty\widehat S^{2d}$. An explicit construction of $\obsfr{K}$ is deferred to Section \ref{sec: construction}.
Section \ref{sec:Whitney towers} starts by recalling the geometric definition of van Kampen's obstruction and basic operations on Whitney disks in dimension $4$.  Lemma \ref{Whitney-config} establishes a relation between Whitney disks and maps of configuration spaces, which illustrates a key connection between geometry and homotopy theory explored in this paper. Section \ref{sec:higher geometric} defines $\gobs{K}$ and analyzes its properties.
The construction of higher obstructions $\gobsn{K}$, in terms of intersection theory of Whitney towers of Schneiderman-Teichner, is outlined in Section \ref{sec: Whitney towers}.
The main result of Section \ref{sec:equality}, Theorem \ref{thm: obstructions coincide}, relates the obstructions $\obs{K}$ and $\gobs{K}$.
In Section~\ref{sec: construction} we construct the lift $\obsfr{K}$ of $\obs{K}$ and use it to give another topological interpretation of $\obs{K}$ in terms of the set of points satisfying a certain collinearity condition.
Section \ref{sec:Examples} recalls the examples of \cite{FKT} and shows that the obstruction $\obs{K}$ detects their non-embeddability in $\R^4$. In the process of doing this, $\obs{K}$ is related to the Arnold class in Lemma \ref{lem: o3 Arnold}.
Section \ref{the tower} gives the construction of the tower $T_n\emb(K, \R^n)$, formulates the higher obstructions $\obsn{K}$, and discusses their properties including a conjectural framed cobordism lift.
We conclude by stating a number of questions and conjectures motivated by our results in Section \ref{sec: Questions}.

{\bf Acknolwedgements}. We would like to thank Danica Kosanovi\'{c}, Rob Schneiderman and Peter Teichner for many discussions on the embedding calculus and Whitney towers.

We also thank the referees for reading the paper carefully and making many useful suggestions.

A substantial part of this project was completed while the authors visited EPFL, Lausanne, as part of the Bernoulli Brainstorm program in July 2019. We are grateful to the Bernoulli Center for warm hospitality and support.

GA was supported in part by Swedish Research Council, grant number 2016-05440.
VK was supported in part by the Miller Institute for Basic Research in Science at UC Berkeley, Simons Foundation fellowship 608604, and NSF Grant DMS-2105467.

\section{The first and second cohomological obstructions to embedding} \label{sec: Constructing the obstruction}
In section~\ref{subsection: van Kampen} we review the classical van Kampen obstruction $\vk{K}$ from a homotopy-theoretic perspective. Then in section \ref{subsection: secondary} we will introduce our main construction: a higher cohomological obstruction $\obs{K}$, defined when $\vk{K}=0$, and depending on a choice of a $\Sigma_2$-equivariant map $f_2\colon \conf{K}{2}\to\conf{\mathbb R^d}{2}$. Finally in section~\ref{subsection: cobordism} we discuss certain refinements $\vkfr{K}$ and $\obsfr{K}$ of $\vk{K}$ and $\obs{K}$ respectively into classes that reside in framed cobordism rather than cohomology. 
\subsection{The van Kampen obstruction}\label{subsection: van Kampen}

Let $K$ continue denoting an $m$-dimensional CW (or simplicial) complex. We are interested in the question whether there exists a topological (or PL) embedding of $K$ in $\mathbb R^d$. As we saw in the introduction, a necessary condition for the existence of an embedding is the existence of a $\Sigma_2$-equivariant map $f_2\colon \conf{K}{2}\to \conf{\mathbb R^d}{2}$. Or, equivalently, a $\Sigma_2$-equivariant map $\conf{K}{2}\to \widetilde S^{d-1}$, 
where $\widetilde S^{d-1}$ denotes the sphere with the antipodal action of $\Sigma_2$. Recall that there is a $\Sigma_2$-equivariant homotopy equivalence $\conf{\R^d}{2}\xrightarrow{\simeq} \widetilde S^{d-1}$ that sends $(x_1, x_2)$ to $\frac{x_2-x_1}{|x_2-x_1|}$. We will occasionally switch back and forth between these spaces.

There is a well-known homotopical/cohomological obstruction to the existence of a $\Sigma_2$-equivariant map $f_2\colon \conf{K}{2}\to \conf{\R^d}{2}$, which we will now review. Let $\widehat\R^d$ denote the $d$-dimensional Euclidean space on which $\Sigma_2$ acts by multiplication by $-1$.
\begin{notation}
Suppose $G$ is a group acting on a space $X$. We let $X_G$ and $X^G$ denote the orbit space and the fixed point space of $X$, respectively. If $X$ and $Y$ are two spaces with an action of $G$, then $G$ acts on the mapping space $\map(X, Y)$ by conjugation. In this case the fixed point space $\map(X, Y)^G$ is the space of equivariant maps from $X$ to $Y$. Also, we sometimes use the notation $X\times_G Y$ to denote the orbit of $X\times Y$ by the diagonal action. 
\end{notation}
Notice that $\Sigma_2$ acts on the trivial vector bundle
$$(K\times K\setminus K) \times \widehat \R^d \to K\times K \setminus K.$$
Passing to orbit spaces, one obtains the vector bundle \begin{equation}\label{eq: fibration}
(K\times K\setminus K)\times_{\Sigma_2}\widehat \R^d \to (K\times K\setminus K)_{\Sigma_2}
\end{equation}

Let $\widehat S^d$ be the one-point compactification of $\widehat \R^d$, considered as a space with an action of $\Sigma_2$. Equivalently, $\widehat S^d$ is the unreduced suspension of $\widetilde S^{d-1}$. Note that $\widehat S^d$ has two points fixed by $\Sigma_2$, corresponding to $0$ and $\infty$ in the compactificaton of $\widehat \R^d$. By convention, $\infty$ is the basepoint of $\widehat S^d$. The following elementary lemma gives several conditions for the existence of a $\Sigma_2$-map $K\times K\setminus K \to \widetilde S^{d-1}$.
\begin{lemma} \label{lem: characterizations}
Conditions~\eqref{map} and~\eqref{section} below are equivalent
\begin{enumerate}
    \item \label{map} There exists a $\Sigma_2$-equivariant map $K\times K\setminus K \to \widetilde S^{d-1}$.
    \item \label{section} The vector bundle~\eqref{eq: fibration} has a nowhere vanishing section.
\setcounter{characterizations}{\value{enumi}}
\end{enumerate}
Furthermore, conditions~\eqref{map} and~\eqref{section} above imply conditions~\eqref{homotopic} and~\eqref{cobordism} below. Under the assumption $d\ge \dim(K)+2$, the conditions (1)-(4) are equivalent.
\begin{enumerate}
\setcounter{enumi}{\value{characterizations}}
\item \label{homotopic} The constant map $K\times K\setminus K \to \widehat S^d$ that sends $K\times K\setminus K$ to $0$ is $\Sigma_2$-equivariantly null-homotopic. By this we mean that it is equivariantly homotopic to the constant map that sends $K\times K\setminus K$ to $\infty$. 
\item\label{cobordism} The constant map $K\times K\setminus K \to \Omega^\infty \Sigma^\infty\widehat S^d$ which is the map of part~\ref{homotopic} followed by the suspension map $\widehat S^d\to \Omega^\infty \Sigma^\infty \widehat S^d$ is $\Sigma_2$-equivariantly null-homotopic.
\end{enumerate}
\end{lemma}
\begin{proof}
The vector bundle~\eqref{eq: fibration} has a nowhere vanishing section if and only if the sphere bundle
\[
(K\times K\setminus K)\times_{\Sigma_2}\widetilde S^{d-1} \to (K\times K\setminus K)_{\Sigma_2}
\]
has a section. It is well-known that sections of this bundle are in bijective correspondence with $\Sigma_2$-equivariant maps $K\times K\setminus K \to \widetilde S^{d-1}$~\cite[Proposition 8.1.3]{tD}, which is why~\eqref{map} and~\eqref{section} are equivalent.

Suppose there is a $\Sigma_2$-equivariant map $K\times K\setminus K \to \widetilde S^{d-1}$. It induces  $\Sigma_2$-equivariant maps 
\[
(K\times K\setminus K) \times I \to \widetilde S^{d-1} \times I\to \widehat S^d
\]
where the latter map is the obvious quotient. This composite map is a null homotopy of the constant zero map $K\times K\setminus K \to \widehat S^d$. This is why~\eqref{map} implies~\eqref{homotopic}. It is obvious that~\eqref{homotopic} implies~\eqref{cobordism}.

For the reverse implication in the last statement of the lemma, let $\widetilde\Omega \widehat S^d$ be the space of paths in $\widehat S^d$ from the basepoint $\infty$ to $0$. There is a canonical $\Sigma_2$-equivariant map $\widetilde S^{d-1}\to \widetilde\Omega \widehat S^d$. It follows from the Blakers-Massey theorem that this map is $2d-3$-connected. It follows that the induced map of mapping spaces
\[
\map(K\times K\setminus K, \widetilde S^{d-1})^{\Sigma_2}\to \map(K\times K\setminus K, \widetilde\Omega \widehat S^d)^{\Sigma_2}
\]
is $2d-2\dim(K)-3$-connected. So if $d-\dim(K)\ge 2$ this map is at least $1$-connected, and therefore induces a bijection on $\pi_0$. But a $\Sigma_2$-equivariant map $K\times K\setminus K\to \widetilde\Omega \widehat S^d$ is the same thing as a $\Sigma_2$-equivariant null homotopy of the constant zero map from $K\times K\setminus K$ to $\widehat S^d$. Thus, under the assumption $d\ge \dim(K)+2$, condition~\eqref{homotopic} implies~\eqref{map}. 

Finally, the map $\widehat S^d\to \Omega^\infty\Sigma^\infty\widehat S^d$ is $2d-1$-connected by the Freudenthal suspension theorem. It follows that~\eqref{cobordism} implies~\eqref{homotopic} when $d\ge \dim(K)+1$, which is a weaker condition than stated in the lemma.
\end{proof}
Lemma~\ref{lem: characterizations} points to several (equivalent) ways to define a cohomological obstruction to the existence of a $\Sigma_2$-equivariant map $K\times K\setminus K\to \widetilde S^{d-1}$. To begin with, the map given in part~\eqref{cobordism} of the lemma can be interpreted as an element of an equivariant stable cohomotopy group, or equivalently an equivariant framed cobordism group of $K\times K\setminus K$. We denote this element by $\vkfr{K}$. Lemma~\ref{lem: characterizations} says that $\vkfr{K}$ is a complete obstruction to the existence of a $\Sigma_2$-equivariant map $K\times K\to \widetilde S^{d-1}$ when $\dim(K)+2\le d$. 

The natural map of spectra $\Sigma^\infty S^0\to H\mathbb Z$ induces a $\Sigma_2$-equivariant map 
\begin{equation}\label{eq: Hurewicz}
\Omega^\infty \Sigma^\infty \widehat S^d\to \Omega^\infty H\mathbb Z \wedge \widehat S^d\simeq K(\mathbb Z[(-1)^d], d).
\end{equation}
Here $K(\mathbb Z[(-1)^d], d)$ denotes the Eilenberg-Mac Lane space with an action of $\Sigma_2$, that on the non-trivial homotopy group realizes the representation $Z[(-1)^d]$, which is the trivial representation if $d$ is even and the sign representation if $d$ is odd. Any two such Eilenberg-Mac Lane spaces are weakly equivariantly equivalent.

Composing the maps in Lemma~\ref{lem: characterizations}\eqref{cobordism} and~\eqref{eq: Hurewicz}, we obtain a $\Sigma_2$-equivariant map
\[
K\times K\setminus K \to K(\mathbb Z[(-1)^d], d).
\]
This map defines an element in the equivariant cohomology group $\vk{K}\in H^d_{\Sigma_2}(K\times K\setminus K; \mathbb Z[(-1)^d])$. This is the classical van Kampen obstruction. It is the same as the Euler class of the vector bundle~\eqref{eq: fibration}. The classical van Kampen obstruction is a complete obstruction to the existence of a $\Sigma_2$-equivariant map $K\times K \setminus K\to \widetilde S^{d-1}$ when $d=2\dim(K)$. We are especially interested in the case when $4=d=2\dim(K)=\dim(K)+2$. In this case, the cohomological obstruction is a complete obstruction to the existence of an equivariant map (but not to the existence of an embedding $K\hookrightarrow \mathbb R^d$), and using the framed cobordism version does not add information. But in other situations $\vkfr{K}$ contains more information than $\vk{K}$.
\begin{remark}
The framed cobordism viewpoint points to a geometric interpretation of the van Kampen obstruction. It is perhaps even more convincing in the context of smooth manifolds. In that context, the analogue of the van Kampen obstruction is the obstruction for lifting from the first to the second stage of the Goodwillie-Weiss tower. In other words, it is the first obstruction to an immersion of a smooth manifold $M$ into $\mathbb R^d$ being regularly homotopic to an embedding. This obstruction is an element in the relative equivariant cobordism group $\Omega_{fr}^{\widehat {\mathbb R}^d}(M\times M, M)$, and it can be interpreted as the framed cobordism class of the double points manifold of an immersion. This is explained, for example, in the introduction to~\cite{Munson}. In the case of topological embeddings of a $2$-dimensional complex in $\mathbb R^4$, the van Kampen obstruction also can be interepreted as a double points obstruction. Of course this interpretation is well-known, and indeed it was how van Kampen thought about it. We review this in Section~\ref{sec:van Kampen}. 
\end{remark}

\subsection{The secondary obstruction} \label{subsection: secondary} Now let us consider the next step. Suppose we have a finite complex $K$ for which $\vk{K}$ (or $\vkfr{K}$) vanishes, and suppose we choose a $\Sigma_2$-equivariant map $f_2\colon\conf{K}{2}\to\conf{\mathbb R^d}{2}$. We want to know if $f_2$ is $\Sigma_2$-equivariantly homotopic to the deleted square of some embedding $f\colon K\hookrightarrow \mathbb R^d$.

Suppose $W$ is a space with an action of $\Sigma_2$. Then we endow the space $W\times W\times W$ with an action of $\Sigma_3$ via the homeomorphism $W\times W\times W\cong \map_{\Sigma_2}(\Sigma_3, W)$. In particular, the spaces $\conf{X}{2}^3$ (for any space $X$) and $(\widetilde S^{d-1})^3$ are equipped with a natural action of $\Sigma_3$ in this way.

For any space $X$, a $\Sigma_3$-equivariant map $\conf{X}{3}\to \conf{X}{2}^3$ is the same thing as a $\Sigma_2$-equivariant map $\conf{X}{3}\to \conf{X}{2}$, where $\Sigma_2\subset \Sigma_3$ is identified, as usual, with the subgroup permuting $1, 2$. There is an obvious $\Sigma_2$-equivariant projection map $\conf{X}{3}\to \conf{X}{2}$ which sends $(x_1, x_2, x_3)$ to $(x_1, x_2)$. This map induces a canonical $\Sigma_3$-equivariant map
\begin{equation}\label{eq: three-two}
\begin{array}{ccccccccc}
p_X &\colon &\conf{X}{3} & \to & \conf{X}{2}&\times &\conf{X}{2} & \times & \conf{X}{2} \\[5pt]
& &(x_1, x_2, x_3) &\mapsto & (x_1, x_2)&,& (x_2, x_3)&,& (x_3, x_1)
\end{array}
\end{equation}
This map is natural with respect to embeddings of $X$. Therefore, an embedding $f\colon K\hookrightarrow \mathbb R^d$ induces a commutative square as we saw in the introduction~\eqref{Delta3}. Conversely, if $f_2\colon \conf{K}{2}\to \conf{\R^d}{2}$ is a $\Sigma_2$-equivariant map, then a necessary condition for $f_2$ to be equivarintly homotopic to the deleted square of an embedding is that the homotopy lifting problem in the following diagram has a $\Sigma_3$-equivariant solution
\begin{equation}\label{fig: lifting problem}  \begin{tikzcd}[column sep=5em]
  & \conf{\R^d}{3} \arrow{d}{p_{\R^d}}  \\
\conf{K}{3}\arrow[ru, dashed, shift left=2]\arrow[r, "(f_2)^3\circ p_K"]  & \conf{\R^d}{2}^{\times 3}
\end{tikzcd}
\end{equation}

At this point we want to bring obstruction theory into play. For this, we need to examine the map $p_{\R^d}\colon \conf{\R^d}{3}\to \conf{\R^d}{2}^3$ a little more closely. 

To describe the effect of the map $p_{\R^d}$ in homology, let us recall some facts about the homology of configuration spaces. Recall that there is an equivalence $\conf{\R^d}{2}\simeq S^{d-1}$. Let $u\in H^{d-1}(\conf{\mathbb R^d}{2})$ be a fixed generator. 
\begin{definition} \label{def: Arnold}
The {\it Arnold class} is the following cohomological element.
\[ u\otimes u\otimes 1+(-1)^{d-1} u\otimes 1\otimes u + 1\otimes u\otimes u  \in H^{2d-2}(\conf{\mathbb R^d}{2}\times \conf{\mathbb R^d}{2}\times \conf{\mathbb R^d}{2}).\]
\end{definition}
\begin{remark}\label{remark: arnold invariant}
Notice that the group $\Sigma_3$ acts by $(-1)^{d-1}$ on the Arnold class. That is, even permutations take the Arnold class to itself, and odd permutations multiply it by $(-1)^{d-1}$. This means that the Arnold class is an element of the invariant cohomology group
\[
H^{2d-2}\left(\conf{\mathbb R^d}{2}\times \conf{\mathbb R^d}{2}\times \conf{\mathbb R^d}{2}; \Z[(-1)^{d-1}]\right)^{\Sigma_3}
\]
\end{remark}
The following lemma is well-known.
\begin{lemma}\label{lemma: arnold relation}
\[
p_{\R^d}\colon \conf{\R^d}{3}\to \conf{\R^d}{2}^3
\]
is surjective in cohomology, and its kernel in cohomology is the ideal generated by the Arnold class. 
\end{lemma}
We refer to the statement of this lemma as the Arnold relation. The original reference is~\cite{Arnold}, where it is proved for configuration spaces in $\mathbb R^2$. The general result is proved in~\cite[Lemma 1.3 and Proposition 1.4]{CT}).
The following corollary is an an easy consequence of the lemma, and is also well-known. Let us recall once again that $\conf{\R^d}{2}$ is $\Sigma_2$-equivariantly equivalent to $S^{d-1}$ with the antipodal action. The possible sign representation $\Z[(-1)^d]$ in the statement below arises from the action of $\Sigma_2$ on $H_{d-1}(S^{d-1})$.
\begin{corollary}\label{corollary: homology}
The map $p_{\R^d}$ is $2d-3$-connected, and moreover it induces an isomorphism in homology and cohomology in degrees up to and including $2d-3$. In degree $2d-2$ there is an isomorphism of abelian groups $H_{2d-2}(\conf{\R^d}{3})\cong \Z^2$ and an isomorphism of $\Sigma_3$-modules $$H_{2d-2}(\conf{\R^d}{2}^3)\cong \Z[\Sigma_3]\otimes_{\Z[\Sigma_2]} \Z[(-1)^{d-1}].$$ Moreover, the homomorphism in $H_{2d-2}$ induced by $p_{\R^d}$ fits in a short exact sequence of $\Sigma_3$-modules
\[
0\to H_{2d-2}(\conf{\R^d}{3})\to H_{2d-2}(\conf{R^d}{2}^3) \to Z[(-1)^{d-1}] \to 0
\]
where the second homomorphism can be identified with the canonical surjection of $\Sigma_3$-modules
\[
\Z[\Sigma_3]\otimes_{\Z[\Sigma_2]} \Z[(-1)^{d-1}]\to \Z[(-1)^{d-1}].
\]
\end{corollary}
It is worth noticing that the short exact sequence splits, but not $\Sigma_3$-equivariantly.

Let $F_d$ be the homotopy fiber of the map $p_{\R^d}\colon \conf{R^d}{3}\to \conf{\R^d}{2}^3$. 
It follows from Corollary~\ref{corollary: homology} that the first non-trivial homotopy group of $F_d$ is $\pi_{2d-3}(F_d)$, and it is isomorphic to $\Z$. A priori, the homotopy groups of $F_d$ form a local coefficients system over  $\conf{\R^d}{3}$. We will generally assume that $d\ge 3$. With this assumption, the spaces $\conf{\R^d}{3}, \conf{\R^d}{2}$ and $F_d$ are simply-connected, and the groups $\pi_q(F_d)$ form a trivial coefficients system over $\conf{\R^d}{3}$.  Furthermore, the action of $\Sigma_3$ on $\conf{R^d}{3}$ and $\conf{R^d}{2}^3$ induces a well-defined action of $\Sigma_3$ on $\pi_q(F_d)$. Similarly, the relative homotopy groups $\pi_q(\conf{R^d}{2}^3, \conf{R^d}{3})$ are well-defined abelian groups with an action of $\Sigma_3$, independently of basepoints. 

Taking the first dimension in which the relevant homotopy-group is non-trivial, we obtain isomorphisms of groups with an action of $\Sigma_3$:
\[
\pi_{2d-3}(F_d)\cong \pi_{2d-2}(\conf{R^d}{2}^3, \conf{R^d}{3})\xrightarrow{\cong} H_{2d-2}(\conf{R^d}{2}^3, \conf{R^d}{3}) \cong \Z[(-1)^{d-1}].
\]
Here the first isomorphism is by standard homotopy theory, the second isomorphism is the relative Hurewicz isomorphism, and the third isomorphism follows from Corollary~\ref{corollary: homology}.

We will use this to define an obstruction to the lifting problem indicated in Diagram~\eqref{fig: lifting problem}. Suppose we have a map (and throughout this discussion, whenever we say ``map'' we mean ``$\Sigma_3$-equivariant map'') $\conf{K}{3}\to \conf{\R^d}{2}^3$, and we want to lift it to a map  $\conf{K}{3}\to \conf{\R^d}{3}$. Since the map $\conf{\R^d}{3}\to \conf{\R^d}{2}^3$ is $2d-3$-connected, obstruction theory tells us that the principal obstruction to the existence of a lift lies in the equivariant cohomology group 
\[
H_{\Sigma_3}^{2d-2}(\conf{K}{3}; \pi_{2d-3}(F_d))=H_{\Sigma_3}^{2d-2}(\conf{K}{3}; \Z[(-1)^{d-1}]).
\]
One well-known construction of the obstruction goes through induction on skeleta of $\conf{K}{3}$. See Steenrod~\cite[Part III]{Steenrod} for an exposition of this approach to obstruction theory in the non-equivariant setting, and Bredon~\cite[Chapter II]{Bredon} for the equivariant version. Strictly speaking, $\conf{K}{3}$ does not have a canonical cell structure, but in Section~\ref{sec:equality} we will apply the skeletal approach to obstruction theory to a subspace $\confs{K}{3}$ of $\conf{K}{3}$, which does have a canonical cellular structure.

 But now we will show another, more homotopy theoretic construction of the cohomological obstruction, which uses Eilenberg - Mac Lane spaces instead of cellular cochains. Standard homotopy theoretic arguments show that the two approaches lead to the same cohomological class when applied to cell complexes. There are a couple of advantages to the homotopy-theoretic approach. One is that it does not depend on choosing a cell structure on $\conf{K}{3}$, and is more canonical than the skeletal approach. The second, perhaps more interesting reason is that the homotopy-theoretic method can be easily modified to produce a stronger obstruction, that resides in framed cobordism (a.k.a stable cohomotopy) rather than ordinary cohomology.

We will develop by hand the bits of obstruction theory that we need. 
We refer the reader to~\cite[Section 4.3]{Hatcher} for a more systematic exposition of the approach to obstruction theory via Posnikov towers in the non-equivariant setting. We refer to~\cite[Chapter II.1]{May-Alaska} for a brief review of Postnikov towers for spaces with an action of a group.

Now that we are looking at spaces with an action of $\Sigma_3$, let $K(\Z[(-1)^d], 2d-2)$ denote an Eilenberg-Mac Lane space with an action of $\Sigma_3$ that acts by $\Z[(-1)^d]$ on the non-trivial homotopy group. Lemma~\ref{lemma: cohomology classifying} below is an easy consequence of Corollary~\ref{corollary: homology}. Before stating the lemma, let us review the definition of a $k$-(co)cartesian square diagram. For a thorough review of the concepts surrounding (co)-cartesian cubical diagrams we recommend~\cite{G} or~\cite{Munson-Volic}
\begin{definition}
Suppose that we have a commutative diagram 
\[
\begin{tikzcd}
X_0 \arrow[r] \arrow{d} & X_1 \arrow{d} \\
X_2 \arrow{r}& X_{12}
\end{tikzcd}
\]
One says that the diagram is {\em $k$-cartesian} if the induced map from $X_0$ to the homotopy pullback of 
\[
X_2\to X_{12}\leftarrow X_1
\]
is $k$-connected. Dually, the diagram is {\em $k$-cocartesian} if the induced map from the homotopy pushout 
\[
X_2\leftarrow X_{0}\rightarrow X_1
\]
to $X_{12}$ is $k$-connected.
\end{definition}
Notice that if, say, $X_2\simeq *$, then to say that the square is (co)cartesian is equivalent to saying that $X_0\to X_1\to X_{12}$ is a homotopy (co)fibration sequence.
\begin{lemma}\label{lemma: cohomology classifying}
Assume that $d\ge 3$. There exists a model of the Eilenberg-Mac Lane space $K(\Z[(-1)^d], 2d-2)$ for which there is a $\Sigma_3$-equivariant map 
\[
\conf{R^d}{2}^3 \to K(\Z[(-1)^{d-1}], 2d-2),
\] such that the composite map 
\[
\conf{\R^d}{3}\xrightarrow{p_{\R^d}}\conf{R^d}{2}^3 \to K(\Z[(-1)^{d-1}], 2d-2)
\]
is equivariantly null-homotopic, and the following diagram is $2d-2$-cartesian
\begin{equation}\label{eq: cohomology classifying square}
\begin{tikzcd}
\conf{\R^d}{3} \arrow[r,"p_{\R^d}"] \arrow{d}  & \conf{\R^d}{2}^3 \arrow{d} \\
* \arrow{r}& K(\Z[(-1)^{d-1}], 2d-2)
\end{tikzcd}.
\end{equation}
\end{lemma}
\begin{proof}
Let $C$ be the homotopy cofiber of the map $p_{\R^d}$. Then $C$ is a pointed space, and there is a cocartesian square
\[
\begin{tikzcd}
\conf{\R^d}{3} \arrow[r,"p_{\R^d}"] \arrow{d} & \conf{\R^d}{2}^3 \arrow{d} \\
* \arrow{r}& C
\end{tikzcd}.
\]
By Corollary~\ref{corollary: homology}
the map $p_{\R^d}$ is $2d-3$-connected. Also the space $\conf{\R^d}{3}$ is $d-2$-connected, so the left vertical map in the square diagram is $d-1$-connected. By the Blakers-Massey theorem, it follows that the square is $3d-5$-cartesian. It also follows from Corollary~\ref{corollary: homology} that the bottom non-trivial homology group of $C$ occurs in dimension $2d-2$, and in this dimension the homology group of $C$ is isomorphic to $\Z[(-1)^{d-1}]$ as a $\Sigma_3$-module. We assume that $d$ is at least $3$, so $\conf{\R^d}{2}$ and $\conf{\R^d}{3}$ are simply connected. By Hurewicz theorem, the bottom non-trivial homotopy group of $C$ also occurs in dimension $2d-2$, and furthermore $\pi_{2d-2}(C)\cong \Z[(-1)^{d-1}]$ as a $\Sigma_3$-module. By the theory of equivariant Postnikov towers~\cite[Chapter II.1]{May-Alaska}, there exists a model for the Eilenberg-Mac Lane space $K(\Z[(-1)^{d-1}], 2d-2)$ equipped with a map $C\to K(\Z[(-1)^{d-1}], 2d-2)$ that induces an isomorphism on homotopy groups in dimensions up to $2d-2$. This map is an epimorphism in dimension $2d-1$, because $\pi_{2d-1}(K(\Z[(-1)^{d-1}], 2d-2))=0$. In other words, the map $C\to K(\Z[(-1)^{d-1}], 2d-2)$ is $2d-1$-connected. By substituting $K(\Z[(-1)^{d-1}], 2d-2)$ for $C$ in the square diagram at the beginning of the proof, we obtain the required diagram~\eqref{eq: cohomology classifying square}. By our calculations this diagram is $\min(3d-5, 2d-2)$-cartesian. Since $d\ge 3$, it is $2d-2$-cartesian, as required.
\end{proof}

Now let us consider again the lifting problem in figure~\eqref{fig: lifting problem}. We have a $\Sigma_3$-equivariant map 
\[
\conf{K}{3}\xrightarrow{(f_2)^3\circ p_K} \conf{\R^d}{2}^{3}.
\]
Composing with the $\Sigma_3$-equivariant map $\conf{R^d}{2}^3 \to K(\Z[(-1)^{d-1}], 2d-2)$ constructed in Lemma~\ref{lemma: cohomology classifying} we obtain a composition of maps
\begin{equation}\label{eq: map representing vk}
\conf{K}{3}\xrightarrow{p_K} \conf{K}{2}^3 \xrightarrow{f_2^3} \conf{\R^d}{2}^3 \to K(\Z[(-1)^{d-1}], 2d-2).
\end{equation}
This composition of maps defines an element in the equivariant cohomology group $$H^{2d-2}_{\Sigma_3}(\conf{K}{3}; \Z[(-1)^{d-1}]).$$
\begin{definition}
Let $\obs{K}\in H^{2d-2}_{\Sigma_3}(\conf{K}{3}; \Z[(-1)^{d-1}])$ be the element corresponding to the map~\eqref{eq: map representing vk}.
\end{definition}
The following proposition is an easy consequence of Lemma~\ref{lemma: cohomology classifying}.
\begin{proposition}\label{prop: O3}
The element $\obs{K}$ is an obstruction to the lifting problem in figure~\eqref{fig: lifting problem}. That is, if $\obs{K}\ne 0$ then the map  $(f_2)^3\circ p_K$ does not have a lift. The element $\obs{K}$ is a complete obstruction if $3\dim(K)=2d-2$.
\end{proposition}

It follows in particular that $\obs{K}$ is a complete obstruction to the lifting problem in~\eqref{fig: lifting problem} if $\dim(K)=2$ and $d=4$.

\subsection{A lift to framed cobordism} \label{subsection: cobordism} We saw earlier that the classical, cohomological van Kampen obstruction has a natural lift to a potentially stronger obstruction that lives in equivariant stable cohomotopy, a.k.a equivariant framed cobordism. The obstruction $\obs{K}$ has a similar lift, which we denote $\obsfr{K}$.

\begin{convention}\label{action convention}
Until the end of this section, and in Section \ref{sec: construction}, we consider spaces with an action of $\Sigma_3$ and no other symmetric groups. Likewise, in this section and in Section \ref{sec: construction}, let $\widehat R^2$ be the reduced standard representation of $\Sigma_3$, let $\widehat \R^{2d}=\widehat R^2\otimes \R^d$, and let $\widehat S^{2d}$ be the one-point compactification of $\widehat \R^{2d}$.

As a space, $\widehat S^{2d}$ is simply the $2d$-dimensional sphere. The `hat' is there to indicate that it is a space with a specific action of $\Sigma_3$. In the same vein, let $\widehat \Omega^2\widehat S^{2d}=\map_*(\widehat S^2, \widehat S^{2d})$ be the double loop space $\Omega^2 S^{2d}$, on which $\Sigma_3$ acts via both $S^2$ and $S^{2d}$. Similarly define the space with $\Sigma_3$-action $\widehat \Omega^2\Omega^\infty\Sigma^\infty\widehat S^{2d}$.
\end{convention}
The following proposition is a refinement of Lemma~\ref{lemma: cohomology classifying}.
\begin{proposition}\label{prop: cobordism classifying}
There is a $3d-5$-cartesian diagram of spaces with an action of $\Sigma_3$
\begin{equation}\label{eq: cobordism classifying square}
\begin{tikzcd}
\conf{\R^d}{3} \arrow[r] \arrow{d} & \conf{\R^d}{2}^3 \arrow{d} \\
* \arrow{r}& \widehat \Omega^2\Omega^\infty\Sigma^\infty\widehat S^{2d}
\end{tikzcd}.
\end{equation}
\end{proposition}
We will prove this proposition in Section~\ref{sec: construction}. For the rest of the section, we consider some consequences. It follows from the proposition that given a $\Sigma_2$-equivariant map $f_2\colon \conf{K}{2}\to \conf{\R^d}{2}$, a necessary condition for the lifting problem~\eqref{fig: lifting problem} to have a solution (and therefore also for $f_2$ to be equivariantly homotopic to the deleted square of some embedding) is that the following composition is $\Sigma_3$-equivariantly null-homotopic (compare with~\eqref{eq: map representing vk}):
\[
\conf{K}{3}\xrightarrow{p_K} \conf{K}{2}^3 \xrightarrow{f_2^3} \conf{\R^d}{2}^3 \to \widehat \Omega^2\Omega^\infty\Sigma^\infty\widehat S^{2d}.
\]
We interpret this composition as an element in the equivariant stable cohomotopy of $\conf{K}{3}$, or equivalently in the equivariant framed cobordism group $\obsfr{K}\in \Omega^{\widehat \R^{2(d-1)}}_{fr}(\conf{K}{3})$. The class $\obsfr{K}$ is an obstruction to a solution of the lifting problem~\eqref{fig: lifting problem}. $\obsfr{K}$ is a refinement of $\obs{K}$ in the same way as $\vkfr{K}$ is a refinement of $\vk{K}$. $\obsfr{K}$ is a complete obstruction to the lifting problem if $3\dim(K)\le 3d-5$, while $\obs{K}$ is a complete obstruction if $3\dim(K)\le 2d-2$. Of course when $\dim(K)=2$ and $d=4$ both conditions hold, and $\obsfr{K}$ does not provide any more information than $\obs{K}$. In Section~\ref{geometric interpretation} we will use our specific construction of the classifying map to give another interpretation of $\obs{K}$. 
\begin{remark}
The obstruction $\obsfr{M}$ in the context of smooth embeddings is the subject of Munson's paper~\cite{Munson}. In particular, Proposition~\ref{prop: cobordism classifying} is proved there. We give a different proof in Section~\ref{sec: construction}. As a consequence, we will give another geometric interpretation of $\obsfr{K}$ in Section~\ref{geometric interpretation}. This interpretation is hinted at in [op. cit.].
\end{remark}

\section{Geometric obstructions from Whitney towers} \label{sec:Whitney towers}

This section starts by reviewing a geometric formulation of van Kampen's obstruction (Section \ref{sec:van Kampen}) and operations on Whitney disks (Section \ref{Whitney subsection}) which are commonly used in $4$-manifold topology. These techniques are then used to establish new results: a relation between Whitney disks and equivariant maps of configuration spaces (Section \ref{sec: from Whitney to config}) and higher embedding obstructions for $2$-complexes in $\R^4$ based on intersections of Whitney disks: 
$\gobs{K}$ in Section \ref{sec:higher geometric} and $\gobsn{K}, n>3$ in Section \ref{sec: Whitney towers}. The relation of $\gobs{K}$ to the obstruction $\obs{K}$ defined above is the subject of Section \ref{sec:equality}.

\subsection{The van Kampen obstruction} \label{sec:van Kampen} The discussion in the paper so far concerned the general embedding problem for $m$-complexes in $\R^d$. Here we restrict to the original van Kampen's context where $d=2m$. Later in this section we will specialize further to $m=2$.
We start by recalling a geometric description of the van Kampen obstruction \begin{equation} \label{eq: vk def}
\svk{K}\in H^{2m}_{{\Sigma}_2}(\confs{K}{2}; {\mathbb Z})
\end{equation}

to embeddability of an $m$-complex $K$ into ${\mathbb R}^{2m}$. 
This was the construction outlined by van Kampen in \cite{vK}; the details were clarified in \cite{Shapiro, Wu}, see also \cite{FKT}. As in the introduction the notation $\confs{K}{2}$ denotes the ``simplicial'' configuration space $K\times K\smallsetminus \Delta$ where  $\Delta$ consisting of all products of simplices ${\sigma}_1\times {\sigma}_2$ having a vertex in common. The group $\Sigma_2$ acts on the configuration space $K\times K\smallsetminus {\Delta}$ by exchanging the factors; it may be seen from the description below that the action of $\Sigma_2$ on the coefficients is trivial, cf. \cite{Shapiro, Melikhov} (note that the sign was misstated as
$(-1)^m$ in \cite{FKT}.) 

Note that $\svk{K}$ is an element of the cohomology group of $\confs{K}{2}$, while $\vk{K}$ (considered in the introduction and in Section \ref{sec: Constructing the obstruction}) is an element of the cohomology group of the configuration space $\conf{K}{2}$ defined using the point-set diagonal. 
The invariant $\vk{K}$ is the ``universal'' van Kampen obstruction, independent of the simplicial structure, and $\svk{K}$ may be recovered from it: $\svk{K}=i^*\vk{K}$, where $i$ is the inclusion map $\confs{K}{2}\subset \conf{K}{2}$, cf. \cite[Section 3]{Melikhov}.  A priori $\svk{K}$ could be a weaker invariant since it does not keep track of intersections of adjacent simplices. Nevertheless, it is a complete embedding obstruction for $m$-complexes in $\R^{2m}$ for $m>2$: intersections of  adjacent simplices may be removed using a version of the Whitney trick, cf. \cite[Lemma 5]{FKT}.

\begin{remark}
The obstruction theory in Section \ref{sec: Constructing the obstruction} was developed for embeddings of finite CW complexes. The geometric approach presented here is based on intersection theory and it applies to finite simplicial complexes.
We will interchangeably use the terms {\em cells} and {\em simplices} in the context of simplicial complexes; this should not cause confusion.
\end{remark}

Consider any general position map $f\! : K\longrightarrow {\mathbb R}^{2m}$. Endow the $m$-cells of $K$ with arbitrary orientations, and for any two $m$-cells ${\sigma}_1, {\sigma}_2$ without vertices in common, consider the algebraic intersection number $f({\sigma}_1)\cdot f({\sigma}_2)\in {\mathbb Z}$. This gives a $\Sigma_2$-equivariant cochain 
\begin{equation} \label{eq: vk cohain}
o^{}_f\!: C_{2m}(K\times K\smallsetminus {\Delta})\longrightarrow {\mathbb Z}. 
\end{equation}
Since this is a top-dimensional cochain, it is a cocycle.  Its cohomology class equals the van Kampen obstruction $\svk{K}$. 

The fact that this cohomology class is independent of a choice of $f$ may be seen geometrically as follows (see \cite[Lemma 1, Section 2.4]{FKT} for more details). Any two general position maps $f_0, f_1\! : K\longrightarrow {\mathbb R}^{2m}$ are connected by a $1$-parameter family of maps $f_t$ where at a non-generic time $t_i$ an $m$-cell $\s$ intersects an $(m-1)$-cell $\nu$. Topologically the maps $f_{t_i-{\epsilon}}$ and $f_{t_i+{\epsilon}}$ differ by a ``finger move'', that is tubing ${\sigma}$ into a small $m$-sphere linking $\nu$ in ${\mathbb R}^{2m}$, Figure \ref{fig:FingerMove}. The effect of this elementary homotopy on the van Kampen cochain is precisely the addition of the coboundary ${\delta}(u_{{\sigma},{\nu}})$, where $u_{{\sigma},{\nu}}$ is the $\Sigma_2$-equivaraint ``elementary $(2m-1)$-cochain'' dual to the $(2m-1)$-cells ${\sigma}\times{\nu}, {\nu}\times {\sigma}$.

\begin{figure}[ht]
\centering
\includegraphics[height=2.7cm]{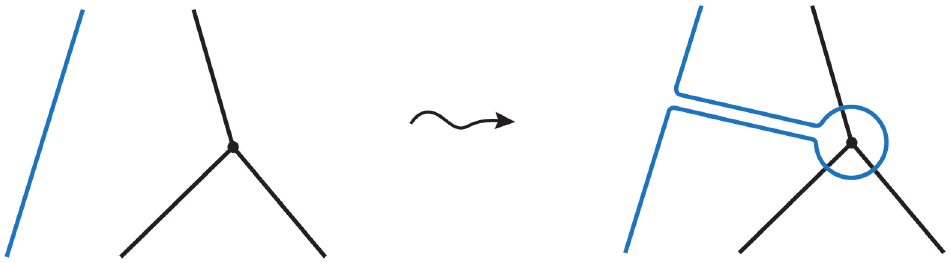}
{    
       \put(-205,35){${\nu} $}
        \put(-272,55){${ {\sigma}} $}
    }  
    \caption{Finger move: homotopy of maps $f\co K\longrightarrow {\mathbb R}^{2m}$}
\label{fig:FingerMove}
\end{figure}

This argument has the following corollary. 

\begin{lemma} \label{trivial vk} 
Any cocycle representative of the cohomology class $\svk{K} \in H^{2m}_{{\mathbb Z}/2}(K\times K\smallsetminus {\Delta}; {\mathbb Z})$ may be realized as the cocycle $o^{}_f$ for some general position map $f\co K\longrightarrow {\mathbb R}^{2m}$.
In particular, if the van Kampen obstruction $\svk{K}$ vanishes then  there exists a general position map $f\co K\longrightarrow {\mathbb R}^{2m}$ such that the cocycle $o^{}_f$ is identically zero. In other words, in this case for any two non-adjacent $2$-cells ${\sigma}$, ${\tau}$ the algebraic intersection number $f({\sigma})\cdot f({\tau})$ is zero. 
\end{lemma}

\subsection{Operations on Whitney disks} \label{Whitney subsection}

The rest of Section \ref{sec:Whitney towers} concerns $2$-complexes in $\R^4$.
Assume the van Kampen class $\svk{K}$ vanishes. By Lemma \ref{trivial vk},  using finger moves on $2$-cells as shown in Figure \ref{fig:FingerMove}, a map $f$ may be chosen so that $f({\sigma}_i)\cdot f({\sigma}_j)=0$ for any non-adjacent $2$-cells ${\sigma}_i, {\sigma}_j$. As usual, one groups intersection points $f({\sigma}_i)\cap f({\sigma}_j)$ into canceling pairs, chooses  Whitney arcs connecting them in  ${\sigma}_i, {\sigma}_j$, and considers Whitney disks $W_{ij}$ for these intersections. Note that all Whitney arcs in each $2$-cell may be assumed to be pairwise disjoint.
Unlike the situation in higher dimensions where by general position a Whitney disk may be assumed to be embedded and to have interior disjoint from $K$, in $4$-space 
generically $W_{ij}$ will have self-intersections and also intersect the $2$-cells of $K$. Moreover, the framing (the relative Euler number of the normal bundle of the Whitney disk) might be non-zero, but it may be corrected by boundary twisting \cite[Section 1.3]{FQ}.
A detailed discussion of Whitney disks in this dimension is given in
\cite[Section 1.4]{FQ}. This section summarizes the operations on Whitney disks and their relation with capped surfaces which will be used in the proofs in Section \ref{sec:equality}.

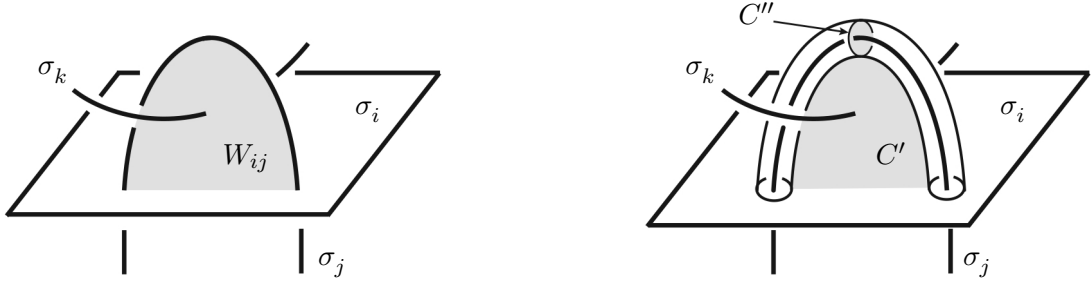
\begin{figure}[ht]
\centering
\includegraphics[height=3.6cm]{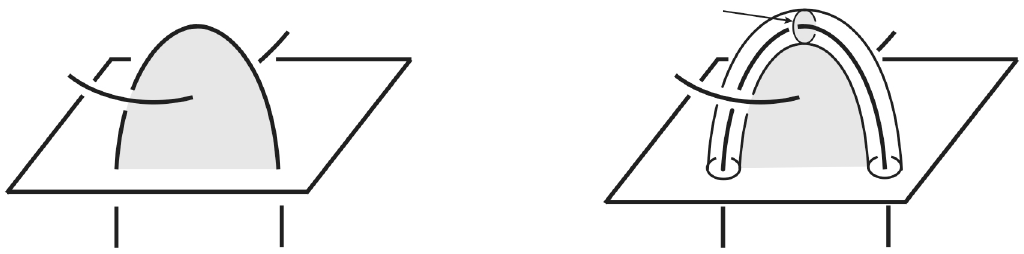}
{\small     
    \put(-294,6){${\sigma}^{}_j $}
       \put(-280,63){${\sigma}^{}_i $}
       \put(-400,78){${\sigma}^{}_k $}
       \put(-330,44){$W_{ij}$}
        \put(-50,6){${\sigma}^{}_j $}
       \put(-36,63){${\sigma}^{}_i $}
       \put(-155,78){${\sigma}^{}_k $}
       \put(-83,44){$C'$}
       \put(-135,97){$C''$}
    }  
    \caption{A Whitney disk and the associated capped surface}
\label{Whitney}
\end{figure}

{\bf Convention.} To avoid cumbersome notation, we will frequently omit the reference to a map $f$ and keep the notation $\s$ for the image of a cell $\s$ under $f$. 

A typical configuration is shown on the left in Figure \ref{Whitney}. It is a usual representation in $3$-space $\R^3\times\{ 0\}$ (the `present') of intersecting surfaces in $\R^4=\R^3\times\R$ where the $\R$ factor is thought of as time. Here $\s_i$ is pictured as a surface in $\R^3$ while $\s_j$, $\s_k$ are arcs which extend as the product (arc$\times I$) into the past and the future. The  Whitney disk $W_{ij}\subset \R^3\times\{ 0\}$ pairs up two generic intersection points $\s_i\cap \s_j$ of opposite signs, and $W_{ij}$ in the figure has a generic intersection point with another $2$-cell $\s_k$. The result of the Whitney move in this setting is shown in Figure \ref{afterWhitney}: the two intersection points $\s_i\cap \s_j$ are eliminated, but two new intersection points $\s_i\cap \s_k$ are created instead. 

\begin{figure}[ht]
\centering
\includegraphics[height=3.6cm]{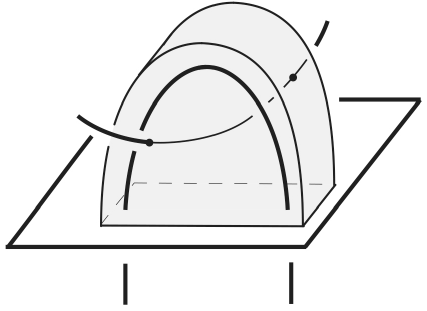}
{\small     
        \put(-41,5){${\sigma}^{}_j $}
       \put(-25,60){${\sigma}^{}_i $}
       \put(-125,71){${\sigma}^{}_k $}
    }  
    \caption{The result of the Whitney move}
\label{afterWhitney}
\end{figure}

In fact, the picture is symmetric with respect to the three sheets $\s_i, \s_j, \s_k$: a neighborhood of the Whitney disk $W_{ij}$ in $\R^4$ is a $4$-ball $D^4$, and the intersection of these three sheets with the boundary $3$-sphere $\partial D^4$ forms the Borromean rings (cf. \cite[Chapter 12]{FQ}), as shown in Figure \ref{fig:Bor}. Thus any two of the sheets can be arranged to be disjoint in this $4$-ball, but not all three simultaneously.

\begin{figure}[ht]
\centering
\includegraphics[height=3cm]{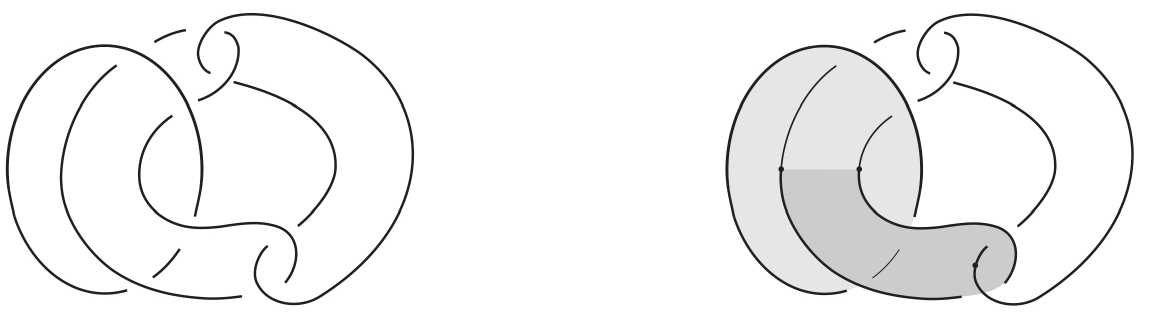}
{\small    
    \put(-6,40){$\s_k $}
    \put(-68,-6){$\s_i$}
    \put(-114,48){$ \s_j$}
    \put(-93,20){$W_{ij}$}
       \put(-198,40){$ \s_k $}
        \put(-265,-6){$ \s_i $}
       \put(-331,40){$ \s_j$}
    }  
    \caption{Left: the Borromean rings in $\partial D^4$. Right: The Whitney disk $W_{ij}$ intersects $\s_k$ in a single point}
\label{fig:Bor}
\end{figure}
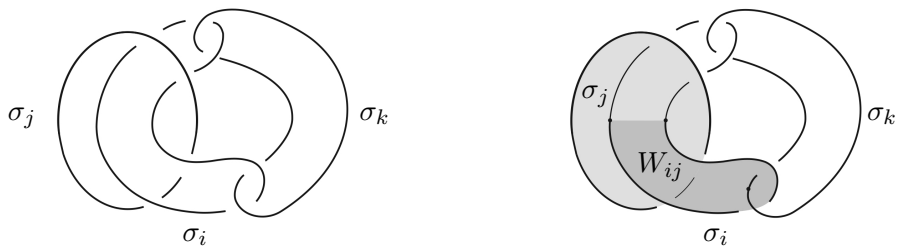

It will be convenient to view these intersections in the context of {\em capped surfaces} (or more generally {\em capped gropes} for higher-order intersections) \cite[Chapter 2]{FQ}. This is shown on the right in Figure \ref{Whitney}: a tube is added to one of the two sheets, say $\s_i$ as shown in the figure, to eliminate the two intersections $\s_i\cap \s_j$ at the cost of adding genus to $\s_i$. The new surface, still denoted $\s_i$, has two {\em caps}: disks attached to a symplectic pair of curves on $\s_i$. One of the caps, $C'$, is obtained from the Whitney disk $W_{ij}$. The other cap is a disk normal to $\s_j$ and may be thought of as a fiber of the normal bundle to $\s_j$. A general translation between Whitney towers and capped gropes is discussed in \cite{Schneiderman}. An advantage of this point of view is the symmetry between the original map of $\s_i$ (intersecting $\s_j$ in two points, as shown on the left in the figure) and the result of the Whitney move where the two intersections $\s_i\cap \s_j$ are eliminated but $\s_i$ acquires two intersections with $\s_k$. The first case is obtained by ambient surgery of the capped surface in the figure on the right along the cap $C''$, and the second case is the surgery along $C'$. There is an intermediate operation, {\em symmetric surgery} (also known as {\em contraction}) \cite[Section 2.3]{FQ} that uses both caps that will be used in the arguments in the next section. The disk obtained by surgery on $C'$ is isotopic to the surgery on $C''$, and the symmetric surgery may be thought of as the half point of this isotopy.

Consider the following {\em splitting} operation on Whitney disks. Suppose a Whitney disk $W_{ij}$ pairing up intersections between $\s_i, \s_j$ intersects two other $2$-cells, $\s_k, \s_l$ as shown on the left in Figure \ref{fig:WhitneySplit}. Consider an arc in $W_{ij}$ (drawn dashed in the figure) which separates the intersections $W_{ij}\cap \s_k$, $W_{ij}\cap \s_l$ and whose two endpoints are in the interiors of the two Whitney arcs forming the boundary of $W_{ij}$. Then a finger move on one of the sheets, say $\s_j$, along the arc introduces two new points of intersection $\s_i\cap \s_j$ and splits $W_{ij}$ into two Whitney disks $W_{ij}', W_{ij}''$ as shown in the figure on the right.
The advantage of the result is that each Whitney disk intersects only one other $2$-cell. In general, if $W_{ij}$ had $m$ intersection points with other $2$-cells, an iterated application of splitting yields $m-1$ Whitney disks, each one with a single intersection point in its interior.

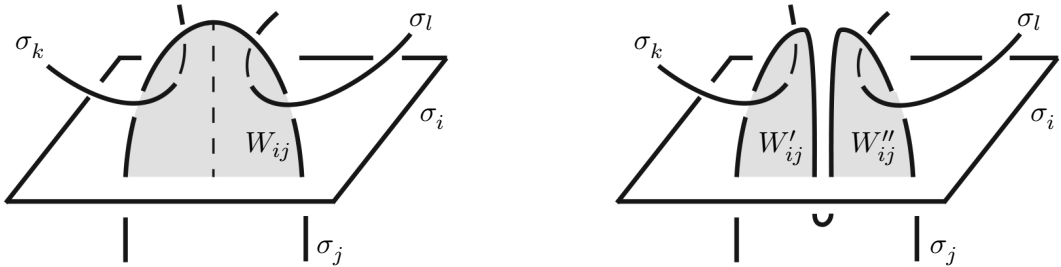
\begin{figure}[ht]
\centering
\includegraphics[height=3.6cm]{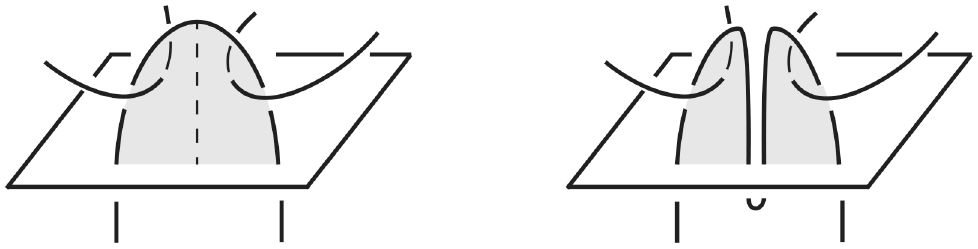}
{\small     
    \put(-283,6){${\sigma}^{}_j $}
       \put(-244,56){${\sigma}^{}_i $}
       \put(-397,83){${\sigma}^{}_k $}
       \put(-310,44){$W_{ij}$}
       \put(-248,92){$\s_l$}
        \put(-51,6){${\sigma}^{}_j $}
       \put(-13,56){${\sigma}^{}_i $}
       \put(-163,83){${\sigma}^{}_k $}
       \put(-18,92){$\s_l$}
       \put(-81,44){$W_{ij}''$}
       \put(-116,44){$W_{ij}'$}
    }  
    \caption{Splitting of a Whitney disk}
\label{fig:WhitneySplit}
\end{figure}

The discussion above referred to the situation where a Whitney disk $W_{ij}$ for  $\s_i\cap \s_j$ intersects $2$-cells which are not adjacent to $\s_i, \s_j$. In general, $W_{ij}$ will have self-intersections as well as intersections with $\s_i, \s_j$ and with $2$-cells adjacent to them. Intersections of these types are not considered in the formulation of the obstruction in Section \ref{sec:higher geometric}. (An obstruction involving these more subtle intersections will be explored in a future work. For example, the Arf invariant of a knot in $S^3$ may be defined using intersections of this type of the disk bounded by the knot in the $4$-ball, see Remark \ref{rem: triple intersection}.)

An ingredient in the formulation of higher obstructions in Section \ref{sec:higher geometric} is a local move on surfaces which replaces an intersection $\s_k\cap W_{ij}$ in Figure \ref{Whitney} with an intersection $\s_i\cap W_{jk}$ or $\s_j\cap W_{ik}$. 

To describe this operation in more detail, start with the model situation in Figure \ref{Whitney} where $W_{ij}$ has a single intersection point with $\s_k$. Perform a finger move on $\s_k$ along an arc from $\s_k\cap W_{ij}$ to a point on the Whitney arc in $\s_i$. The result is shown on the left in Figure \ref{fig:Whitney2}: now $\s_k$ is disjoint from $W_{ij}$ but there are two new intersections between $\s_i$ and $\s_k$. The finger move isotopy of $\s_k$ gives rise to a Whitney disk for these two points, denoted $W'_{ik}$ in the figure. Note however that the two Whitney disks $W_{ij}$, $W'_{ik}$ cannot be both used for Whitney moves since their boundary arcs intersect in $\s_j$. Resolving this intersection by an isotopy of the Whitney arc in the boundary of $W'_{ik}$ yields a Whitney disk $W_{ik}$ on the right in Figure \ref{fig:Whitney2}; this Whitney disk has a single intersection point with $\s_j$. (Note that after this operation the Whitney disk $W_{ij}$ is embedded and disjoint from other $2$-cells; a Whitney move along this disk can be used to eliminate the original two intersections $\s_i\cap \s_j$.)

Therefore to have a well-defined triple intersection number one has to (1) sum over Whitney disks over all pairs of indices, and (2) require that Whitney arcs are disjoint, see Section \ref{sec:higher geometric}.

\begin{figure}[ht]
\centering
\includegraphics[height=3.6cm]{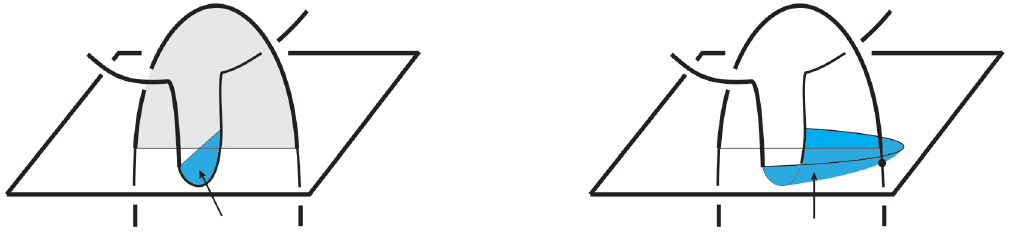}
{\small     
    \put(-306,2){${\sigma}^{}_j $}
       \put(-295,60){${\sigma}^{}_i $}
       \put(-417,84){${\sigma}^{}_k $}
       \put(-345,-2){$W'_{ik}$}
       \put(-339,45){$W_{ij}$}
        \put(-49,2){${\sigma}^{}_j $}
       \put(-39,60){${\sigma}^{}_i $}
       \put(-161,83){${\sigma}^{}_k $}
       \put(-90,-5){$W_{ik}$}
    }  
    \caption{From $\s_k\cap W_{ij}$ to $\s_j\cap W_{ik}$.}
\label{fig:Whitney2}
\end{figure}
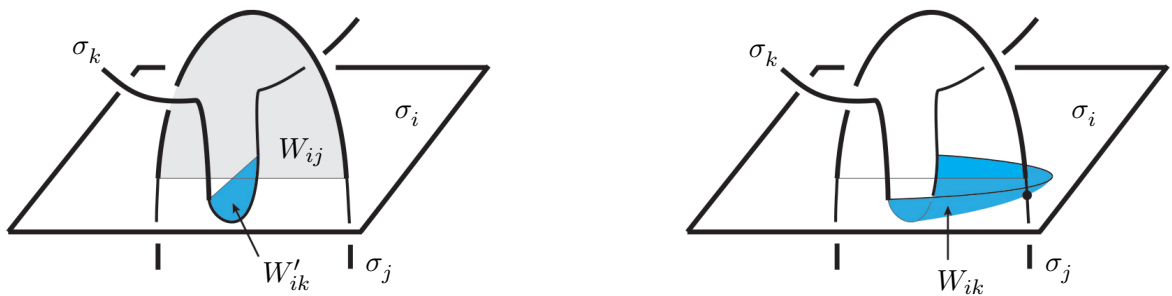

\subsection{From Whitney disks to equivariant maps of configuration spaces} \label{sec: from Whitney to config}

Let $K$ be a $2$-complex and suppose the van Kampen obstruction $\svk{K}$ vanishes. Then by Lemma \ref{trivial vk} there is a map $f\co K\longrightarrow {\mathbb R}^4$ so that the algebraic intersection number of any two non-adjacent $2$-cells in ${\mathbb R}^4$ is zero. As in Section \ref{Whitney subsection}, 
pair up the intersections with Whitney disks, so that all Whitney arcs are disjoint in each $2$-cell. This condition on the Whitney arcs will be assumed throughout the rest of the paper.
The following lemma shows that $f$ together with a choice of Whitney disks $W$ gives rise to a ${\Sigma}_2$-equivariant map $\confs{K}{2}\longrightarrow \conf{\R^d}{2}$.
The proof of this lemma explains a basic idea underlying the connection between geometric and homotopy-theoretic approaches to obstruction theory that is established in this paper. A more involved version of this argument will be given in Section \ref{sec:equality} to show that there exists a $\Sigma_3$-equivariant map of the $5$-skeleton of $\confs{K}{3}$ to $\conf{\R^d}{3}$. 
Recall from Section \ref{Whitney subsection} that any given collection of Whitney disks may be {\em split}, so that any Whitney disk has at most one intersection with a $2$-cell of $K$.

\begin{lemma} \label{Whitney-config}  Let $K$ be a $2$-complex and $f\co K\longrightarrow {\mathbb R}^4$ a general position map such that all intersections of non-adjacent $2$-cells are paired up with split Whitney disks $W$.   This data determines a ${\Sigma}_2$-equivariant map $F_{f,W}\co \confs{K}{2}\longrightarrow \conf{\R^d}{2}$.
\end{lemma}

{\em Proof}.
Given any pair of non-adjacent $2$-cells $\s_i, \s_j$, by assumption all intersections $f({\sigma}_i)\cap f({\sigma}_j)$ are paired up with Whitney disks $W_{ij}$, and the Whitney arcs in each $2$-cell are disjoint. The self-intersections and intersections of the Whitney disks will not be relevant in the following argument because the simplicial diagonal $\Delta$ is missing in the configuration space $\confs{K}{2}$. 
Since the Whitney disks are split, each $W_{ij}$ intersects a single $2$-cell $\s_k$ as in Figure \ref{Whitney}. We treat the special case that $\s_k$ is either $\s_i$ or $\s_j$ right away: if $W_{ij}$ intersects $\s_i$, perform the Whitney move along $W_{ij}$ on $\s_i$; if it intersects $\s_j$ then perform the Whitney move of $\s_j$. This results in self-intersections of either $f(\s_i)$ or $f(\s_j)$ which are irrelevant since we are working with the simplicial configuration space $\confs{K}{2}$, and so the map $F_{f,W}$ does not need to be defined on $\s_i\times\s_i$, $\s_j\times \s_j$. Thus the remaining intersections of $W_{ij}$ are with $2$-cells $\s_k$, $k\neq i,j$.

Next we describe the desired map $F_{f,W}\co \confs{K}{2}\longrightarrow \conf{\R^4}{2}$. By general position the $1$-cells and the $2$-cells of $K$ are mapped in disjointly by $f$, so $f\times f$ defines a $\Sigma_2$-equivariant map on the $3$-skeleton of $\confs{K}{2}$. Thus the goal is to extend it to the $4$-skeleton, that is to define $F_{f,W}$ on each product of two non-adjacent $2$-cells $\s_i\times \s_j$. For each such pair $\s_i, \s_j$ we pick an order $(i,j)$; for the other product $\s_j\times \s_i$ the map $F_{f,W}$ will be defined using $\Sigma_2$ equivariance.

In each $2$-cell $\s_i$ consider disjoint disk neighborhoods of the Whitney arcs for the intersections of $f(\s_i)$ with other $2$-cells; the disk neighborhoods corresponding to $W_{ij}$ are denoted $D_{ij}$, Figure \ref{fig:isotopy}. (In general $W_{ij}$ denotes the entire collection of Whitney disks for $f(\s_i)\cap f(\s_j)$, and $D_{ij}$ denotes the collection of corresponding disk neighborhoods; we illustrate the case of a single component since the argument in general is directly analogous.) If $f(\s_i)\cap f(\s_j)=\emptyset$, $D_{ij}$ is defined to be empty. Now consider the map $\widetilde f_{ij}\co K\longrightarrow {\mathbb R}^4$ which coincides with $f$ in the complement of the disk $D_{ij}$. In this disk $\widetilde f_{ij}$ is defined to be the result of the Whitney move on $f(\s_i)$ along the Whitney disk $W_{ij}$, making $\widetilde f({\sigma}_i)$ disjoint from $\widetilde f({\sigma}_j)$. If $W_{ij}$ intersected another $2$-cell $\s_k$ as in Figure \ref{Whitney}, as a result of this move  $\widetilde f_{ij}({\sigma}_i)$ intersects $\widetilde f_{ij}(\s_k)=f(\s_k)$.

\begin{figure}[ht]
\centering
\includegraphics[height=3.6cm]{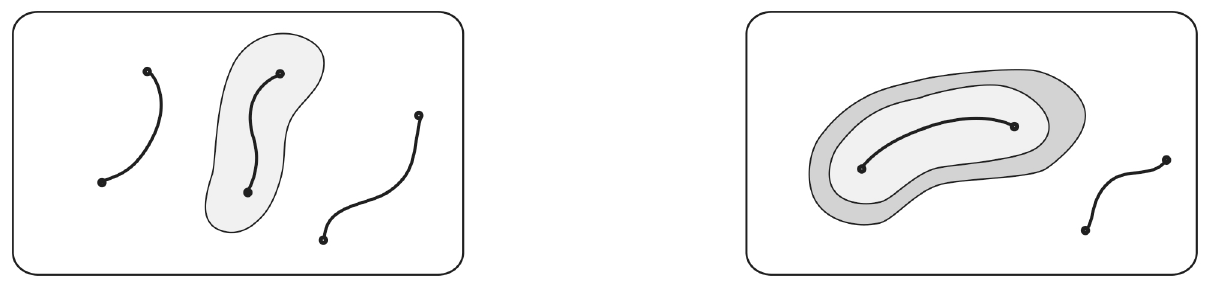}
{\scriptsize     
    \put(-260,10){${\sigma}^{}_i $}
       \put(-356,43){$D_{ij}$}
       \put(-351,26){$p_1$}
       \put(-330,80){$q_1$}
       \put(-58,61){$C_{ji}$}
       \put(-181,10){${\sigma}^{}_j $}
       \put(-100,47){$D_{ji}$}
       \put(-126,33.5){$p_2$}
       \put(-72,62){$q_2$}
    }  
    \caption{Defining the map $\s_i\times \s_j\longrightarrow \conf{\R^4}{2}$. Here $f(p_1)=f(p_2)$ and $f(q_1)=f(q_2)$ are two double points in $f(\s_i)\cap f(\s_j$).}
\label{fig:isotopy}
\end{figure}

Consider a collar $C_{ji}=\partial D_{ji}\times I$ on $\partial D_{ji}$ in $\s_j\smallsetminus {\rm int}(D_{ji})$, Figure \ref{fig:isotopy}. The collars are chosen small enough so that they are disjoint from each other in $\s_j$ for various Whitney arcs. Define 
\begin{equation} \label{collar complement} F_{f,W}|_{\s_i \times (\s_j\smallsetminus (D_{ji}\cup C_{ji}))} := (f\times f)|_{\s_i \times (\s_j\smallsetminus (D_{ji}\cup C_{ji}))}.
\end{equation}

This defines a map into the configuration space $\conf{\R^4}{2})$ since $f(\s_i)$ is disjoint from  $ f(\s_j\smallsetminus (D_{ji}\cup C_{ji}))$. On ${\s_i \times D_{ji}}$ the map is defined using the result of the Whitney move:
\begin{equation} \label{map in disk}  F_{f,W}|_{\s_i \times D_{ji}}:= (\widetilde f_{ij}\times \widetilde f_{ij}) |_{\s_i \times D_{ji}}= (\widetilde f_{ij}\times f) |_{\s_i \times D_{ji}}
\end{equation}

It remains to define $F_{f,W}$ on $\s_i\times C_{ji}$ interpolating between the maps (\ref{collar complement}), (\ref{map in disk}).
If the Whitney disk $W_{ij}$ was framed and embedded then the original map $f$ and the result of the Whitney move $\widetilde f_{ij}$ would be isotopic, with the isotopy supported in the interior of $D_{ij}$. In general, without these assumptions, these maps are homotopic rather than isotopic. Denote by $f^t_{ij}\co K\times I\longrightarrow \R^4$ this homotopy $f\simeq \widetilde f_{ij}$ given by the Whitney move, and supported in $D_{ij}$.

Identify $(x,y,t)\in \s_i\times \partial D_{ji}\times [0,1]$ with $(x,y_t)\in \s_i\times C_{ji}$ using the product structure on the collar $C_{ji}$. Using this identification,  the following map sends a point $(x,y_t)$ to $(f^t_{ij}(x), f(y_t))$:
\begin{equation} \label{homotopy def}F_{f,W}|_{\s_i \times C_{ji}}:= (f^t_{ij}\times f^t_{ij}) |_{\s_i \times C_{ji}} = (f^t_{ij}\times f) |_{\s_i \times C_{ji}}.
\end{equation}

 This matches $\widetilde f_{ij}\times f$ on $\s_i\times \partial D_{ji}$ and $f\times f$ on $\s_i\times \partial (D_{ji}\cup C_{ji})$. 
The result is a continuous map $\s_i\times \s_j\longrightarrow \conf{\R^4}{2}$, giving rise to a desired ${\Sigma}_2$-equivariant map $\confs{K}{2}\longrightarrow \conf{\R^d}{2}$.
\qed

A key point in the above proof is that even though the result of the Whitney move $\widetilde f_{ij}({\sigma}_i)$ intersects $\widetilde f_{ij}(\s_k)=f(\s_k)$, this does not affect the definition of the map $F_{f,W}$ on $\s_i\times \s_k$. The assumption of Lemma \ref{Whitney-config} is insufficient for producing a map of $3$-point configuration spaces, as we make precise in the next subsection.

\subsection{An obstruction from intersections of Whitney disks} \label{sec:higher geometric} 
We are now in a position to formulate our geometric embedding obstruction for $2$-complexes in $\R^4$ which is defined when the van Kampen obstruction vanishes.
Under this assumption, following Lemma \ref{trivial vk} consider a map $f\co K\longrightarrow {\mathbb R}^4$ where the intersection number of any two non-adjacent $2$-cells $f(\s_i)\cap f(\s_j)$ in ${\mathbb R}^4$ is zero. As in Section \ref{sec: from Whitney to config}, consider a collection $W=\{ W_{ij} \}$ of Whitney disks for $f(K)$, where $W_{ij}$ denotes the Whitney disks for $f(\s_i)\cap f(\s_j)$. As above, the Whitney arcs are assumed to be disjoint in each $2$-cell $\s_i$. 

The obstruction $\gobs{K}$, defined below, depends on the choice of $f$ and of Whitney disks $W$. Indeed, in the context of obstruction theory one expects that higher obstructions generally depend on choices of trivializations of lower order obstructions. Recall from Section \ref{sec: Constructing the obstruction} that the obstruction $\obs{K}$ to lifting to a $\Sigma_3$-equivariant map $\conf{K}{3}\longrightarrow \conf{\R^4}{3}$ depends on the choice of a $\Sigma_2$-equivariant map $f_2\co \conf{K}{2}\longrightarrow \conf{\R^4}{2}$. Moreover, by Lemma \ref{Whitney-config} the geometric data -- $f$ and $W$ -- determine such a map $f_2$ on the simplicial configuration space $\confs{K}{2}$. The relation between the two theories is extended further in Section \ref{sec:equality}.

\begin{definition}[\sl The obstruction $\gobs{K}$] \label{def: W3} Let $K, f, W$ be as above, and endow the $2$-cells of $K$ with arbitrary orientations. The orientation of Whitney disks $W_{ij}$, where $(i,j)$ is an {\em ordered} pair, is induced from the orientation on its boundary which is oriented from $-$ intersection to $+$ intersection along $f(\s_i)$  and from from $+$ to $-$ along $f(\s_j)$.
Consider the $6$-cochain:
\begin{equation} \label{eq: o3 cocohain} 
w_3\co C_6(\confs{K}{3})\longrightarrow {\mathbb Z},
\end{equation}
defined as follows.
Let ${\sigma}_i, {\sigma}_j, {\sigma}_k$ be $2$-cells of $K$ which pairwise have no vertices in common, and define 
\begin{equation} \label{eq: 02 def}
w_3(\s_i\times\s_j\times\s_k)\, =\,  W_{ij}\cdot f({\sigma}_k)   + W_{jk}\cdot f({\sigma}_i) +W_{ki}\cdot f({\sigma}_j),
\end{equation}
where the algebraic intersection numbers are defined using the orientation convention discussed above.
Note that changing the order of $i,j$ reverses the orientation of $W_{ij}$, so the cochain $w_3$ in (\ref{eq: 02 def}) is $\Sigma_3$ equivariant, where $\Sigma_3$ acts on $\Z$ according to the sign representation.
This $6$-cochain is a cocycle since it is a top-dimensional cochain on $\confs{K}{3}$. The resulting cohomology class is denoted $$\gobs{K, f, W}\in H^{6}_{\Sigma_3}(\confs{K}{3}; \Z[(-1)]).$$
When $f, W$ are clear from the context, the notation will be abbreviated to $\gobs{K}$.
\end{definition}

It is worth noting that the local move in Figure \ref{fig:Whitney2} shifts the intersection numbers between the terms of (\ref{eq: 02 def}); it is the sum that gives a meaningful invariant (see also Remark \ref{rem: triple intersection} below.)
Geometrically (\ref{eq: 02 def}) measures intersection numbers that are an obstruction to finding disjoint embedded Whitney disks needed to construct an embedding $K \hookrightarrow \R^4$. The definition depends on various choices: the pairing of $\pm$ intersections of $f(\s_i)\cap f(\s_j)$, and choices of Whitney arcs and of Whitney disks. By comparing it to the obstruction $\obs{K}$ in the next section, we show that it really depends only on the homotopy class of the map $F_{f,W}$ constructed in Lemma \ref{Whitney-config}, a fact that is not apparent from the geometric framework of the above definition. 

In addition to these cell-wise intersection considerations, of course properties of the obstruction $\gobs{K}$ depend on the cohomology of the configuration space $\confs{K}{3}$. This aspect of the obstruction is discussed in Lemma \ref{trivial vk3}, and the consequence of its vanishing is the subject of Section \ref{sec: Whitney towers}.

\begin{remark} \label{rem: Wh disk in example}
It is not difficult to see that in the example of \cite{FKT} there is a map of the $2$-complex into $\R^4$ with precisely two $2$-cells intersecting in two algebraically canceling points, with a Whitney disk intersecting one other $2$-cell as in Figure \ref{Whitney}. It follows that the corresponding cochain (\ref{eq: o3 cocohain}) is non-zero on precisely one $6$-cell of $\confs{K}{3}$; this example is discussed in detail in Section \ref{sec:Examples}.
\end{remark}

\begin{remark} \label{rem: triple intersection}
Our Definition \ref{def: W3} extends to the setting of $2$-complexes in $\R^4$ the idea of using intersections of Whitney disks with surfaces that has been widely used in $4$-manifold topology.
The construction of this type in the simplest {\em relative} case: $K^2=\coprod^3 D^2$, the disjoint union of three disks whose boundary curves form a given three-component link $L$ in $S^3=\partial D^4$, is a reformulation of Milnor's $\bar\mu$-invariant \cite{Milnor} $\bar\mu_{123}(L)$, sometimes referred to as the triple linking number. Such intersections were used to define an obstruction to representing three homotopy classes of maps of $2$-spheres into a $4$-manifold by maps with disjoint images in \cite{Matsumoto, Yamasaki}, and in the non-simply connected setting in \cite{ST0}. A version considering self-intersections to define the Arf invariant and the Kervaire-Milnor invariant was given in \cite[10.8A]{FQ}, and an extension to non-simply connected $4$-manifolds in \cite{ST0}.
\end{remark}

The definition of $\gobs{K}$ shares some of the nice features of the geometric definition (\ref{eq: vk cohain}) of the van Kampen obstruction. Specifically, we will now describe the higher order analogue (``stabilization'') of the finger move homotopy in Figure \ref{fig:FingerMove} and of Lemma \ref{trivial vk}.

\begin{definition}[\sl Stabilization] \label{def:stabilization}
This operation applies to any two $2$-cells $\s_1, \s_2$ and a $1$-cell $\nu$ of $K$ which are all pairwise non-adjacent, Figure \ref{fig:Whitney3}a. 
Perform a finger move introducing two canceling $\s_1$-$\s_2$ intersections, and let $W'_{12}$ denote the resulting embedded Whitney disk pairing these two intersection, Figure \ref{fig:Whitney3}b. Also consider $S^2_{\nu}$, a small $2$-sphere linking $f(\nu)$ in $\R^4$. The final modification applies to the Whitney disk: $\overline{W}_{12}$ is formed as a connected sum of $W_{12}$ and $S^2_{\nu}$, Figure \ref{fig:Whitney3}c.
\end{definition}

\begin{figure}[ht]
\centering
\includegraphics[height=3.7cm]{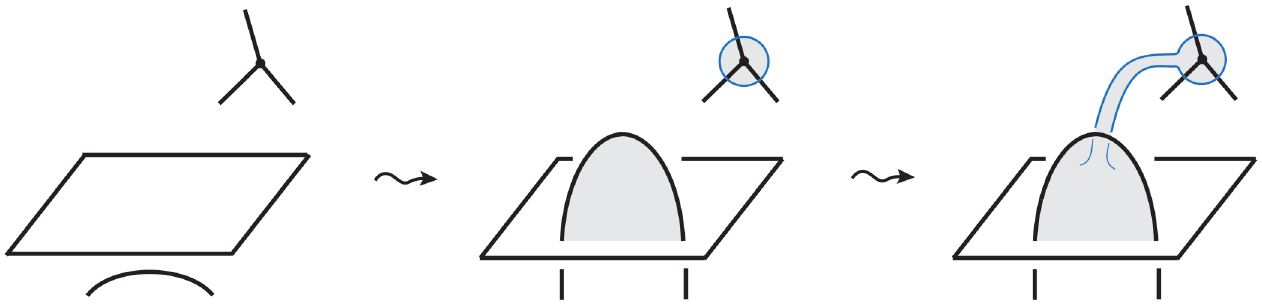}
{\scriptsize     
    \put(-356,2){$f({\sigma}^{}_1) $}
       \put(-363,41){$f({\sigma}^{}_2) $}
       \put(-192,2){$\overline{f}({\sigma}^{}_1) $}
       \put(-199,39){$\overline{f}({\sigma}^{}_2) $}
       \put(-342,86){$f(\nu)$}
      \put(-437,-2){(a)}
        \put(-63,30){$\overline{W}_{12}$}
       \put(-165,83){$S^2_{\nu} $}
       \put(-225,30){$W'_{12}$}
       \put(-110,-2){(c)}
       \put(-271,-2){(b)}
    }  
    \caption[caption]{Stabilization (modifying the obstruction cocycle by a coboundary)
    }
\label{fig:Whitney3}
\end{figure}
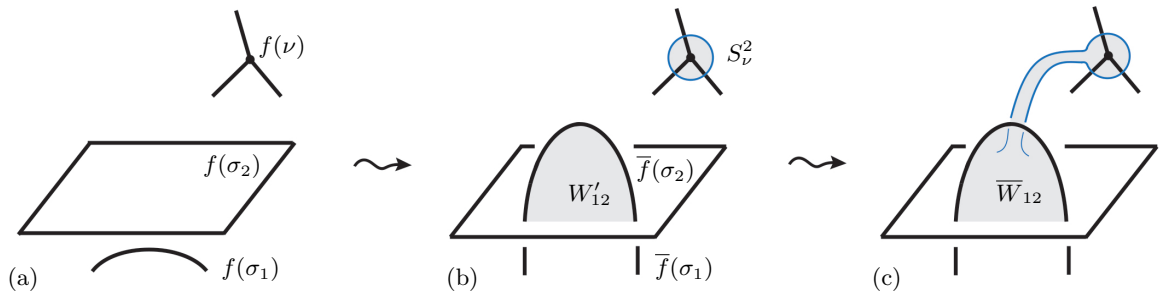

\begin{proposition} 
Let $(\overline{f}, \overline{W})$ be the result of a stabilization applied to $(f,W)$. Then the $\Sigma_2$-equivariant map $F_{\overline{f}, \overline{W}}\co \confs{K}{2}\longrightarrow \conf{\R^4}{2}$ associated to $(\overline{f}, \overline{W})$ in Lemma \ref{Whitney-config} is $\Sigma_2$-equivariantly homotopic to $F_{f,W}$.
\end{proposition}

{\em Proof.} The Whitney disk $\overline{W}_{12}$ is used only in the restriction of the map $F_{\overline{f}, \overline{W}}$ to $\s_1\times \s_2$ (and equivariantly to $\s_2\times \s_1$). 
When $f(\nu)$ and all $2$-cells adjacent to it are omitted from the picture, the Whitney disks $\overline{W}_{12}, W'_{12}$ in Figure \ref{fig:Whitney3} are isotopic. Thus it is clear from the proof of Lemma \ref{Whitney-config} that the maps of configurations spaces corresponding to these two Whitney disks are homotopic. (Note that the interior of $\overline{W}_{12}$ is disjoint from $f(\s_1)$ since $\s_1, \nu$ were assumed to be non-adjacent. Thus the result of the Whitney move on $\overline{f}(\s_2)$ along $\overline{W}_{12}$ is disjoint from $f(\s_1)$.) Moreover, the map $f$ in Figure  \ref{fig:Whitney3}a is isotopic to the result of the Whitney move applied to $\overline{f}$ in Figure  \ref{fig:Whitney3}b, so the induced maps on configuration spaces are again homotopic.
\qed

We are in a position to formulate the analogue of Lemma \ref{trivial vk} for the new obstruction.

\begin{lemma} \label{trivial vk3} 
Any cocycle representative of the cohomology class $$\gobs{K,f, W} \in H^{6}_{\Sigma_3}(\confs{S}{3}; \Z[-1])$$ may be realized as the cocycle $w_3(K,f',W')$ associated to some map $f'$ and Whitney disks $W'$.
In particular, if the cohomology class $\gobs{K,f,W}$ is trivial then there exist $f', W'$ whose associated cocycle is identically zero.
\end{lemma}

{\em Proof.} Consider a generator $C_{\s_1,\s_2,\nu}$ of $\Sigma_3$-equivariant $5$-cochains on $\confs{S}{3}$, corresponding to non-adjacent $2$-cells  $\s_1, \s_2$ and $1$-cell $\nu$ of $K$. 
The stabilization operation $(f,W)\mapsto(\overline{f},\overline{W})$, shown in Figure \ref{fig:Whitney3}, changes the cocycle $w_3(K,f,W)$ by a coboundary $\pm \delta C_{\s_1,\s_2,\nu}$, where the sign depends on the orientation of the sphere $S^2_{\nu}$. Thus changing $w_3(K,f,W)$ by any coboundary may be realized by a suitable sequence of stabilizations.
\qed

As we explain in the next subsection, the vanishing of the cohomology class $\gobs{K,f, W}$ has a geometric consequence: the existence of another layer of Whitney disks, in turn leading to a higher order obstruction.

\subsection{Higher order obstructions from Whitney towers} \label{sec: Whitney towers}
The notion of Whitney towers encodes higher order intersections of surfaces in $4$-manifolds, where the vanishing of the intersections inductively enables one to find the next layer of Whitney disks.  In a sense Whitney towers approximate an embedded disk as the number of layers increases. A closely related notion of capped gropes \cite[Chapter 2]{FQ} is extensively used in the theory of topological $4$-manifolds: they may be found in the context of surgery and of the $s$-cobordism conjecture where surfaces have duals, cf. Proof of Theorem 5.1A in \cite{FQ}. We will use the notion of Whitney towers and their intersection theory developed in \cite{ST, ST1}.  Only a brief summary of the relevant definitions is given below; the reader is referred to the above references for details.

In the setting of this paper the ambient $4$-manifold is $\R^4$, and the surfaces are the images of non-adjacent $2$-cells of a $2$-complex $K$ under a general position map $f\co K\longrightarrow \R^4$. Moreover, we will use the {\em non-repeating} version of Whitney towers considered in \cite{ST1}.

Whitney towers have a parameter, {\em order}, and are defined inductively.  Whitney towers of order $0$ are just surfaces in general position in a $4$-manifold.  Their intersection numbers may be used to define the van Kampen obstruction, as discussed in Section \ref{sec:van Kampen}.
A Whitney tower of order $1$ is a collection of surfaces with trivial intersection numbers, together with a collection of Whitney disks pairing up the intersection points. (As in the preceding sections, all Whitney disks are assumed to be framed, and have disjoint boundaries.) This is the setting for the obstruction in Definition \ref{def: W3}. Note that the Whitney tower incorporates both the map $f$ and the Whitney disks $W$, so $\gobs{K, f, W}$ may be thought of as being defined in terms of a Whitney tower.

All surface stages and intersection points between them in a general Whitney tower are inductively assigned an order in $\Z_{\geq 0}$ as follows.
The base of the construction ({\em order $0$}) is a collection of the original immersed surfaces in $\R^4$.
All surfaces of higher order are Whitney disks pairing up intersections of surfaces of lower order.
The order of an intersection point of surfaces of orders $n_1, n_2$ is defined to be $n_1+n_2$. A Whitney disk pairing up intersection points of order $n$ is said to have order $n+1$. 

Finally, a Whitney tower $W$ of order $n+1$ is defined inductively as a Whitney tower of order $n$ together with a collection of Whitney disks pairing up all intersections of order $n$.
For example, a tower of order $2$ is illustrated on the left in Figure \ref{fig:WhitneyTower}, with the surfaces $\s$ of order $0$ and Whitney disks $V$ of order $1$ and $W$ of order $2$. 

\begin{definition} \label{f admits Whitney tower}
A map $f\co K\longrightarrow \R^4$ {\em admits a Whitney tower} of order $n$ if this condition holds for the images under $f$ of each $n$-tuple of pairwise non-adjacent $2$-cells. 
\end{definition}

We would like to emphasize that in general Definition \ref{f admits Whitney tower} refers not a single Whitney tower, but rather there is a Whitney tower of height $n$ for each $n$-tuple of pairwise non-adjacent $2$-cells.
Note that given a $2$-complex $K$, an obstruction to the existence of a map $f$ admitting a Whitney tower of order $n$ for any $n\geq 1$ is in particular an obstruction to the existence of an embedding $K\hookrightarrow \R^4$.

\begin{figure}[ht]
\centering
\includegraphics[height=3cm]{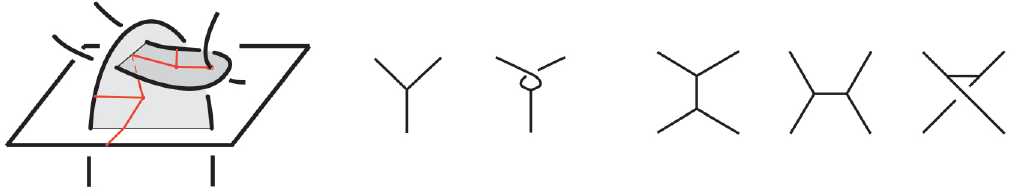}
{\small
\put(-435,27){$\s_j$}
\put(-425,3){$\s_i$}
\put(-435,74){$\s_k$}
\put(-353,82){$\s_l$}
}
{\scriptsize
\put(-380,32){$V$}
\put(-370,48){$W$}
}
{\large 
\put(-250,41){$+$}
\put(-201,41){$= 0 =$}
\put(-122,41){$-$}
\put(-60,41){$+$}
}
\caption{Left: a Whitney tower of order $2$ and the associated tree.  Right: the AS relation and the IHX relation}
\label{fig:WhitneyTower}
\end{figure}
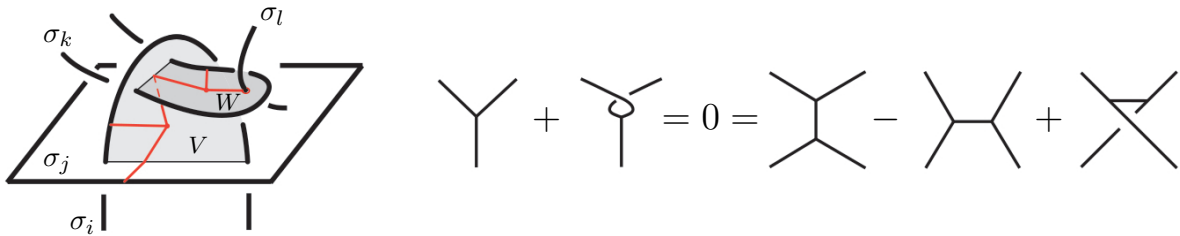

With this terminology at hand, we are ready to formulate a geometric consequence of Lemma \ref{trivial vk3}.

\begin{corollary} \label{cor: height2}
Let $f\co K\longrightarrow \R^4$ be a map 
admitting a Whitney tower of order $1$. (In other words, $f$ is an immersion with double points paired up with Whitney disks $W$, as in Section \ref{sec:higher geometric}.) Suppose the cohomology class $$\gobs{K,f,W} \in H^{6}_{\Sigma_3}(\confs{S}{3};\Z[-1])$$ is trivial. Then there exists a map $\widetilde{f}\co K\longrightarrow \R^4$, obtained from $f$ by stabilizations, which admits a Whitney tower of order $2$.
\end{corollary}

Indeed, by Lemma \ref{trivial vk3} there exists a map $f'$ and Whitney disks $W'$ such that for each triple of (pairwise non-adjacent) $2$-cells, the intersection invariant (\ref{eq: 02 def}) is trivial. By \cite[Theorem 2]{ST}, the map $\widetilde{f}$ is regularly homotopic to $f'$ which admits a Whitney tower of order $2$, as claimed. 

It follows from Lemma \ref{trivial vk} that if $K$ has trivial van Kampen's obstruction, there exists a map of $K$ into $\R^4$ which admits a Whitney tower of height $1$. Corollary \ref{cor: height2} gives the analogue for the next obstruction: if the class $\gobs{K,f,W}=0$, there exists a map admitting a Whitney tower of height $2$. To define higher obstruction theory, we will now discuss the intersection invariants of Whitney towers.

The obstruction cochain in equation (\ref{eq: 02 def}) was defined using an explicit formula with intersection numbers between Whitney disks and $2$-cells. An elegant way of formulating the intersection invariant \cite{ST} for a general Whitney tower is in terms of trees, described next.

Each unpaired intersection point $p$ of a Whitney tower determines a trivalent tree $t_p$: the trivalent vertices correspond to Whitney disks and the leaves are labeled by (distinct) $2$-cells of $K$. The tree embeds in the Whitney tower, as shown on the left in Figure \ref{fig:WhitneyTower}, and it inherits a cyclic orientation of each trivalent vertex from this embedding. (Recall that Whitney disks are oriented as in Definition \ref{def: W3}.)

The relevant obstruction group in our context will be denoted $\mathcal{T}_n$.  It is defined as a quotient of the free abelian group generated by trivalent trees with $n+2$ leaves (and thus $n$ trivalent vertices). The leaves are labeled by non-repeating labels $\{ 1,\ldots, n+2\}$, and the trivalent vertices are cyclically oriented. The quotient is taken with respect to the AS and IHX relations, shown on the right in Figure \ref{fig:WhitneyTower}. These relations are well-known in the study of finite type invariants; in the context of Whitney towers the AS (anti-symmetry) relation corresponds to switching orientations of Whitney disks, and the IHX relation reflects choices of Whitney arcs, see \cite{CST}.

Following \cite[Section 2.1]{ST}, the {\em intersection tree} ${\tau}_n$ of an order $n$ Whitney tower $W$ is defined to be 
\begin{equation} \label{def: intersection inv}
{\tau}_n(W):= \sum_p {\epsilon}(p)\,  t_p\; \in {\mathcal T}_n,
\end{equation}
where the sum is taken over all unpaired (order $n$) intersections points $p$, and ${\epsilon}(p)$ is the sign of the intersection. For example, for order $1$ Whitney tower the intersection trees are the $Y$ tree with two possible cyclic orderings of the trivalent vertex; the obstruction group ${\mathcal T}_1$ is isomorphic to $\Z$, and the intersection invariant matches the formula (\ref{eq: 02 def}).

Let $\confs{K}{n}$ denote $K^{\times n}$ minus the simplicial diagonal consisting of all products of simplices ${\sigma}_1\times\ldots\times {\sigma}_n$, where at least two of the simplices ${\sigma}_i$, ${\sigma}_j$ have a vertex in common for some $i\neq j$. 
The symmetric group ${\Sigma}_n$ acts in a natural way on the configuration space $\confs{K}{n}$ and also on ${\mathcal T}_{n-2}$. The following definition extends Definition \ref{def: W3} to all $n\geq 3$.

\begin{definition}[\sl The obstruction $\gobsn{K}$] \label{def: Wn} Let $n\geq 3$ and suppose a map $f\co K\longrightarrow \R^4$ admits a Whitney tower $W$ of order $n-2$. Endow the $2$-cells of $K$ with arbitrary orientations; orientations of all Whitney disks in $W$ are then determined as in Definition \ref{def: W3}. 
Consider the $\Sigma_n$-equivariant $2n$-cochain:
\begin{equation} \label{eq: on cocohain} 
w_n\co C_{2n}(\confs{K}{n})\longrightarrow {\mathcal T}_{n-2},
\end{equation}
whose value on the $2n$-cell $\s_1\times\ldots \times\s_n$ is given by the intersection invariant (\ref{def: intersection inv}) of the Whitney tower on the $2$-cells
$f(\s_1), \ldots, f(\s_n)$.
It is a cocycle since it is a top-dimensional cochain on $\confs{K}{2n}$. The resulting cohomology class is denoted $$\gobsn{K, W}\in H^{2n}_{\Sigma_n}(\confs{K}{n}; {\mathcal T}_{n-2}).$$
\end{definition}

Thus $\gobsn{K, W}$ is an obstruction to increasing the order of a given Whitney tower $W$ to $n-1$; in particular it is an obstruction to using the data of the Whitney tower $W$ to find an  embedding of $K$.

\begin{remark} \label{rem: coefficients}
Note that ${\mathcal T}_{n-2}$ is isomorphic to $\Z^{(n-2)!}$, cf. \cite[Lemma 19]{ST1}; compare this with the coefficients of the cohomology group in Theorem \ref{thm: cohom obs}. 
\end{remark}

We note that there is an analogue of stabilization in Definition \ref{def:stabilization} for higher trees generating ${\mathcal T}_n$, and an analogue of Corollary \ref{cor: height2} for higher obstructions ${\mathcal W}_n$. Thus there is an obstruction theory for $2$-complexes in $\R^4$ formulated entirely within the context of intersections of Whitney towers. As we mentioned previously, this paper is centered around the first new obstruction, ${\mathcal W}_3$; we plan to study higher obstructions in more detail in a future work. $\obs{K}$ and $\gobs{K}$ are related in the next section; a conjectural relation between $\obsn{K}$ and $\gobsn{K}$ for $n>3$ is stated in Section \ref{sec: Questions}.

\section{Comparing the cohomological and geometric obstructions} \label{sec:equality}
Here we will relate the obstruction $\obs{K}$ defined in Section \ref{sec: Constructing the obstruction} and $\gobs{K}$ from Section \ref{sec:Whitney towers}; the main result of this section is Theorem \ref{thm: obstructions coincide}. 
Before we state the result, a brief digression is needed to compare the settings of the two obstructions.\footnote{The second author would like to thank Pedro Boavida de Brito for motivating questions.}  As discussed in Section \ref{sec:van Kampen}, the two versions of the van Kampen obstruction are related by $\svk{K}=i^*\vk{K}$, where $i$ is the inclusion map $\confs{K}{2}\subset \conf{K}{2}$. 
The assumption in the theorem below is that $\vk{K}$ is trivial; it follows that $\svk{K}$ vanishes as well, and therefore there exists a map $f\co K\longrightarrow {\mathbb R}^4$ and a collection of Whitney disks for intersections of non-adjacent simplices.
Then Lemma \ref{Whitney-config} gives a $\Sigma_2$-equivariant map $F_{f,W}\co \confs{K}{2}\longrightarrow \conf{R^4}{2}$. However the starting point for the obstruction $\obs{K}$ is a $\Sigma_2$-equivariant map $\conf{K}{2}\longrightarrow \conf{R^4}{2}$. To relate the two contexts, for a given  simplicial $2$-complex we will take a subdivision fine enough to ensure that the inclusion $\confs{K}{2}\hookrightarrow\conf{K}{2}$ is a homotopy equivalence. Then $F_{f,W}$ induces a map (well defined up to equivariant homotopy) $\conf{K}{2}\longrightarrow \conf{R^4}{2}$, which is needed to define $\obs{K}$.\footnote{There are also other ways of relating the two settings; for example one may define a ``simplicial'' version of $\obs{K}$ as the homotopy-lifting obstruction in (\ref{Delta3prime}) where $\confs{K}{3}\longrightarrow \confs{K}{2}^{\times 3}$ is used instead.}

Without loss of generality we will assume that the Whitney disks are split as discussed in Section \ref{Whitney subsection}. 

\begin{theorem} \label{thm: obstructions coincide}
Given a $2$-complex $K$ with trivial van Kampen's obstruction $\vk{K}$, let $W$ be a collection of split Whitney disks for double points of a map $f\co K\longrightarrow {\mathbb R}^4$.  Let $F_{f,W}\co \confs{K}{2}\longrightarrow \conf{R^4}{2}$ be the $\Sigma_2$-equivariant map
determined by $f, W$ in Lemma \ref{Whitney-config}.
Then 
\begin{equation} \label{eq: equal}
\gobs{K,f,W}=i^* \obs{K} \in H^{6}_{\Sigma_3}(\confs{K}{3}; \Z[(-1)]),
\end{equation}
where $i\co \confs{K}{3}\longrightarrow \conf{K}{3}$ is the inclusion map.
\end{theorem}

{\em Proof.} For convenience of the reader, the proof of Theorem \ref{eq: equal} is divided into steps.

{\bf Step 1: subdivision}.
The pullback $i^* \obs{K,F_{f,W}}$ is  the obstruction to the existence of a $\Sigma_3$-equivariant dashed map 
making the following diagram commute up to homotopy.
\begin{equation} \label{eq: C K 3 2}
\begin{tikzcd}[column sep = large]
\confs{K}{3} \arrow[r, dashed] \arrow{d}{p_K} & \conf{ \R^4}{3} \arrow{d}{p_{\R^4}} \\
\confs{K}{2}^3 \arrow{r}{(F_{f,W})^3}& \conf{\R^4}{2}^3
\end{tikzcd}
\end{equation}
The first step of the proof is to use subdivision to reduce to a model situation where precisely one of the following holds for the image under $f$ of each $2$-cell $\s$ of $K$:

\begin{enumerate}
    \item{$\s$ is mapped in disjointly from all other non-adjacent $2$-cells,}
\item{$\s$ intersects exactly one other non-adjacent $2$-cell in two points, or}
\item{$\s$ has a single intersection point with one of the Whitney disks.}
\end{enumerate}

(Moreover, the Whitney disks are already assumed to be split, so each one intersects at most one $2$-cell as in Figure \ref{Whitney}.)
To begin with, each $2$-cell $\s$ of $K$ has a finite number of disjoint Whitney arcs, as shown in Figure \ref{fig:isotopy}, and a finite number of intersection points with Whitney disks. The conditions (1)-(3) above are achieved by subdividing so that each $2$-cell contains at most one Whitney arc or intersection point with a Whitney disk. For each pair on intersections of $2$-cells $\s_i, \s_j$ as in case (3) we will choose a particular ordering of $i,j$ that will determine which sheet is pushed by the Whitney move.

Let $K'$ denote the $2$-complex obtained as the result of the subdivision and let $f'\co K'\longrightarrow \R^4$ be the resulting map. The map $F_{f,W}$ in Lemma \ref{Whitney-config} was defined by local modifications of $f$ in disk neighborhoods of the Whitney arcs; $F_{f',W'}$ may be assumed to be defined with respect to the same disk neighborhoods (which are now located in distinct $2$-cells of $K'$). It follows that $F_{f,W}$ is the composition 
$$ \confs{K}{2}\longrightarrow \confs{K'}{2}\longrightarrow \conf{R^4}{2}
$$
of the inclusion and $F_{f',W'}$. Moreover, the cochain (\ref{eq: o3 cocohain}) defining $\gobs{K}$ is natural with respect to subdivisions, so $\gobs{K}$ is the pullback of $\gobs{K'}$ under the inclusion $\confs{K}{3}\longrightarrow \confs{K'}{3}$.
Thus it suffices to prove Theorem \ref{thm: obstructions coincide} for $K'$. For the rest of the proof we will revert to the notation $K$ for the $2$-complex, assuming it is subdivided to satisfy conditions (1)-(3).

{\bf Step 2: a lift on the ${\mathbf 5}$-skeleton}.
Since the homotopy fiber of the map $p_{\R^4}\co \conf{\R^4}{3}\longrightarrow \conf{\R^4}{2}^3$ is $4$-connected,
there is a lift in (\ref{eq: C K 3 2}) on the $5$-skeleton ${\rm Sk}^5\confs{K}{3}$.

\begin{construction} \label{5-skeleton}  The construction described below defines a particular $\Sigma_3$-equivariant map of the $5$-skeleton, $F\co {\rm Sk}^5\confs{K}{3}\longrightarrow \conf{\R^4}{3}$, lifting up to homotopy the $\Sigma_3$-equivariant map 
${\rm Sk}^5\confs{K}{3} \longrightarrow \conf{\R^4}{2}^3$. Its specific geometric form will be used for identifying the point preimages of the map to $S^3\vee S^3$ in diagram (\ref{fibration diag1}). The construction relies on the capped surface description of the Whitney move (Figure \ref{Whitney}), and is an extension of Lemma \ref{Whitney-config}.
\end{construction}

Consider the map on the $4$-skeleton induced by $f$: given any pairwise non adjacent $2$-cell $\s$ and $1$-cells $\nu$, $\tau$, by general position $f(\s), f(\nu)$ and $f(\tau)$ are pairwise disjoint; $F$ is defined on $\s\times\nu\times \tau$ (and its orbit under the $\Sigma_3$ action) by the Cartesian product $f^{\times 3}$. 

The main part of the construction concerns the extension of this map to the $5$-cells. 
We will define $F$ on the boundary of each $6$-cell $\partial(\s_1\times\s_2\times\s_3)$, where $\s_i$, $i=1,2,3$ are $2$-cells of $K$, so that the definition is consistent on the overlap of the boundaries of $6$-cells. The map will be defined for a particular ordering $\s_1,\s_2, \s_3$ and extended to triple products corresponding to other orderings using $\Sigma_3$ equivariance.

There are three cases:
\begin{enumerate}[label=(\roman*)]
    \item{ the images of $\s_i, i=1,2,3,$ are pairwise disjoint,}
    \item{two of them, say $\s_1, \s_3$ intersect, and $W_{13}\cap \s_2=\emptyset$,}
    \item{two of them, say $\s_1, \s_3$ intersect, and $W_{13}\cap \s_2$ is a point.}
\end{enumerate}

In case (i) the map $F$ is defined on $\partial(\s_1\times\s_2\times\s_3)$ as the Cartesian cube $f^{\times 3}$. Consider case (ii). The boundary of the product  $\partial ({\sigma}_1\times {\sigma}_2\times {\sigma}_3)$ naturally decomposes as the union of three parts. The definition of $F$ on two of the parts is again $f^{\times 3}$. 
The definition of $F$ on $\s_1\times \partial \s_2\times \s_3$  is an analogue of the proof of Lemma \ref{Whitney-config}. It is defined on $D_1\times\partial \s_2\times \s_3$
as $f\times f\times  \widetilde f$, where $\widetilde f$ is the result of the Whitney move on $\s_3$, and $D_1$ is a disk neighborhood of the Whitney arc in $\s_1$. As in the proof of that lemma an isotopy in a collar $C$ on the boundary of $D_1$ is used, so that on $\partial \s_1\times \partial \s_2\times \s_3$ the map $F$ equals $f^{\times 3}$.

Now consider the most interesting case (iii), shown in Figure \ref{Whitney}. As in the previous case consider a smaller disk neighborhood $D_1$ of the Whitney arc in $\s_1$.
We will work in the $4$-ball neighborhood of the Whitney disk $W_{13}$; the intersection of ${\sigma}_i, i=1,2,3$ with $\partial D^4$ forms the Borromean rings, illustrated in Figure \ref{fig:Bor}.  
The disk $\s_3$ may be converted into a punctured torus as in Figure \ref{Whitney}.

It will be convenient to represent disks in $D^4$ as movies in $D^3\times [-1,1]$ with time $-1\leq t\leq 1$, where most of the activity takes place at time $t=0$. The remaining figures in this section illustrate $D^3\times \{ 0\}$. Figure \ref{Whitney1} shows the capped torus (referred to above) bounded by $\partial \s_3$ in this representation. The punctured torus consists of two plumbed bands, with caps $C'$ (intersecting $\s_2$) and $C''$ (intersecting $D_1$). The intersections of $D_1$ and $\s_2$ with the slice $D^3\times \{0\}$ are arcs; they extend as (arc$\times I$) into the past and the future. 

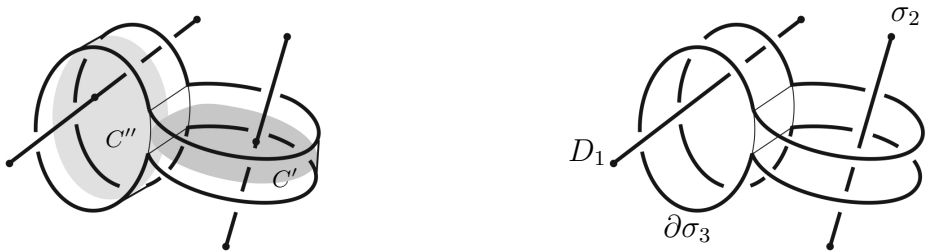
\begin{figure}[ht]
\hspace*{-.7cm}\includegraphics[height=3.2cm]{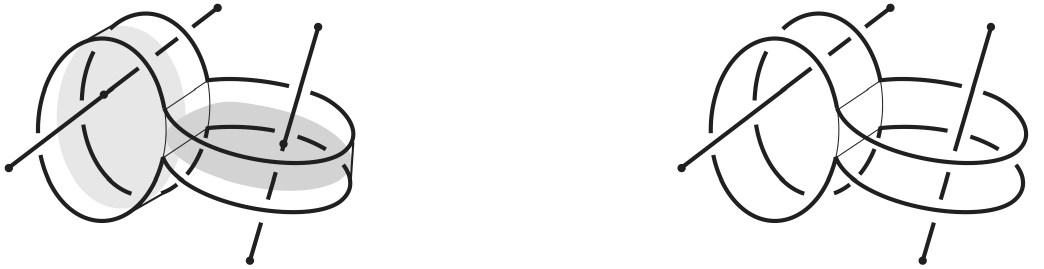}
{    
    \put(-18,88){${\sigma}_2 $}
       \put(-140,35){$D_1 $}
       \put(-104,6){$\partial {\sigma}_3$}
       }
       {\scriptsize
       \put(-315,40){$C''$}
       \put(-252,25){$C'$}
    }  
    \caption{Left: the capped torus bounded by $\partial \s_3$ with caps $C', C''$.  Right: the map $f$ defining $F_{12}$.}
\label{Whitney1}
\end{figure}

The disks bounded by $\s_3$ in Figures \ref{Whitney2}, \ref{Whitney3} are the surgeries along the two caps and the symmetric surgery, and they will be entirely in the present. The original map $f$ is recovered by the surgery along the cap $C''$ (Figure \ref{Whitney2}, left), and the result of the Whitney move $\widetilde f$ is the surgery on $C'$ (Figure \ref{Whitney2}, right).

We will now proceed to define $F$ on the three parts of the boundary $\partial ({D}_1\times {\sigma}_2\times {\sigma}_3)$.
The map $F_{12}\co{D}_1\times {\sigma}_2\times \partial {\sigma}_3 \longrightarrow \conf{\R^4}{3}$ is defined as the Cartesian product $f^{\times 3}$ where $f$ is the original map $K\longrightarrow \R^4$; it is an embedding when restricted to $ {D}_1\coprod {\sigma}_2\coprod \partial {\sigma}_3\hookrightarrow {\mathbb R}^4$, Figure \ref{Whitney1} (right).

The maps $F_{23}\co \partial {D}_1\times {\sigma}_2\times {\sigma}_3 \longrightarrow \conf{\R^4}{3}$, $F_{13}\co {D}_1\times \partial{\sigma}_2\times  {\sigma}_3 \longrightarrow \conf{\R^4}{3}$ are defined respectively as $f^{\times 3}$, $(\widetilde f)^{\times 3}=f\times f\times \widetilde f$ where $f$ is again the original map which restricts to an embedding $f\co \partial{D}_1\coprod {\sigma}_2\coprod {\sigma}_3\hookrightarrow {\mathbb R}^4$, and $\widetilde f \co {D}_1\coprod \partial {\sigma}_2\coprod {\sigma}_3\hookrightarrow {\mathbb R}^4$ is the result of the Whitney move on $\s_3$, Figure \ref{Whitney2}.

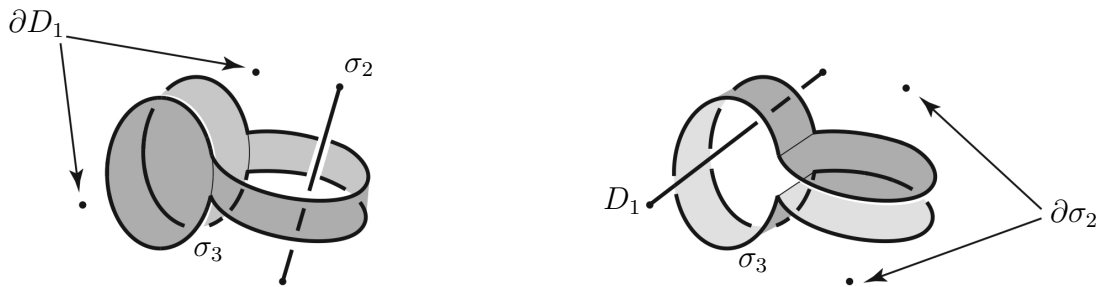
\begin{figure}[ht]
\centering
\includegraphics[height=3.5cm]{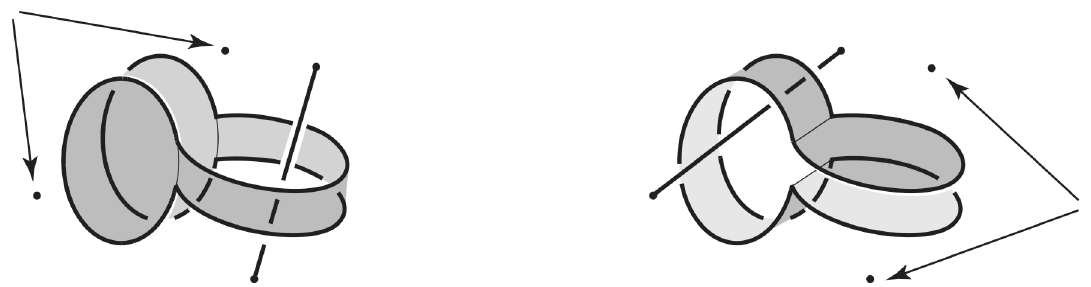}
{    
    \put(-1,25){${\partial {\sigma}_2} $}
       \put(-170,30){$D_1 $}
       \put(-119,8){$ {\sigma}_3$}
        \put(-268,82){${ {\sigma}_2} $}
       \put(-395,97){$\partial D_1 $}
       \put(-325,12){$ {\sigma}_3$}
    }  
    \caption{The map $f$ defining $F_{23}$ (left) and $\widetilde f$ defining $F_{13}$ (right).}
\label{Whitney2}
\end{figure}

The only part of the definition where the map differs from $f^{\times 3}$ is ${D}_1\times \partial{\sigma}_2\times  {\sigma}_3$, where $F$ is defined as $(\widetilde f)^{\times 3}=f\times f\times \widetilde f$.
As in case (ii) and in the proof of Lemma \ref{Whitney-config}, consider a collar $C$ on $\partial D_1$ in $\s_1$ and extend  $F$ to ${C}\times \partial{\sigma}_2\times  {\sigma}_3$ using an isotopy from $\widetilde f$ to $f$. The half point of the isotopy, the symmetric surgery discussed above, is shown in Figure \ref{Whitney3}. Finally, the map is set to be $f^{\times 3}$ on $(\s_1\smallsetminus({C}\cup D_1))\times \partial{\sigma}_2\times  {\sigma}_3$.

\begin{figure}[ht]
\centering
\includegraphics[height=3.5cm]{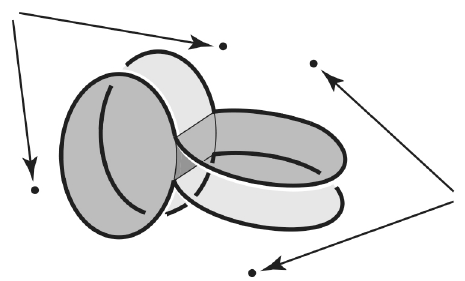}
{
    \put(-2,25){$\partial {\sigma}_2 $}
       \put(-178,98){$\partial{\sigma}_1 $}
       \put(-112,10){${\sigma}_3$}
    }  
    \caption{The symmetric surgery on the capped torus.}
\label{Whitney3}
\end{figure}

The map $F$ is well-defined on the $5$-skeleton: consider an overlap $\partial \s_2\cap \partial \s_2'$, where $\s_2$ intersects $W_{13}$ as in case (iii) and $\s_2'$ is disjoint from $W_{13}$, as in case (ii). The definition in the two cases above assigns the same map to $\s_1\times (\partial \s_2\cap \partial \s_2')\times \partial \s_3$.

The constructed map $F\co {\rm Sk}^5\confs{K}{3}\longrightarrow \conf{\R^4}{3}$ lifts
${\rm Sk}^5\confs{K}{3} \longrightarrow \conf{\R^4}{2}^3$ up to homotopy because the surgeries on the two caps, defining $F$, are isotopic.
This concludes the description of the map $F$ in Construction \ref{5-skeleton}.

{\bf Step 3: comparing obstructions on the cochain level}. 
In the remainder of the proof of Theorem \ref{thm: obstructions coincide} we will show that the cohomology classes $\gobs{K}$, $i^*\obs{K}$ coincide on the cochain level. 
The value of the cocycle $w_3$ in (\ref{eq: 02 def}) is zero on the $6$-cell $D^6:=\s_1\times \s_2\times \s_3$ in cases (i), (ii) above, and it equals $\pm 1$ in case (iii).
Recall that $i^*\obs{K}$ is the $6$-dimensional cohomological obstruction to lifting the map $\confs{K}{3}\to \conf{\R^4}{2}^3$
in the diagram (\ref{eq: C K 3 2}) to a map $\confs{K}{3}\to \conf{\R^4}{3}$. We now recall the skeletal construction of the obstruction, which we mentioned in Section~\ref{sec: Constructing the obstruction}. According to the skeletal approach, a choice of a lift $F$ defined on the $5$-dimensional skeleton of $\confs{K}{3}$ determines a cochain representative of $i^*\obs{K}$. A change of choice alters the cochain by a coboundary.  The value of the obstruction cochain on the $6$-cell $D^6$ is the element represented by $F(\partial D^6)$ in $\pi_5$ of the homotopy fiber of the map $p_{\R^4}\co \conf{\R^4}{3}\longrightarrow \conf{\R^4}{2}^3$; we will focus on the non-trivial case (iii) to match it with the value of $w_3$.

As we did in Section~\ref{sec: Constructing the obstruction}, let us denote the fiber of the map $p_{\R_4}\colon\conf{\R^4}{3}\to \conf{R^4}{2}^3$ by $F_4$. We saw earlier that $F_4$ is $4$-connected, and $\pi_5(F_4)\cong\Z$ (see Corollary~\ref{corollary: homology} and the remark immediately following it). Now we need to identify the generator of $\pi_5(F_4)$ as a Whitehead product.

Choose two points $x_1, x_2\in \R^4$. This choice determines an embedding, which is also a homotopy equivalence $$S^3\vee S^3\stackrel{\simeq}{\hookrightarrow} \R^4\setminus\{x_1, x_2\}$$ as a wedge sum of two round spheres whose centers are $x_1$ and $x_2$. Furthermore, we have an embedding $$\R^4\setminus \{x_1, x_2\} \hookrightarrow \conf{\R^4}{3},$$ which sends $x$ to $(x_1, x_2, x)$. 

Let $\alpha, \beta$ be the two standard generators of $\pi_3(S^3\vee S^3)$. We also denote the images of $\alpha$ and $\beta$ in $\pi_3(\conf{\R^4}{3})$ by the same letters.

\begin{lemma}
The Whitehead product $[\alpha, \beta]$ is mapped to zero by $p_{\R^4}$. It lifts to a generator of $\pi_5(F_4).$
\end{lemma}
\begin{proof}
{
Consider the following diagram:
\[
\begin{tikzcd}
& S^3 \vee S^3 \arrow[d, "\simeq"] & \\
F_4\arrow[r] &\R^4\setminus \{x_1, x_2\} \arrow[r]\arrow[d] & S^3 \times S^3 \arrow[d] \\
& \conf{\R^4}{3} \arrow[d, "p_{12}"]\arrow[r, "p_{13}\times p_{23}"] & \conf{\R^4}{2}\times \conf{\R^4}{2}\arrow[d] \\ 
& \conf{\R^4}{2} \arrow[r] & * 
\end{tikzcd}
\]
In this diagram $p_{ij}$ is the map that sends $(x_1, x_2, x_3)$ to $(x_i, x_j)$. The space $F_4$ is the total homotopy fiber of the bottom square. It is naturally equivalent to the homotopy fiber of the top right horizontal map, which is the map between vertical fibers of the bottom square. The top right horizontal map is in turn naturally equivalent to the inclusion $S^3\vee S^3 \hookrightarrow S^3\times S^3$, so we have an identification of $F_4$ with the homotopy fiber of this inclusion. By classical homotopy theory, the first non-trivial homotopy group of the homotopy fiber of this inclusion is $\pi_5$, it is isomorphic to $\Z$, and it is generated by the Whitehead product $[\alpha, \beta]$ (see~\cite[Theorem XI.1.7]{Whitehead}).
}
\end{proof}
Now we continue the analysis of $F(\partial D^6)$. The value of $F(\partial D^6)$ will be determined as follows.
As above, consider the fibration $p_{12}\co \conf{\R^4}{3}\longrightarrow \conf{\R^4}{2}$, $p_{12}(x_1, x_2, x_3)=(x_1, x_2)$:
\begin{equation} \label{fibration diag1}
\centering
\begin{tikzcd}
& {\mathbb R}^4\smallsetminus 2\, {\rm points} \arrow[r, "\simeq"] \arrow[d] & S_{13}^3\vee S_{23}^3 \\
\partial ({\sigma}_1\times {\sigma}_2\times {\sigma}_3) \arrow[r, "F"] & \conf{\R^4}{3}\arrow[d, "p_{12}"]  \arrow[dr, bend right=10, "\bar p_{12}"] & \\
& \conf{\R^4}{2}\arrow[r,"\simeq"] & S_{12}^3
\end{tikzcd}
\end{equation}

The composition $p_{12}\,\circ F$  is null-homotopic, where the map $F\co \partial ({\sigma}_1\times {\sigma}_2\times {\sigma}_3) \longrightarrow \conf{\R^4}{3}$ is the result of Construction \ref{5-skeleton}. In fact, it is clear from Figure \ref{Whitney1} that $\bar p_{12}\circ F$ is not surjective: its image is contained in a ball $D^3\subset S^3$. Trivializing the fibration over $D^3$, the map $F$ lifts to the fiber, yielding a map $\widetilde F\co S^5=\partial ({\sigma}_1\times {\sigma}_2\times {\sigma}_3)\longrightarrow S_{13}^3\vee S_{23}^3$. The remainder of the proof of Theorem \ref{thm: obstructions coincide} amounts to checking that the homotopy class of this map in ${\pi}_5(S_{13}^3\vee S_{23}^3)$ represents the Whitehead product of the two wedge summands. 

{\bf Step 4: identifying the homotopy class as the Whitehead product}.
The compositions of the map $\widetilde F$ with the projections of $S_{13}^3\vee S_{23}^3$ onto the wedge summands are homotopic to $\bar p_{13}\circ F$, $\bar p_{23}\circ F$ in the diagram (\ref{fibration diag2}). In both diagrams, the map $\bar p_{ij}\co \conf{\R^4}{3}\longrightarrow S^3_{ij}$ is given by $\bar p_{ij}(x_1, x_2, x_3)=(x_i, x_j)/|x_i-x_j|$, $i\neq j\in\{1,2,3\}$.

\begin{equation} \label{fibration diag2}
\centering
\begin{tikzcd}
& & & S_{13}^3\\
\partial ({\sigma}_1\times {\sigma}_2\times {\sigma}_3) \arrow[r, "F"] & \conf{\R^4}{3}\arrow[d, "\bar p_{12}"']  \arrow[urr, "\bar p_{13}"] \arrow[drr, "\bar p_{23}"']  &  & \\
& S_{12}^3 & & S_{23}^3
\end{tikzcd}
\end{equation}

Using the Pontryagin construction, the homotopy class of $\widetilde F$ in  ${\pi}_5(S_{13}^3\vee S_{23}^3)$ can be determined by the linking number of point preimages of $\bar p_{13}\circ F$, $\bar p_{23}\circ F$. 

\begin{figure}[ht]
\centering
\includegraphics[height=3.5cm]{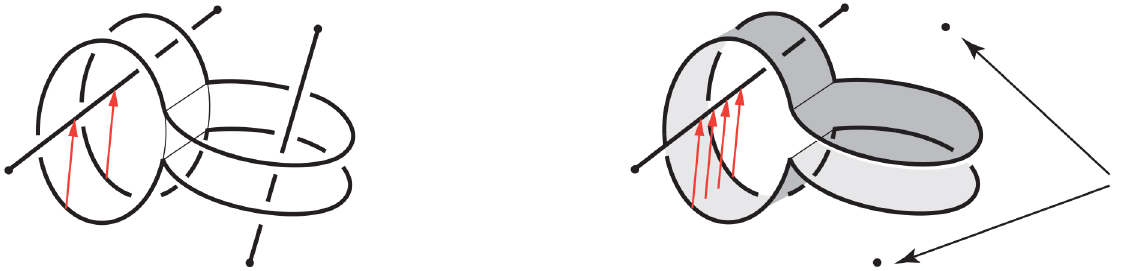}
{    
    \put(-2,28){${\partial {\sigma}_2} $}
       \put(-196,37){${\sigma}_1 $}
       \put(-138,12){$ {\sigma}_3$}
        \put(-292,90){${ {\sigma}_2} $}
       \put(-426,40){${\sigma}_1 $}
       \put(-366,10){$ \partial{\sigma}_3$}
    }  
    \caption{}
\label{Whitney4}
\end{figure}

A transverse point preimage of $\bar p_{13}\circ F$ is shown in Figure \ref{Whitney4}, where a point in $S^3_{13}$ is represented as a vector $v$ in ${\mathbb R}^4$ (colored red online). The preimage of $\bar p_{13}\circ F_{23}$ (defined on the left in Figure \ref{Whitney2}) is empty. The preimage of $\bar p_{13}\circ F_{12}$ is shown on the left of Figure \ref{Whitney4} and consists of two disks $\{ a_0\}\times\s_2\times \{ b_0\}$, $\{ a_1\}\times\s_2\times \{ b_1\}$. Here $a_0, a_1\in \s_1$ and $b_0, b_1\in\partial \s_3$ are the endpoints of the two vectors parallel to $v$ shown in the figure. The preimage of $\bar p_{13}\circ F_{13}$ is shown on the right of Figure \ref{Whitney4} and consists of the annulus $\{ a_t\}\times\partial\s_2\times \{ b_t\}$, $0\leq t\leq 1$. The entire point preimage of $\bar p_{13}\circ F$ is a $2$-sphere assembled of these two disks and the annulus.

Similarly, the point preimage of $\bar p_{23}\circ F$ is analyzed in Figure \ref{Whitney5}. It consists of a $2$-sphere which is the union of two disks $\s_1\times \{ c_0\} \times \{ d_0\}$, $\s_1\times \{ c_1\} \times \{ d_1\}$ and the annulus $\partial \s_1\times \{ c_t\} \times \{ d_t\}$, $0\leq t\leq 1$. Here $c_t\in \s_2$ and $d_t\in\s_3$ for all $t\in [0,1]$, with $d_0, d_1\in \partial \s_3$.

Consider $\s_3$ as a product of two intervals $I_1\times I_2$, and let $D^3_i=\s_i\times I_i$, $i=1, 2$. Note that the two $0$-spheres $(b_0, b_1)$ and $(d_0, d_1)$ link in $\partial \s_3$. Reparametrize $\s_1\times \s_2\times\s_3$ as $D^3_1\times D^3_2$. 
The point preimages of $\bar p_{13}\circ F$, $\bar p_{23}\circ F$ are seen to be the two $2$-spheres $ \{*\} \times\partial D^3_2$,  $\partial D^3_1\times \{*\} \subset \partial (D^3_1\times D^3_2)$.
This concludes the proof of Theorem \ref{thm: obstructions coincide}.
\qed

\begin{figure}[ht]
\centering
\includegraphics[height=4.5cm]{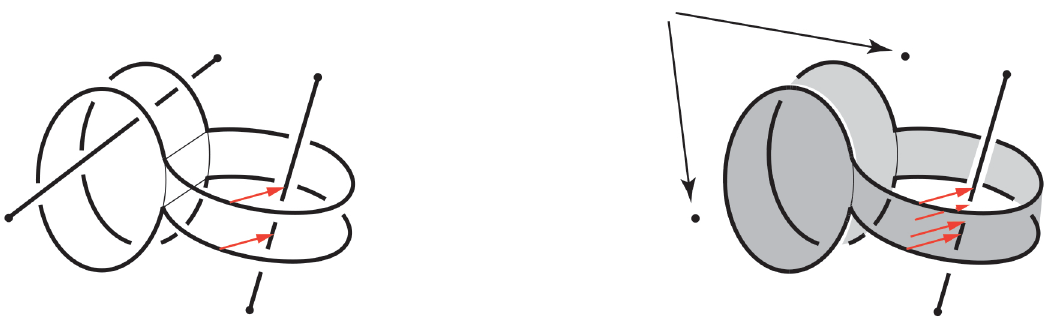}
{    
    \put(-10,95){${ {\sigma}_2} $}
       \put(-170,125){$\partial {\sigma}_1 $}
       \put(-105,10){$ {\sigma}_3$}
         \put(-285,95){${ {\sigma}_2} $}
      \put(-433,40){${\sigma}_1 $}
       \put(-370,10){$ \partial{\sigma}_3$}
    }  
    \caption{}
\label{Whitney5}
\end{figure}

\begin{remark}
Link-homotopy invariants using Whitehead products in configuration spaces were defined and studied in \cite{Koschorke}. The context of the above proof is similar, but the actual method and details of the proof are independent of the results of \cite{Koschorke}.
\end{remark}

\section{A triple collinearity interpretation}\label{sec: construction}
In this section we prove Proposition~\ref{prop: cobordism classifying}. That is, following the notational convention \ref{action convention}, we will construct a $\Sigma_3$-equivariant map $\conf{\R^d}{2}^3\to \widehat\Omega^2 \Omega^\infty \Sigma^\infty \widehat S^{2d}$ that makes the square~\eqref{eq: cobordism classifying square} $3d-5$-cartesian. In fact, we will do something slightly stronger. Namely, we will construct a $\Sigma_3$-equivariant map (recall that ${\widetilde S}^{d-1}$ is $\Sigma_2$-equivariantly equivalent to $\conf{\R^d}{2}$)
\[
f\colon ({\widetilde S}^{d-1})^3\to \widehat \Omega^2 \widehat S^{2d}
\]
such that the following composition is $\Sigma_3$-equivariantly null-homotopic
\[
\conf{\R^3}{3}\to ({\widetilde S}^{d-1})^3\xrightarrow{f} \widehat\Omega^2 \widehat S^{2d}
\]
and moreover the following square diagram is $3d-5$-cartesian. 
\begin{equation}\label{eq: classifying square}
\begin{tikzcd}
\conf{\R^d}{3} \arrow[r] \arrow{d} & ({\widetilde S}^{d-1})^3 \arrow{d}{f} \\
* \arrow{r}& \widehat \Omega^2 \widehat S^{2d}
\end{tikzcd}.
\end{equation}
We call a map $f$ with these properties a {\it classifying map}. Since there is a natural map $$\widehat\Omega^2 \widehat S^{2d} \to \widehat\Omega^2\Omega^\infty \Sigma^\infty \widehat S^{2d}$$ that is $4d-3$-connected, it follows that the square~\eqref{eq: classifying square} is $3d-5$-cartesian if and only if the square~\eqref{eq: cobordism classifying square} is $3d-5$-cartesian. 

The following lemma gives a practical way to verify that a given map is a classifying map. 
\begin{lemma}\label{lemma: cocartesian}
Suppose that we have a $\Sigma_3$-equivariant map
\[
f\colon ({\widetilde S}^{d-1})^3\to \widehat\Omega^2 \widehat S^{2d}
\]
satisfying the following conditions:
\begin{enumerate}
    \item The composite map 
    \begin{equation}\label{eq: classifying sequence}
\conf{\R^d}{3}\to ({\widetilde S}^{d-1})^3\to  \widehat\Omega^2 \widehat S^{2d}
\end{equation}
is equivariantly null-homotopic.
\item $f$ induces an epimorphism on $H_{2d-2}$ (or, equivalently, a monomorphism on $H^{2d-2}$).
\end{enumerate}
Then $f$ is a classifying map.
\end{lemma}
\begin{proof}
By Lemma~\ref{lemma: arnold relation} and Corollary~\ref{corollary: homology}, the homology of the space $
\conf{\R^d}{3}$ is concentrated in degrees $0, d-1, 2(d-1)$. Similarly the homology of $({\widetilde S}^{d-1})^3$ is concentrated in degrees $i(d-1)$, where $i\le 3$. The map $$
\conf{\R^d}{3}\to ({\widetilde S}^{d-1})^3$$ induces an isomorphism on $H_{d-1}$ and a monomorphism on $H_{2(d-1)}$. The cokernel of this map in $H_{2(d-1)}$ is isomorphic to $\Z$, which is also isomorphic to $H_{2(d-1)}(\Omega^2 S^{2d})$. Our assumption implies that the homomorphism from the cokernel of $f$ in $H_{2(d-1)}$ to $H_{2(d-1)}(\Omega^2 S^{2d})$ is an epimorphism from $\Z$ to $\Z$. Therefore it is an isomorphism. Since all the spaces in the diagram~\ref{eq: classifying square} have trivial homology in dimension above $2(d-1)$ and below $3(d-1)$, it follows that the square is $3d-4$-cocartesian. Furthermore, the maps from $\conf{\R^d}{3}$ to $({\widetilde S}^{d-1})^3$ and to $*$ are $2d-3$ and $d-1$-connected respectively. By the Blakers-Massey theorem, the square~\eqref{eq: classifying square} is $3d-5$-cartesian.
\end{proof}
Now we are ready to construct a classifying map. We will use the Thom-Pontryagin collapse map associated with the diagonal inclusion $\widetilde S^{d-1}\hookrightarrow (\widetilde S^{d-1})^3$. To get a clean description of the $\Sigma_3$-equivariant properties of this collapse map, let us first consider a more general setting, where $M$ is a manifold with a free action of $\Sigma_2$. The action of $\Sigma_2$ can be extended to an action of $\Sigma_3$ via the surjective homomorphism $\Sigma_3\twoheadrightarrow \Sigma_2$. In this way, we consider $M$ as a space with an action of $\Sigma_3$. 

The group $\Sigma_3$ acts on $M^3$ via either one of the identifications
\[
M^3\cong \map_{\Sigma_2}(\Sigma_3, M)\cong \map(\Sigma_3/\Sigma_2, M).
\]
The diagonal inclusion $\Delta\colon M\hookrightarrow M^3$ is a $\Sigma_3$-equivariant map (note again that the action of $\Sigma_3$ on $M$ is not trivial). The normal bundle of this inclusion has an induced action of $\Sigma_3$. The normal bundle is $\Sigma_3$-equivariantly isomorphic to the quotient bundle $3\tau/\Delta(\tau)$. Here $\tau$ is the tangent bundle of $M$, $3\tau=\tau\oplus\tau\oplus\tau$, and $\Delta(\tau)$ is the diagonal copy of $\tau$ in $3\tau$. We denote the normal bundle by $\widehat{2\tau}$. It is the tensor product of $\tau$ with $\widehat \R^2$. Let $M^{\widehat 2\tau}$ denote the Thom space of the normal bundle.
The Thom-Pontryagin collapse map associated with $\Delta$ is a $\Sigma_3$-equivariant map $M^3\to M^{\widehat 2\tau}$.

Now apply this to the case $M=\widetilde S^{d-1}$, the $(d-1)$-dimensional sphere, endowed with the antipodal action of $\Sigma_2$. The Thom-Pontryagin collapse map has the form
\[
(\widetilde S^{d-1})^3\to (\widetilde S^{d-1})^{\widehat 2\tau}
\]
Note that this is an unpointed map, as the space $(\widetilde S^{d-1})^3$ does not have an equivariant basepoint. Sometimes we like to think of the collapse map as a pointed map 
\[
(\widetilde S^{d-1})^3_+\to (\widetilde S^{d-1})^{\widehat 2\tau}
\]
Let us take smash product of this map with $\widehat S^2$, to obtain the following $\Sigma_3$-equivariant map
\[
(\widetilde S^{d-1})^3_+\wedge \widehat S^2\to (\widetilde S^{d-1})^{\widehat 2\tau}\wedge\widehat S^2.
\]
Now observe that there is a homeomorphism 
\[
(\widetilde S^{d-1})^{\widehat 2\tau}\wedge\widehat S^2\cong (\widetilde S^{d-1})^{\widehat2(\tau\oplus \R)}.
\]
Next, recall that the tangent bundle $\tau$ of $S^{d-1}$ satisfies the isomorphism $\tau \oplus \R\cong \R^d$. Under this isomorphism, the natural action of $\Sigma_3$ on $\tau\oplus \R$ corresponds to the sign action on $\R^d$. It follows that there is a $\Sigma_3$-equivariant homeomorphism
\[
(\widetilde S^{d-1})^{\widehat2(\tau\oplus \R)}\cong (\widetilde S^{d-1})^{\widehat2(\R^d)}\cong \widetilde S^{d-1}_+ \wedge \widehat S^{2d}.
\]
The action of $\Sigma_3$ on $\widehat S^{2d}$ is induced by the tensor product of the standard action on $\widehat \R^2$ and the sign action on $\R^d$. But this is equivalent to just the standard action of $\widehat\R^2$, without the sign twist.

Next we compose this homeomorphism with the collapse map $\widetilde S^{d-1}_+ \wedge \widehat S^{2d}\to \widehat S^{2d}$, and pre-compose with the (suspended) Thom-Pontryagin collapse map above. We obtain the map
\[
(\widetilde S^{d-1})^3_+\wedge \widehat S^2\to \widehat S^{2d}.
\]
Taking an adjoint, we obtain an unpointed $\Sigma_3$-equivariant map
\begin{equation}\label{eq: first classifying}
(\widetilde S^{d-1})^3\to \widehat \Omega^2\widehat S^{2d}.
\end{equation}
This is our model for a classifying map.
\begin{lemma}\label{lemma: Thom Pontryagin obstruction}
The map~\eqref{eq: first classifying} is a classifying map.
\end{lemma}
\begin{proof}
We need to check that the map satisfies the hypotheses of Lemma~\ref{lemma: cocartesian}. The first hypothesis is that the composite map 
\[
\conf{\R^d}{3}\to ({\widetilde S}^{d-1})^3\to  \widehat\Omega^2 \widehat S^{2d}
\]
is equivariantly null homotopic. By construction, the second map factors through the Thom-Pontryagin collapse map associated with the inclusion of the thin diagonal of $(\widetilde S^{d-1})^3$. Clearly the space $\conf{\R^d}{3}$, which is the complement of the fat diagonal of $(\widetilde S^{d-1})^3$, is contained in the complement of the thin diagonal, and therefore the restriction of the Thom-Pontryagin collapse to $\conf{\R^d}{3}$ is (equivariantly) null homotopic.

The second hypothesis that we need to check is that the following homomorphism is an epimorphism
\[
H_{2d-2}(({\widetilde S}^{d-1})^3)\to  H_{2d-2}(\widehat\Omega^2 \widehat S^{2d})
\]
This is equivalent to showing that the adjoint map
\[
S^2\wedge (S^{d-1}\times S^{d-1}\times S^{d-1})_+ \to S^{2d}
\]
Induces an epimorphism on $H_{2d}$ (till the end of this proof we will omit the `tilde' and `hat' decorations, since we are not concerned with the action of $\Sigma_3$ at this point).
Once again we recall that this map factors through the Thom-Pontryagin collapse as follows 
\[
S^2\wedge (S^{d-1}\times S^{d-1}\times S^{d-1})_+\to S^2\wedge (S^{d-1})^{2\tau}\xrightarrow{\cong}  S^{d-1}_+ \wedge S^{2d}\to  S^{2d}.
\]
We need to prove that this composite map induces an epimorphism on $H_{2d}$. To see this, choose a point $*\in S^{d-1}$ and consider the inclusion $S^{d-1}\times S^{d-1}\times \{*\} \hookrightarrow (S^{d-1})^3$. This inclusion intersects the thin diagonal transversely at a single point $(*,*,*)\in (S^{d-1})^3$. It follows quite easily that the composite map 
\[
S^2\wedge S^{d-1}\times S^{d-1}\times \{*\}_+ \to S^2\wedge (S^{d-1}\times S^{d-1}\times S^{d-1})_+ \to S^{2d}
\]
is the double suspension of the Thom-Pontryagin collapse map associated with the inclusion of a point $(*,*)\hookrightarrow S^{d-1}\times S^{d-1}$. In other words, it is the double suspension of the map $S^{d-1}\times S^{d-1}\to S^{2d-2}$ that collapses the complement of a Euclidean neighborhood of $(*,*)$. Clearly this map is surjective on $H_{2d}$, and therefore the map $S^2\wedge (S^{d-1}\times S^{d-1}\times S^{d-1})_+ \to S^{2d}$ is also surjective on $H_{2d}$.
\end{proof}

\label{geometric interpretation}
Lemma~\ref{lemma: Thom Pontryagin obstruction} leads to a geometric interpretation of the obstruction class $\obs{K}$ as a triple collinearity condition. In the manifold case, such an interpretation was hinted at by Munson~\cite{Munson}. We will describe this geometric interpretation in the case of embedding a 2-dimensional complex in $\R^4$.

To begin with, Lemma~\ref{lemma: Thom Pontryagin obstruction} tells that the obstruction class $\obs{K}$ is the pullback of a Thom class. Suppose, as usual, that $K$ is a 2-dimensional simplicial complex and we have a $\Sigma_2$-equivariant map \[f_2\colon K\times K\setminus K \to \widetilde S^3.\] 
Recall that we associate with $f_2$ a $\Sigma_3$-equivariant map
\[
\begin{array}{ccccccccc}
f_2^3\circ p_K& \colon &  \conf{K}{3}&\to& \widetilde S^3&\times &\widetilde S^3&\times &\widetilde S^3 \\[5pt]
& &(k_1, k_2, k_3) &  \mapsto & (f_2(k_1, k_2)&,& f_2(k_2, k_3)&,& f_2(k_3, k_1)).
\end{array}
\]
Recall that the cohomological obstruction $\obs{K}$ is determined by the map  $f_2^3\circ p_K$. The following lemma is really a corollary of Lemma~\ref{lemma: Thom Pontryagin obstruction}. Let $\widetilde S^3_\Delta\subset \widetilde S^3 \times \widetilde S^3\times \widetilde S^3$ be the diagonal copy of $\widetilde S^3$.
\begin{lemma}\label{lemma: Thom pullback}
The obstruction class $\obs{K}$ is the pullback of the Thom class of the normal bundle of $\widetilde S^3_\Delta$ in $\widetilde S^3 \times \widetilde S^3\times \widetilde S^3$.
\end{lemma}
\begin{proof}
It follows from Lemma~\ref{lemma: Thom Pontryagin obstruction}, that $\obs{K}$ is represented by the following composition of maps from $\conf{K}{3}$ to an Eilenberg - Mac Lane space
\[
\conf{K}{3}\xrightarrow{p_K} (\conf{K}{2})^3\xrightarrow{f_2^{\times 3}} (\widetilde S^3)^3 \to (\widetilde S^3)^{2\tau} \to \widehat \Omega^2\widehat S^{8}\to \widehat\Omega^2\Omega^\infty H\mathbb Z \wedge \widehat S^{8}\cong K(\mathbb Z[-1]; 6).
\]
Unraveling the definitions, one finds that the composition of the maps $(\widetilde S^3)^3\to K(\mathbb Z[-1]; 6)$ represents the Thom class of the normal bundle of the diagonal in $(\widetilde S^3)^3$. It follows that the composition of the maps $\conf{K}{3}\to K(\mathbb Z[-1]; 6)$ represents the pullback of the Thom class along $f_2^3\circ p_K$.
\end{proof}

Now we can use Lemma~\ref{lemma: Thom pullback} to interpret $\obs{K}$ as an intersection class. The lemma says that that $\obs{K}$ is the pullback of the Thom class along $f_2^3\circ p_K$.
For the purpose of geometric interpretation, let us restrict the domain of this map to $\confs{K}{3}$, where $\confs{K}{3}$ denotes, as usual, the union of triple products of simplices $\sigma_1\times\sigma_2\times \sigma_3$ of $K$ that are pairwise disjoint. $\confs{K}{3}$ is a subspace of $\conf{K}{3}$.
After subdividing if necessary, we may assume that the inclusion $\confs{K}{3}\hookrightarrow \conf{K}{3}$ is a homotopy equivalence. 

For the purpose of this discussion we restrict the domain of $f_2^3\circ p_K$ to be $\confs{K}{3}$ rather than $\conf{K}{3}$.

Without loss of generality, $(f_2^3\circ p_K)(\conf{K}{3})$ may be assumed to intersect $\widetilde S^3_\Delta$ transversely.
Then the set 
\[(f_2^3\circ p_K)^{-1}(\widetilde S^3_\Delta)=\{(k_1, k_2, k_3)\in\confs{K}{3}\mid f_2(k_1, k_2)=f_2(k_2, k_3)=f_2(k_3, k_1)\}\]
is a finite collection of points,  contained in the union of interiors of open $6$-dimensional cells of $\confs{K}{3}$. 


Consider the following cellular cochain $\mathcal{I}_3\colon C_6(\confs{K}{3})\to \mathbb Z$. For every cell of the form $\s_i\times \s_j\times \s_k$, where $\s_i, \s_j, \s_k$ are pairwise-disjoint $2$-dimensional simplices of $K$, $\mathcal{I}_3(\s_i\times \s_j\times \s_k)$ is defined to be the algebraic intersection number $$\left( f_2^3\circ p_K(\s_i\times \s_j\times \s_k)\right) \, \cdot \, S^3_\Delta.$$ 
The algebraic intersection number depends on a choice of orientation of each $2$-dimensional simplex of $K$ and also on a choice of orientation of $\widetilde S^3$.

Recall that $\Sigma_3$ acts on $\widetilde S^3$ by the pullback of the antipodal action of $\Sigma_2$ along the surjective homomorphism $\Sigma_3\twoheadrightarrow\Sigma_2$. It follows that $\Sigma_3$ acts trivially on the set of orientations of $\widetilde S^3$. On the other hand, the action of $\Sigma_3$ on the set of orientations of $\widetilde S^3\times \widetilde S^3 \times \widetilde S^3$ is non-trivial; it is the pullback of the free action of $\Sigma_2$. This implies that the cochain $\mathcal I_3$ really is an equivariant cochain with coefficients in the sign representation. In symbols, $\mathcal I_3\in C^6(\confs{K}{3}; \mathbb Z[-1])$. Since $\confs{K}{3}$ is a $6$-dimensional cell complex, $\mathcal I_3$ is automatically a cocycle, so it represents an element in equivariant cohomology $[\mathcal I_3]\in H^6_{\Sigma_3}(\confs{K}{3}; \mathbb Z[-1])$.
Now we are ready for the main result of this subsection.
\begin{proposition}
The cohomology class $[\mathcal I_3]$ coincides with the obstruction class $\obs{K}$.
\end{proposition}
\begin{proof}
It follows from Lemma~\ref{lemma: Thom pullback} that $\obs{K}$, or rather the restriction of $\obs{K}$ to $\confs{K}{3}$, is the pullback of the Thom class of the normal bundle of the diagonal in $\widetilde S^3\times \widetilde S^3\times \widetilde S^3$. Under a transversality assumption, the pullback of the Thom class is the intersection number with the diagonal, which is the definition of $\mathcal I_3$.
\end{proof}
\begin{remark}
Here is a heuristic explanation why the intersection number $\mathcal I_3$ is an obstruction to the existence of an embedding. Suppose that $f_2$ is a normalized deleted square of some embedding $f\colon K\hookrightarrow \R^4$. I.e., suppose that
\[
f_2(k_1, k_2)=\frac{f(k_2)-f(k_1)}{|f(k_2)-f(k_1)|}
\]
Then for all $k_1, k_2, k_3$, the three vectors $f(k_2)-f(k_1), f(k_3)-f(k_2), f(k_1)-f(k_3)$ sum up to zero. On the other hand, $(f_2^3\circ p_K)^{-1}(\widetilde S^3_\Delta)$ represents the set of triples $k_1, k_2, k_3$ where these three vectors are {\it co-directed}. It is natural that this set represents an obstruction to the existence of $f$.

(It would be interesting to compare this with the interpretation of the second coefficient of the Conway polynomial of a knot in terms of collinear triples in \cite{BCSS}.)
\end{remark}

\section{Examples where the obstruction does not vanish} \label{sec:Examples}

An explicit $2$-complex $K$ which does not embed into ${\mathbb R}^4$, but has a vanishing van Kampen obstruction was constructed in \cite{FKT}. The proof of non-embeddability in \cite{FKT} is group-theoretic in nature (using the Stallings theorem) and is quite different from the methods of this paper. In this section we reprove the non-embeddability of $K$ by showing that our obstruction is realized in this example. 

Let us begin by reviewing the construction of the complex $K$ in~\cite{FKT}. Let $\Delta^6$ (respectively ${\Delta^6}'$) be the six-dimensional simplex with vertex set $v_1,\ldots ,v_7$ (respectively $v_1',\ldots ,v_7'$).
Denote the triangle on vertices $v_a, v_b, v_c$ by $\Delta_{abc}$ and similarly the triangle on vertices $v_a', v_b', v_c'$ by $\Delta_{abc}'$.

Let $\mathrm{sk}^n\Delta^6$ denote the $n$-skeleton of $\Delta^6$.  Let $G_7$ (respectively $G'_7$) be the 2-skeleton of $\Delta^6$ minus the 2-cell associated with the triangle $\Delta_{123}$ (respectively the analogous subcomplex of ${\Delta^6}'$). 

Let $K_0=G_7\vee G'_7$ be the wedge sum obtained by identifying $v_1$ and $v_1'$ (in~\cite{FKT} the authors add an edge $v_1v_1'$, but this difference does not matter). Finally, let $K$ be the complex obtained by attaching to $K_0$ a 2-cell along the commutator of the loops $v_1v_2v_3v_1$ and $v_1'v_2'v_3'v_1'$. The closure of this 2-cell is a torus embedded in $K$. We denote this torus simply by $\Delta_{123}\times \Delta_{123}'$.

\begin{remark} 
\label{rem: examples}
This example admits an immediate generalization to a family of examples, where instead of two copies of $G_7$ and a basic commutator of two loops as above, one takes $n$ copies of the $2$-complex $G_7$ and an element of the mod $2$ commutator subgroup of the free group $F_n$ on $n$ generators. The analysis below also goes through for such commutators which are not in the next (second, in the convention of \cite[Lemma 7]{FKT}) term of the mod $2$ lower central series of $F_n$; for simplicity of notation we concentrate on the basic example described above. We expect that the examples corresponding to higher commutators are detected by our  higher obstructions $\obsn{K}, \gobsn{K}$; see Section \ref{sec: Questions}.
\end{remark}

As explained in~\cite{FKT}, van Kampen showed that $\mathrm{sk}^2\Delta^6$ can not be embedded in $\mathbb R^4$, but $G_7$ can. It follows that the complex $K_0$ can be embedded in $\mathbb R^4$. 

Let $S\subset G_7$ be the sphere that is the union of the four 2-cells that are disjoint from the triangle $\Delta_{123}$, namely the cells corresponding to $\Delta_{456}, \Delta_{457}, \Delta_{467}$ and $\Delta_{567}$. $S$ is the dual tetrahedron to the triangle $\Delta_{123}$ in the $6$-simplex. Dually, let $S'\subset G_7'$ be the dual sphere to the triangle $\Delta_{123}'$.

The following key result about embeddings of $K_0$ into $\mathbb R^4$ is proved in~\cite{FKT} (we do not reprove it).
\begin{proposition}[\cite{FKT}, Lemma 6]\label{prop: link}
For any PL embedding of $K_0$ into $\mathbb R^4$, the linking numbers of $S, S'$ and $\Delta_{123}, \Delta_{123}'$ satisfy the following
(see figure~\ref{fig:Example} for a schematic illustration):
\[
\mathrm{link}(S, \Delta_{123})\equiv \mathrm{link}(S', \Delta_{123}')\equiv 1 (\mathrm{mod} 2).
\]
\[
\mathrm{link}(S, \Delta_{123}')= \mathrm{link}(S', \Delta_{123})=0.
\]
\end{proposition}

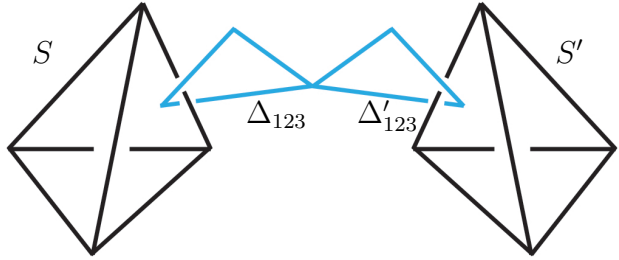
\begin{figure}[ht]
\centering
\includegraphics[height=3.5cm]{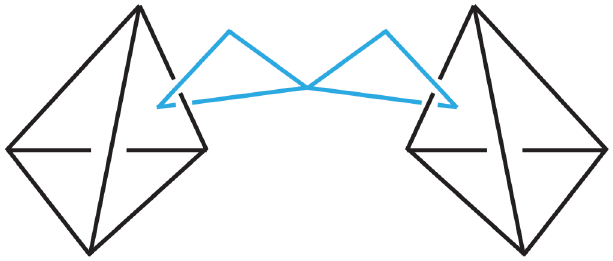}
{    
       \put(-223,75){$S$}
        \put(-142,52){$\Delta_{123}$}
        \put(-25,75){$S'$}
        \put(-100,52){$\Delta_{123}'$}
    }  
    \caption{The $2$-complex $K$ is obtained by attaching a $2$-cell along the commutator of $\Delta_{123}$ and $\Delta_{123}'$.}
\label{fig:Example}
\end{figure}
It is also shown in~\cite{FKT} that the van Kampen obstruction vanishes on $K$.
Now we can state the main result of this section. Of course it is also proved in~\cite{FKT}, using fundamental group instead of cohomology.
Another, more recent, viewpoint on this result using triple intersections may be found in \cite[Lemma 2.4]{AMSW}. We present a novel approach, in terms of the cohomology of configuration spaces and specifically the Arnold class.
\begin{proposition}\label{prop: example}
Suppose $f_2\colon \conf{K}{2}\to \conf{\mathbb R^4}{2}$ is a $\Sigma_2$-equivariant map, such that the restriction of $f_2$ to $\conf{K_0}{2}$ is equivalent to the deleted square of some embedding $f\colon K_0\hookrightarrow \mathbb R^4$. Then the following composition map 
\begin{equation}\label{eq: composition}
\conf{K}{3}\xrightarrow{p_K} \conf{K}{2}\times \conf{K}{2}\times \conf{K}{2}\xrightarrow{f_2^3} \conf{\mathbb R^4}{2}\times \conf{\mathbb R^4}{2}\times \conf{\mathbb R^4}{2}
\end{equation}
does not lift to a $\Sigma_3$-equivariant map
\[
\conf{K}{3}\to\conf{\mathbb R^4}{3}.
\]
\end{proposition}
It follows in particular that no embedding $K_0\hookrightarrow \mathbb R^4$ can be extended to an embedding $K\hookrightarrow \mathbb R^4$.  Also recall from Remark \ref{rem: Wh disk in example} that an embedding of $K_0$ extends to an order $1$ Whitney tower. It follows from Lemma \ref{Whitney-config} and Propositions \ref{prop: link} and \ref{prop: example} that no embedding of $K_0$ extends to an order $2$ tower. A more general statement is likely to be true, that there does not exist any map $K\longrightarrow \R^4$ (not necessarily restricting to an embedding of $K_0$) admitting an order $2$ Whitney tower. Its proof requires an extension of Proposition \ref{prop: link} from embeddings to maps of Whitney towers which is outside the scope of this paper.

To prove the proposition, we give a cohomological interpretation of our obstruction $\obs{K}$ in terms of the Arnold class, which may be of independent interest. Consider, once again, the problem of constructing a $\Sigma_3$-equivariant lift in a diagram of the following form
\[ \begin{tikzcd}
  & \conf{\R^4}{3} \arrow[d, "p_{\R^4}"]  \\
\conf{K}{3} \arrow[r, "f_2^3\circ p_K "]\arrow[ru, dashed, shift left=2] & \conf{\R^4}{2}\times \conf{\R^4}{2}\times \conf{\R^4}{2}
\end{tikzcd}
\]
Recall the definition of the Arnold class 
\[u\otimes u\otimes 1- u\otimes 1\otimes u+1\otimes u\otimes u\in H^6(\conf{\mathbb R^4}{2}\times \conf{\mathbb R^4}{2}\times \conf{\mathbb R^4}{2}).\]
By Lemma~\ref{lemma: arnold relation}, this class generates the kernel of $p_{\R^4}$ in $H^6$. We get the following easy sufficient condition for our obstruction to be non-zero.
\begin{lemma}\label{lemma: easy cohomological}
Referring to the diagram above, suppose $h^*(u\otimes u\otimes 1- u\otimes 1\otimes u+1\otimes u\otimes u)\ne 0$. Then a lift does not exists and~$\obs{K}\ne 0$.
\end{lemma}

One can make the connection between $\obs{K}$ and the Arnold class a little more precise. By definition, $\obs{K}$ is an element in the $\Sigma_3$-equivariant cohomology group $H_{\Sigma_3}^6(\conf{K}{3}; \mathbb{Z}[-1])$. There is a natural homomorphism
\[
H_{\Sigma_3}^6(\conf{K}{3}; \mathbb{Z}[-1])\to H^6(\conf{K}{3}; \mathbb{Z}[-1])^{\Sigma_3}\subset H^6(\conf{K}{3})
\]
\begin{lemma} \label{lem: o3 Arnold}
The image of $\obs{K}$ in $H^6(\conf{K}{3})$ under this homomorphism is (the image of) the Arnold class under the map $f_2^3\circ p_K\colon \conf{K}{3}\to \conf{\R^4}{2}^3$.
\end{lemma}
\begin{proof}
We saw in Section~\ref{sec: Constructing the obstruction} that $\obs{K}$ is represented by a map
\[
\conf{\mathbb R^n}{2}^3 \to K(\mathbb Z, 2n-2)
\]
with the property that the sequence
\[
\conf{\mathbb R^n}{3}\to \conf{\mathbb R^n}{2}^3 \to K(\mathbb Z, 2n-2)
\]
induces a split short exact sequence in $H_{2n-2}$ and in $H^{2n-2}$. It follows that the map $\conf{\mathbb R^n}{2}^3 \to K(\mathbb Z, 2n-2)$ representing our obstruction sends a generator of $H^{2n-2}(K(\mathbb Z, 2n-2))$ to a generator of $\ker(H^{2n-2}(\conf{\mathbb R^n}{2}^3)\to H^{2n-2}(\conf{\mathbb R^n}{3})$, which is precisely the Arnold class (up to sign, which we can adjust).
\end{proof}
Now let us prove the main result of this section.
\begin{proof}[Proof of Proposition~\ref{prop: example}]
The map $f_2^3\circ p_K$ of~\eqref{eq: composition} induces a homomorphism in cohomology
\[
(f_2^3\circ p_K)^*\colon H^6(\conf{\mathbb R^4}{2}\times \conf{\mathbb R^4}{2}\times \conf{\mathbb R^4}{2})\to H^6(\conf{K}{3}). 
\]
By Lemma~\ref{lemma: easy cohomological}, it is enough to show that this homomorphism does not send the element $u\otimes u\otimes 1- u\otimes 1\otimes u+1\otimes u\otimes u$ to zero.

Inside $K$ there are three disjoint subspaces: the spheres $S$ and $S'$, and the torus $\Delta_{123}\times \Delta_{123}'$. Since these subspaces are disjoint, the obvious inclusion $$S\times S'\times (\Delta_{123}\times \Delta_{123}')\hookrightarrow K\times K\times K$$
factors through an inclusion
$$i\colon S\times S'\times (\Delta_{123}\times \Delta_{123}')\hookrightarrow \conf{K}{3}.$$
We will want to give names to elements in the cohomology of $S\times S'\times (\Delta_{123}\times \Delta_{123}')$. For this purpose, let $v, v', \tau_{123},$ and $\tau_{123}'$ be generators of $H^2(S), H^2(S'), H^{1}(\Delta_{123}), H^1(\Delta_{123}')$ respectively.

Consider the composition
\begin{multline}\label{eq: detect}
S\times S'\times (\Delta_{123}\times \Delta_{123}')\xrightarrow{i} \conf{K}{3}\xrightarrow{p_K}\\ \to \conf{K}{2}\times \conf{K}{2}\times \conf{K}{2}\xrightarrow{f_2^3}  \conf{\mathbb R^4}{2}\times \conf{\mathbb R^4}{2}\times \conf{\mathbb R^4}{2}.
\end{multline}
We want to analyze the effect of this map on cohomology. Recall that for any space $X$, the map $p_X\colon \conf{X}{3}\to \conf{X}{2}^3$ is defined by the formula
\[
p_X(x_1, x_2, x_2)=((x_1, x_2), (x_2, x_3), (x_3, x_1)).
\]
For $i=1, 2, 3$, let $pr_X^i\colon \conf{X}{2}^3\to \conf{X}{2}$ be the projection onto the $i$-th factor. Notice that $pr_X^i\circ p_X(x_1, x_2, x_3)=(x_i, x_{i+1})$, where $i+1$ is computed cyclically. We have the following commutative diagram:
\[ \begin{tikzcd}
S\times S' \times (\Delta_{123}\times \Delta'_{123})\arrow[r, "p_K\circ i"]\arrow[d]  & \conf{K}{2}^3 \arrow[r, "f_2^3"]\arrow[d, "pr_K^1"] & \conf{R^4}{2}^3\arrow[d, "pr_{\R^4}^1"]  \\
S\times S' \arrow[r] & \conf{K}{2} \arrow[r, "f_2"] & \conf{\R^4}{2}
\end{tikzcd}
\]
The composition of the two maps at the bottom is a map $S\times S'\to \conf{\mathbb R^4}{2}$. This map is zero on reduced cohomology for the obvious reason that the target only has non-trivial cohomology in degree $3$ and the source has trivial cohomology in degree $3$. Notice that the right vertical map $pr_{\R^4}^1$ takes the cohomology generator $u$ to $u\otimes 1\otimes 1$. 
It follows that the composition of the two horizontal map at the top of the diagram takes the classes $u\otimes u \otimes 1$ and $u\otimes 1 \otimes u$ to zero.

It remains to see what happens to the third summand of the Arnold class, namely $1\otimes u\otimes u$. For this purpose consider the following diagram
\[ \begin{tikzcd}
S\times S' \times (\Delta_{123}\times \Delta'_{123})\arrow[r, "p_K\circ i"]\arrow[d]  & \conf{K}{2}^3 \arrow[r, "f_2^3"]\arrow[d, "pr_K^3"] & \conf{R^4}{2}^3\arrow[d, "pr_{\R^4}^3"]  \\
S\times (\Delta_{123}\times \Delta'_{123}) \arrow[r] & \conf{K}{2} \arrow[r, "f_2"] & \conf{\R^4}{2}
\end{tikzcd}
\]

It follows from Proposition~\ref{prop: link} that the composition of maps at the bottom of this diagram sends the generator $u$ of $H^{3}(\conf{\mathbb R^4}{2}$ to an odd multiple of $v\otimes \tau_{123}\otimes 1$. It follows that the composition of the maps at the top sends $1\otimes 1\otimes u$ to an odd multiple of $v\otimes 1\otimes \tau \otimes 1$. A similar diagram shows that the composition of top maps sends $1\otimes u \otimes 1$ to an odd multiple of $1\otimes v'\otimes 1 \otimes \tau'$. 

Putting all these calculations together, it follows that the map $f_2^3\circ p_K \circ i$ sends the Arnold class to an odd multiple of $v\otimes v'\otimes \tau \otimes \tau'$, which is not zero. It follows that $f_2^3\circ p_K$ does not send the Arnold class to zero. This is what we wanted to prove.
\end{proof}

\begin{remark} \label{rem: examples in other dim}
In the discussion above we focused on the case of $2$-complexes in $\R^4$, but a similar calculation 
shows that the obstruction $\obs{K}$ detects non-embeddability of examples (with vanishing obstruction $\vk{K}$) in all dimensions outside the metastable range, $2d<3(m+1)$, such that 
$d \geq {\rm max}(4, m)$. Such examples of $m$-dimensional complexes were constructed and shown to not admit an embedding in $\R^d$ in \cite{SS, SSS}. 
The construction involves the Whitehead product of meridional spheres $S^l, l=d-m-1$, linking two $m$-spheres $S$, $S'$, rather than the commutator of loops in the construction above. 
To be precise, this remark concerns the version of the examples from \cite{SS, SSS} which is a direct analogue of the example from \cite{FKT} and illustrated in figure \ref{fig:Example}. That is, the wedge sum is taken so that the spheres $S$, $S'$ do not have vertices in common. We discuss the examples of \cite{SS, SSS} again in
Section \ref{sec: convergence}.

There are three disjoint subspaces in the complex: the spheres $S, S'$, and a $2l$-torus, and the calculation of the Arnold class analogous to the above shows that it is non-trivial. This gives a unified proof of non-embeddability of the examples in \cite{FKT} and in \cite{SS, SSS}, while the arguments in these original references are quite different, both from each other and from the new perspective in this paper.
\end{remark}

\section{Higher obstructions} \label{the tower}
In this section we show how the obstruction $\vk{K}$ and $\obs{K}$ can be extended to a sequence of obstructions $\obsn{K}$, using a primitive version of the Goodwillie-Weiss tower. We will then give a conjectural description of a framed cobordism refinement of $\obsn{K}$.
\begin{definition}
Let $\Fin$ be the category of finite sets and injective functions between them, and $\Fin_n\subset \Fin$ be the full subcategory consisting of sets of cardinality at most $n$.  
\end{definition}

As before, let $\emb(K, \R^d)$ denote the space of topological embeddings of $K$ into $\R^d$. In the case when $i$ is a finite set and $X$ is any space, $\emb(i, X)$  is the configuration space of ordered $i$-tuples of pairwise distinct points of $X$. We also denote this space by $\conf{X}{i}:=\emb(i, X)$.

Given a small category $\mathcal C$ and functors $F, G\colon \mathcal C\to \Top$, we let $\nat_{\mathcal C}(F, G)$ denote the space of natural transformations from $F$ to $G$, and let $\hnat_{\mathcal C}(F, G)$ denote the space of derived natural transformations from $F$ to $G$. In other words, $\hnat_{\mathcal C}(F, G)$ is the space of natural transformations from a cofibrant replacement of $F$ to a fibrant replacement of $G$. The (co)fibrant replacements can be taken in any Quillen model structure on the functor category $[\mathcal C, \Top]$, where the weak equivalences are defined levelwise. We will use the projective model structure, in which every functor is fibrant. 

\begin{remark}
To save notation, if $F, G$ are functors $\mathcal C^{\op} \to \Top$, we will use the notation $\hnat_{\mathcal C}(F, G)$ rather than $\hnat_{\mathcal C^{\op}}(F, G)$.
\end{remark}

A topological space $K$ determines a functor $\conf{K}{-}\colon \Fin^{\op} \to \Top$ that sends a set $i$ to $\conf{K}{i}=\emb(i, K)$. A topological embedding $f\colon K\hookrightarrow \R^d$ gives rise to a natural transformation $\conf{K}{-}\to \conf{\R^d}{-}$, which sends an embedding $\alpha\colon i\hookrightarrow K$ to the embedding $f\circ \alpha\colon i\hookrightarrow \R^d$. This gives rise to natural maps.
\begin{equation}\label{eq: main map}
\emb(K, \R^d) \to \nat_{\Fin}(\conf{K}{-}, \conf{\R^d}{-}) \to  \hnat_{\Fin}(\conf{K}{-}, \conf{\R^d}{-}) .
\end{equation}
One useful feature of the space $\hnat_{\Fin}(\conf{K}{-}, \conf{\R^d}{-})$ is that it admits a natural tower of approximations. 
\begin{definition}\label{def: primitive}
For each $n\ge 1$ define
\begin{equation} \label{eq:tower}
T_n\emb(K, \R^n)=\hnat_{\Fin_n}(\conf{K}{-}, \conf{\R^d}{-})
\end{equation}
\end{definition}
The inclusions of categories $\cdots\Fin_{n-1}\subset \Fin_n\subset \cdots \subset \Fin$ give rise to a tower whose homotopy inverse limit is equivalent to $\hnat_{\Fin}(\conf{K}{-}, \conf{\R^d}{-})$
\[
\hnat_{\Fin}(\conf{K}{-}, \conf{\R^d}{-})\to \cdots \to T_n\emb(K, \R^n) \to T_{n-1}\emb(K, \R^d)\to \cdots
\]
\begin{remark}
Readers familiar with the embedding calculus of Goodwillie and Weiss will readily recognize $T_n\emb(K, \R^n)$ as a primitive analogue of the $n$-the Taylor approximation in the Goodwillie tower. Indeed, the Goodwillie-Weiss construction is essentially the same as the one in Definition~\ref{def: primitive}, except that instead of the category $\Fin_n$ of sets with at most $n$ elements, they use the category whose objects are manifolds diffeomorphic to the disjoint union of at most $n$ copies of $\R^m$, and whose morphisms are smooth embeddings. At least this is one way to construct the Goodwillie-Weiss tower. For more information about this approach to the Goodwillie-Weiss calculus see the paper of Boavida and Weiss~\cite{BW}.
\end{remark}
The following lemma is an immediate consequence of the existence of the map~\eqref{eq: main map}.
\begin{lemma}
If $\hnat_{\Fin_n}(\conf{K}{-}, \conf{\R^d}{-})$ is empty for some $n$, then there does not exists an embedding of $K$ into $\R^d$.
\end{lemma}
Our goal is to study obstructions for a path component of $T_{n-1}\emb(K, \R^d)$ to be in the image of $T_{n}\emb(K, \R^d)$. For this purpose it is useful to have an inductive description of $T_{n}\emb(K, \R^d)$. Such a description is given by Proposition~\ref{proposition: inductive} below. The proposition is elementary and no doubt well-known. But for completeness we will give a proof. We need some preparation.
\begin{definition}
Let 
\[
\subconf{\R^d}{n}=\underset{S\subsetneq \{1,\ldots, n\}}{\holim}\conf{\R^d}{S}
\]
\end{definition}
In words, $\subconf{\R^d}{n}$ is the homotopy limit of all the ordered configuration spaces of proper subsets of $\{1, \ldots, n\}$ into $\R^d$. 
\begin{remark}
 It is worth noting that $\subconf{\R^d}{3}\simeq \conf{\R^d}{2}^{3}$-a space that we encountered in sections~\ref{sec: Constructing the obstruction} and~\ref{sec: construction}. Everything we are doing in this section is a generalization of what we did in those two sections for $n=2, 3$.
\end{remark}
There is another, equivalent, construction of the space $\subconf{\R^d}{n}$ that will come up. Let $\Fin_{n-1}\downarrow n$ be the category whose objects are injective maps of sets $i\hookrightarrow n$, where $n$ is shorthand for $\{1, \ldots, n\}$ and $i\in \Fin_{n-1}$ denotes a set with strictly fewer elements than $n$. Morphisms in $\Fin_{n-1}\downarrow n$ are commuting triangles. There is a functor from $\Fin_{n-1}\downarrow n$ to the category (poset) of proper subsets of $\{1, \ldots, n\}$ which sends an injective map $i\hookrightarrow n$ to its image. This functor is easily seen to be faithful, full and surjective, so it is an equivalence of categories. Therefore it induces an equivalence
\begin{equation}\label{eq: extensions}
\underset{S\subsetneq \{1, \ldots, n\}}{\holim} \conf{\R^d}{S}
\xrightarrow{\simeq} \underset{i\hookrightarrow n\in \Fin_{n-1}\downarrow n}{\holim} \conf{\R^d}{i}.
\end{equation}
Another notion that we will use in the proof of Proposition~\ref{proposition: inductive} is that of a homotopy right Kan extension. Let us quickly review what this is. Suppose $\mathcal C$ is a category and $\mathcal C_0$ is a subcategory. Next, suppose $G\colon \mathcal C_0\to \Top$ is a functor defined on a subcategory of $\mathcal C$. Then let $RG\colon \mathcal C \to \Top$ denote the homotopy right Kan extension of $G$ from $\mathcal C_0$ to $\mathcal C$. Recall that $RG$ can be defined on the objects of $\mathcal C$ by the following formula
\[
RG(x)=\underset{x\to z\in x\downarrow \mathcal C_0}{\holim} G(z).
\]
The homotopy right Kan extension is a derived right adjoint to the restriction functor. This means that for any functor $F\colon \mathcal C \to \Top$ there is a natural equivalence
\begin{equation}\label{eq: adjunction}
\hnat_{\mathcal C}(F, RG)\simeq \hnat_{\mathcal C_0}(F|_{\mathcal C_0}, G)
\end{equation}
The adjunction also means that there is a natural transformation of functors $F\to RF|_{\mathcal C_0}$. If $\mathcal C_0$ is a full subcategory of $\mathcal C$ then this natural transformation is an equivalence when evaluated on objects of $\mathcal C_0$.

Now we are ready to state and prove the inductive description of $T_n\emb(K, \R^d)$.
\begin{proposition}\label{proposition: inductive}
There is a homotopy pullback square, where the right vertical map is induced by the canonical map $\conf{\R^d}{n}\to \subconf{\R^d}{n}$
\[
\begin{array}{ccc}
T_n\emb(K, \R^d) & \to &  \map(\conf{K}{n}, \conf{\R^d}{n})^{\Sigma_n}\\
\downarrow & & \downarrow \\
T_{n-1}\emb(K, \R^d) & \to & \map(\conf{K}{n}, \subconf{\R^d}{n})^{\Sigma_n}
\end{array}
\]
\end{proposition}
\begin{proof}
Since $T_n\emb(K, \R^d)=\hnat_{\Fin_n}(\conf{K}{-}, \conf{\R^d}{-})$, our task is to prove that there exists a homotopy pullback diagram of the following form
\begin{equation}\label{diagram: desired}
\begin{tikzcd}
\hnat_{\Fin_n}(\conf{K}{-}, \conf{\R^d}{-}) \ar{r}\ar{d}&  \map(\conf{K}{n}, \conf{\R^d}{n})^{\Sigma_n}\ar{d}\\
\hnat_{\Fin_{n-1}}(\conf{K}{-}, \conf{\R^d}{-}) \ar{r} & \map(\conf{K}{n}, \subconf{\R^d}{n})^{\Sigma_n}
\end{tikzcd}
\end{equation}
The strategy is to express all four corners of this square as spaces of homotopy natural transformations between functors defined on $\Fin_n$ using homotopy right Kan extension.

Let $R_{n-1}^n\conf{\R^d}{-}$ be the homotopy right Kan extension of the functor $\conf{\R^d}{-}$ from $\Fin_{n-1}$ to $\Fin_n$. By~\eqref{eq: adjunction} we know that restriction from $\Fin_n$ to $\Fin_{n-1}$ induces an equivalence
\[
\hnat_{\Fin_n}(\conf{K}{-}, R_{n-1}^n\conf{\R^d}{-})\xrightarrow{\simeq} \hnat_{\Fin_{n-1}}(\conf{K}{-}, \conf{\R^d}{-}).
\]
Now let us analyse the functor $R_{n-1}^n\conf{\R^d}{-}$. There is a natural transformation of (contravariant) functors on $\Fin_n$
\[
\conf{\R^d}{-}\to R_{n-1}^n\conf{\R^d}{-}.
\]
This natural transformation is an equivalence when evaluated on objects of $\Fin_{n-1}$ because $\Fin_{n-1}$ is a full subcategory of $\Fin_n$. On the other hand, we have the following formula for $R_{n-1}^n\conf{\R^d}{n}$ 

\[
R_{n-1}^n\conf{\R^d}{n}\simeq \underset{i\hookrightarrow n\in \Fin_{n-1}\downarrow n}{\holim} \conf{\R^d}{i}
\]
By~\eqref{eq: extensions} we have an equivalence
\[\subconf{\R^d}{n}=\underset{S\subsetneq \{1, \ldots, n\}}{\holim} \conf{\R^d}{S}
\xrightarrow{\simeq} \underset{i\hookrightarrow n\in \Fin_{n-1}\downarrow n}{\holim} \conf{\R^d}{i} .
\]
Therefore there is an equivalence 
\[
R_{n-1}^n\conf{\R^d}{n}\simeq \subconf{\R^d}{n}.
\]
And the map $\conf{\R^d}{n}\to R_{n-1}^n\conf{\R^d}{n}$ is equivalent to the natural map $\conf{\R^d}{n}\to \subconf{\R^d}{n}$.
Now consider the full subcategory of $\Fin_n$ consisting of the single object $n$ and its endomorphisms. This category is the symmetric group, and we will denote it by $\Sigma_n$. A functor from $\Sigma_n$ to $\Top$ is the same thing as a space with an action of $\Sigma_n$. Given a space $X_n$ with an action of $\Sigma_n$, we let $R_{\Sigma}^{\Fin}X_n(-)$ denote the homotopy right Kan extension of this functor from $\Sigma_n$ to $\Fin_n$. Since there are no morphisms in $\Fin_n$ from $n$ to smaller sets, it follows that $R_{\Sigma}^{\Fin} X_n(n)=X_n$ and $R_{\Sigma}^{\Fin}X_n(i)\simeq *$ for $i<n$. 

From the discussion above we conclude that there is a homotopy pullback square of functors from $\Fin_n$ to $\Top$
\begin{equation}\label{diagram: coskeletal}
\begin{tikzcd}
\conf{\R^d}{-} \arrow{r}\arrow{d} & R_{\Sigma}^{\Fin}\conf{\R^d}{n}(-) \arrow{d}\\
R_{n-1}^n\conf{\R^d}{-} \arrow{r} & R_{\Sigma}^{\Fin}\subconf{\R^d}{n}(-)
\end{tikzcd}
\end{equation}
Indeed, when evaluated at a set $i<n$, the vertical morphisms in this square are equivalences, and when evaluated at $n$, the horizontal morphisms are equivalences. So it is a homotopy pullback square of functors. 

Applying $\hnat_{\Fin_n}(\conf{K}{-}, -)$ to~\eqref{diagram: coskeletal}, we obtain a homotopy pullback square
\[
\begin{tikzcd}
\hnat_{\Fin_n}(\conf{K}{-},\conf{\R^d}{-}) \ar{r}\ar{d} & \hnat_{\Fin_n}(\conf{K}{-}, R_{\Sigma}^{\Fin}\conf{\R^d}{n}(-))\ar{d} \\
\hnat_{\Fin_n}(\conf{K}{-}, R_{n-1}^n\conf{\R^d}{-}) \ar{r} & \hnat_{\Fin_n}(\conf{K}{-}, R_{\Sigma}^{\Fin}\subconf{\R^d}{n}(-))
\end{tikzcd}
\]
Using the fact that right Kan extension is derived right adjoint to restriction we obtain that this square is equivalent to the desired square~\eqref{diagram: desired} at the beginning of the proof. So we have proved that a homotopy pullback square of this form exists.
\end{proof}
\begin{remark}
One can interpret the homotopy pullback square~\eqref{diagram: coskeletal} as an inductive description of the coskeletal filtration on a functor defined on a (generalized) Reedy category. See~\cite[Section 6]{BeMo}.
\end{remark}
Proposition~\ref{proposition: inductive} leads to an inductive procedure for constructing obstructions to the existence of an embedding $K\hookrightarrow \R^d$. Suppose we have a point $g_{n-1}\colon \hnat_{\Fin_{n-1}}(\conf{K}{-}, \conf{\R^d}{-})$, and we want to know if (the path component of) $g_{n-1}$ lies in the image of $$\hnat_{\Fin_{n}}(\conf{K}{-}, \conf{\R^d}{-}).$$ The bottom map in diagram~\eqref{diagram: desired} sends $g_{n-1}$ 
to a $\Sigma_n$-equivariant map $$\tilde f_n\colon\conf{K}{n}\to \subconf{\R^d}{n},$$ which really factors as a composite 
\[
\conf{K}{n}\to \subconf{K}{n}\to \subconf{\R^d}{n}.
\]
The path component of $g_{n-1}$ is in the image of a path component of 
$\hnat_{\Fin_{n}}(\conf{K}{-}, \conf{\R^d}{-})$ if and only if $\tilde f_n$ lifts up to homotopy to a $\Sigma_n$-equivariant map $f_n\colon \conf{K}{n}\to \conf{\R^d}{n}$, as per the following diagram
\begin{equation} \label{eq: C C0}
\begin{tikzcd}
 & \conf{\R^d}{n} \arrow{d}  \\
\conf{K}{n} \arrow{r}{\tilde f_n} \arrow[ru, dashed , "f_n", shift left =1]& \subconf{\R^d}{n}
\end{tikzcd}
\end{equation}
At this point obstruction theory kicks in. We will assume that $d\ge 3$, so that the spaces $\conf{\R^d}{n}$ and $\subconf{\R^d}{n}$ are simply connected. The first obstruction to the existence of a lift $f_n$ lies in the equivariant cohomology of $\conf{K}{n}$ with coefficients in the first non-trivial homotopy group of the homotopy fiber of the map $\conf{\R^d}{n}\to \subconf{\R^d}{n}$. The following proposition is known~\cite{GW, Kosanovic}.
\begin{proposition}
The map $\conf{\R^d}{n}\to \subconf{\R^d}{n}$ is $(d-2)(n-1)+1$-connected. Let $F$ be the homotopy fiber of this map. The first non-trivial homotopy group of $F$ is
\[
\pi_{(d-2)(n-1)+1}(F)\cong \Z^{(n-2)!}.
\]
\end{proposition}
Since the space $F$ is simply-connected, the action of $\Sigma_n$ on these spaces induces a well-defined action on the first non-trivial homotopy group of $F$. Thus the group $\Z^{(n-2)!}$ is a representation of $\Sigma_n$. Standard obstruction theory implies the following result.
\begin{theorem} \label{thm: cohom obs}
Suppose, as above, that we have a point $g_{n-1}\colon \hnat_{\Fin_{n-1}}(\conf{K}{-}, \conf{\R^d}{-})$, and we want to know if (the path component of) $g_{n-1}$ lies in the image of $$\hnat_{\Fin_{n}}(\conf{K}{-}, \conf{\R^d}{-}).$$
There is a cohomological obstruction $\obsnfull{K}{g_{n-1}}$ to the existence of such a lift. The class $\obsnfull{K}{g_{n-1}}$ is an element of the equivariant cohomology group.
\[
H^{(d-2)(n-1)+2}_{\Sigma_n}\left(\conf{K}{n}, \Z^{(n-2)!}\right).
\]
We will write simply $\obsn{K}$ rather than $\obsnfull{K}{g_{n-1}}$, when $g_{n-1}$ is irrelevant. If $\dim(K)\cdot n= (d-2)(n-1)+2$ then $\obsn{K}$ is a complete obstruction to the existence of a lift $f_n$. In particular, this holds when $\dim(K)=2$ and $d=4$.
\end{theorem}
It is easy to see that for $n=2, 3$ the general definition of $\obsn{K}$ agrees with the definitions of $\vk{K}$ and $\obs{K}$ that we saw earlier.

We end this section by describing a conjectural refinement of $\obsn{K}$ to an obstruction $\obsnfr{K}$ living in equivariant stable cobordism, extending the definitions of $\vkfr{K}$ and $\obsfr{K}$ that we saw earlier. 

The construction of $\obsnfr{K}$ uses a geometric realization of the group $\pi_{(d-2)(n-1)+1}(F)\cong \Z^{(n-2)!}$ as the cohomology of the {\it space of non 2-connected graphs}. Recall that a graph $G$ is called {\it $2$-connected} if $G$ connected, and for every vertex $x$, $G\setminus\{x\}$ is connected.
\begin{definition}
For $n>1$, let $\Delta_n^2$ be the poset of non-trivial non $2$-connected graphs with vertex set $\{1,\ldots, n\}$. Let $T_n$ be the unreduced suspension of the geometric realization of $\Delta_n^2$.
\end{definition}
The space $T_n$ was initially introduced by Vassiliev, and was studied in the paper~\cite{Babsonetal}. For example, $T_2=S^0$, $T_3=S^2$, with the standard (non-trivial) action of $\Sigma_3$. 

The following is well-known~\cite{Babsonetal}.
\begin{theorem}
There is a homotopy equivalence $$T_n\simeq \bigvee_{(n-2)!} S^{2n-4}$$
\end{theorem}

\begin{conj}\label{conj: principal}
There is a natural $\Sigma_n$-equivariant map
\[
\subconf{\R^d}{n} \to \mbox{map}_*(T_n, \Omega^\infty\Sigma^\infty S^{d(n-1)})
\]
So that there is an $(d-2)n+1$-cartesian square
\[
\begin{tikzcd}
\conf{\R^d}{n} \ar{r}\ar{d} & \subconf{\R^d}{n}\ar{d}\\
* \ar{r} & \operatorname{map}_*(T_n, \Omega^\infty\Sigma^\infty S^{d(n-1)})
\end{tikzcd}
\]
\end{conj}
Assuming the conjecture, we have natural maps
\[
T_{n-1}\emb(K, \R^d)  \to  \map(\conf{K}{n}, \subconf{\R^d}{n})^{\Sigma_n} \to \operatorname{map}_*(\conf{K}{n}_+\wedge T_n, \Omega^\infty\Sigma^\infty S^{d(n-1)})^{\Sigma_n}
\]
This map associates to a point in $T_{n-1}\emb(K, \R^d) $ an element of the equivariant stable cohomotopy group of $\conf{K}{n}\wedge T_n$. This element is an obstruction $\obsnfr{K}$ to the point of $T_{n-1}\emb(K, \R^d) $ being in the image of $T_{}\emb(K, \R^d) $. The obstruction is complete so long as $d\ge\dim(K)+2$. 
\begin{remark}
Our reasons to believe Conjecture~\ref{conj: principal} come from Orthogonal Calculus~\cite{WeissOrth}. The functor that sends $\R^d$ to the spectrum $$\operatorname{map}_*(T_n, \Omega^\infty\Sigma^\infty S^{d(n-1)})$$ is the bottom non-trivial layer of the difference between $\conf{\R^d}{n}$ and $\subconf{\R^d}{n}$. In fact, the conjecture is almost a formal consequence of the existence of orthogonal calculus and what we know about the derivatives of functors related to $\conf{\R^d}{n}$. However, it would be interesting to have an explicit map
\[
\subconf{\R^d}{n} \to \mbox{map}_*(T_n, \Omega^\infty\Sigma^\infty S^{d(n-1)})
\]
with some sort of geometric interpretation. The Thom-Pontryagin collapse map that we defined for the case $n=3$ in Section~\ref{sec: construction} does not seem to generalize easily to higher values of $n$. 
\end{remark}

\section{Questions and conjectures} \label{sec: Questions}

In conclusion we will mention several problems motivated by the results of this paper. 

\subsection{Equivalence of higher obstructions.}
Theorem \ref{thm: obstructions coincide} shows that the obstruction $\gobs{K}$ equals the pullback of $\obs{K}$ to $H^{6}_{\Sigma_3}(\confs{K}{3}; \Z[(-1)])$. We conjecture that the analogous relation holds for higher obstructions as well. More precisely, we conjecture that a map $f\colon K\longrightarrow \R^4$ together with a Whitney tower of order $n-2$ for any $n\geq 4$ determines a point in $T_{n-1}\emb(K, \R^4)$, and the obstruction $\gobsn{K, W}$ from Definition  \ref{def: Wn} equals the pullback of $\obsn{K}$ to $H^{2n}_{\Sigma_n}\left(\confs{K}{n}, \Z^{(n-2)!}\right)$.

\subsection{Conjectural higher cohomological obstructions.}  \label{sec: Massey}

Recall the discussion of the Arnold class (Definition \ref{def: Arnold}) and its relation with the obstruction $\obs{K}$ (Lemma \ref{lem: o3 Arnold}).
Here we formulate a certain version of Massey products, defined when the Arnold class vanishes. 

For convenience of choosing signs below, we will restrict ourselves to the case $d=4$;  analogous cohomological classes can be constructed for any $d$.  

Notice that for every choice of indices $i, j$ satisfying $i\ne j$ and $1\le i,j \le 4$, there is a map $r_{i,j}\colon\subconf{\R^4}{4}\to \conf{\R^4}{2},$ induced by restriction to the ordered pair $(i, j)$. As usual, let us pick a generator $u\in H^{3}(\conf{\mathbb R^4}{2}$. Let $U$ be some cocycle representative of $u$, and for each $i, j$ as above let $U_{ij}$ be the pullback of $U$ along $r_{i,j}$.
Assume that there exists a map $\tilde f_4\colon \conf{K}{4}\longrightarrow \subconf{\R^d}{4}$ as in diagram (\ref{eq: C C0}).
Denote by $V_{ij}$ the $3$-cocycles on $\conf{K}{4}$ obtained as the pull-backs of $U_{ij}$, $i,j\in \{1,\ldots,4\}$.
Consider 
\[
V_{i,j,k}:=V_{ij}V_{jk}+V_{jk}V_{ki}+V_{ki}V_{ij}.
\]
It follows from the existence of the map $\tilde f_4\colon \conf{K}{4}\longrightarrow \subconf{\R^d}{4}$ and the resulting vanishing of the Arnold class in Lemma \ref{lem: o3 Arnold} that for each subset $\{i,j,k\}\subset \{1,\ldots, 4\}$ the cohomology class of $V_{i,j,k}$ in $H^6(\conf{K}{4})$ is trivial. Consider $5$-cochains $X_{ijk}$ on $C(K,4)$, defined by ${\delta}X_{ijk}=V_{ijk}$. Consider the $8$-cochain on $\conf{K}{4}$:
\begin{equation} \label{eq: Y1}
Y_{(12)(34)}:=X_{123}(V_{14}-V_{24}) + X_{234}(V_{31}-V_{41})+X_{341}(V_{32}-V_{42})+X_{412}(V_{13}-V_{23})
\end{equation}
One checks that this is a cocycle; in fact there are two additional cocycles which we denote $Y_{(13)(24)}, Y_{(14)(23)}$; for example
\begin{equation} \label{eq: Y2}
Y_{(13)(24)}:=X_{123}(V_{34}-V_{14}) + X_{234}(V_{41}-V_{21})+X_{341}(V_{12}-V_{32})+X_{412}(V_{23}-V_{43})
\end{equation}
The sum of these three cocycles is zero. We conjecture that these cohomology classes are obstructions to lifting $\tilde f_4$ in diagram (\ref{eq: C C0}) to a map $f_4 \colon \conf{K}{4}\longrightarrow \conf{\R^d}{4}$, and that they are related to the obstructions $\operatorname{{\mathcal O}_4\!}{(K)}, \operatorname{{\mathcal W}_4\!}{(K)}\in H^{8}_{\Sigma_4}\left(\conf{K}{4}, \Z^{2}\right)$ analogously to Lemma \ref{lem: o3 Arnold}. Moreover, formulas (\ref{eq: Y1}), (\ref{eq: Y2}) suggest that these classes admit a systematic generalization to $C(K,n)$ for larger $n$ as well.

\subsection{Relation to other obstructions for a class of $2$-complexes in ${\mathbb R}^4$} 
For $2$-complexes $K$ with $H_1(K; {\mathbb Q})=0$, a sequence of obstructions to embeddability in ${\mathbb R}^4$ was defined in \cite{Kr}. The context in that reference, in terms of Massey products on $3$-manifold boundary of $4$-dimensional thickenings of $K$, is quite different from the setting of our work. Determining how the obstructions in \cite{Kr}, for $2$-complexes with vanishing first homology with rational coefficients, fit in the framework developed in this paper is an interesting question.

\subsection{Almost embeddings and complexes with vanishing obstructions} \label{sec: convergence} 
A PL {\em almost-embedding} of a complex $K$ is a PL map $K\longrightarrow \R^d$ such that non-adjacent
simplices of $K$ do not intersect in the image \cite[Section 4]{FKT}.\footnote{We would like to thank Arkadiy Skopenkov for pointing out the relation to almost-embeddings and the relevance of the examples of \cite{SS, SSS}, discussed in this section.}
Note that an almost embedding gives rise to a $\Sigma_n$-equivariant map of simplicial configuration spaces $\confs{K}{n}\longrightarrow \confs{\R^d}{n}$; in fact both the obstructions $\gobsn{K}$ (for $2$-complexes in $R^4$) and the ``simplicial'' version of $\obsn{K}$, defined using $\confs{K}{n}$, are obstructions to the existence of a PL almost embedding. The same conclusion for the finer obstructions $\obsn{K}$ (defined using $\conf{K}{n}$ rather than $\confs{K}{n}$) is also possible but not immediate.

For any pair of dimensions $m, d$ outside the stable range, that is $2d< 3(m+1)$, with $d\geq {\rm max}(4,n)$, the authors of
\cite{SS, SSS} constructed examples of $m$-complexes which PL almost embed but do not PL embed into $\R^d$; see also related discussion in \cite[Section 5]{Skopenkov}. These examples are similar  to the complex illustrated in figure \ref{fig:Example}, except that the wedge sum identifies vertices of $S, S'$; PL embeddability is established using the finger move familiar in $4$-manifold topology, and its higher dimensional analogues. 

These observations suggest the question of whether the vanishing of the obstructions $\obsn{K}$ for all $n$ is sufficient for almost embeddability of a $2$-complex $K$ in $\R^4$. In another direction, recall that the embedding obstructions $\obsn{K}$ formulated in this paper use a weaker version of the Goodwillie-Weiss
tower. One may wonder whether a different version of the 
tower may be formulated, converging in dimensions $2d< 3(m+1)$.

As mentioned in the introduction, in the special {\em relative} case where $K$ is the disjoint union of disks $ D^2_i$ and the embedding problem in the $4$-ball has a prescribed boundary condition -- a link $L$ formed by the boundaries of the disks $ \partial D^2_i$ in $S^3=\partial D^4$ -- our obstructions correspond to the Milnor invariants (with non-repeating coefficients) of $L$. There are well-known examples (boundary links) which have trivial Milnor's invariants but are not slice. (Further, there are examples \cite{CO} of links with vanishing Milnor invaraints which are not concordant to boundary links.)
However
in our context there is no boundary condition present, and there is considerable flexibility in
re-embedding, thus the obstructions from link theory do not admit an immediate analogue for embedding of complexes.

\subsection{Intrinsic characterization of the obstructions.} Given a $2$-complex $K$ with trivial $\vk{K}$, is there an intrinsic characterization of classes in $H^{3}(\confs{K}{2})$ that arise (as the pullback of a generator of $H^{3}(\conf{\R^4}{2})$) from maps to $\R^4$ as in Lemma \ref{Whitney-config}? The proof of non-embeddability of examples in Section \ref{sec:Examples} relies on Proposition \ref{prop: link}. A characterization of such classes $H^{3}(\confs{K}{2})$ might lead to an obstruction theory (the Arnold class, and higher cohomological operations in Section \ref{sec: Massey}) defined without a reference to maps into configurations spaces of $\R^4$.

\subsection{Complexity of embeddings.}
There have been recent advances in the subject of complexity of embeddings of complexes into Euclidean spaces, both from algorithmic and geometric perspectives, cf. \cite{MTW, FK}. In higher dimensions there is an upper bound $O({\rm exp}(N^{4+\epsilon}))$ on the {\em refinement complexity} (r.c.), i.e. the number of subdivisions needed to PL embed a simplicial $m$-complex (with trivial $\vk{K}$) into $\R^{2m}$, $m>2$, in terms of the number $N$ of simplices of $K$. For $2$-complexes in $\R^4$ the complexity problem is open. The examples in \cite{FK} (relying on the van Kampen obstruction) have exponential r.c., and the embedding problem in this dimension is NP-hard \cite{MTW}. But to the authors' knowledge it is an open question whether r.c. could even be non-recursive (and correspondingly whether the embedding problem is algorithmically undecidable). It is a natural question whether the higher obstruction theory developed in this paper may be used to shed new light on the problem.

\end{document}